\documentclass[11pt, reqno, twoside, letterpaper]{amsart}

\usepackage[
 letterpaper, twoside,
 inner=1.35in, outer=1.40in,
 top=1.25in,  bottom=1.25in,
 headsep=16pt, footskip=30pt,
]{geometry}

\usepackage{amsmath}
\usepackage{mathtools}

\usepackage{libertinus}
\usepackage{microtype}

\makeatletter
\g@addto@macro\normalsize{%
 \setlength\abovedisplayskip{13pt plus 3pt minus 4pt}%
 \setlength\belowdisplayskip{13pt plus 3pt minus 4pt}%
 \setlength\abovedisplayshortskip{0pt plus 3pt}%
 \setlength\belowdisplayshortskip{9pt plus 3.5pt minus 3pt}%
}
\makeatother

\usepackage{xcolor}
\usepackage{enumitem}
\usepackage{etoolbox}
\usepackage{csquotes}    

\usepackage{cite}

\usepackage{pdflscape}
\usepackage{array}
\usepackage{booktabs}
\usepackage{ragged2e}
\usepackage{tabularx}
\usepackage{graphicx}
\usepackage{eso-pic}

\newcolumntype{L}[1]{>{\RaggedRight\arraybackslash}p{#1}}
\newcolumntype{C}[1]{>{\Centering\arraybackslash}p{#1}}
\newcolumntype{Z}{>{\RaggedRight\arraybackslash}X}

\usepackage{fancyhdr}
\fancypagestyle{landscapetable}{%
  \fancyhf{}
  \fancyfoot[C]{\thepage}
  
}

\newcommand{\landscapepagenumber}{%
  \AddToShipoutPictureFG*{%
    \AtPageLowerLeft{%
      \put(\LenToUnit{\dimexpr\paperwidth - 12mm\relax},%
           \LenToUnit{0.5\paperheight}){%
        \rotatebox{90}{\makebox[0pt]{\normalfont\normalsize\thepage}}%
      }%
    }%
  }%
}

\numberwithin{equation}{section}

\theoremstyle{plain}
\newtheorem{X}{X}[section]
\newtheorem{theorem}[X]{Theorem}
\newtheorem{lemma}[X]{Lemma}
\newtheorem{corollary}[X]{Corollary}
\newtheorem{proposition}[X]{Proposition}
\newtheorem*{theorem*}{Theorem}

\theoremstyle{definition}

\theoremstyle{remark}
\newtheorem{remark}[X]{Remark}
\newtheorem*{remark*}{Remark}

\allowdisplaybreaks[1]

\renewcommand{\le}{\leqslant}
\renewcommand{\ge}{\geqslant}

\newcommand{\proofstep}[1]{%
 \par\medskip
 \noindent\emph{#1.}%
}

\makeatletter
\apptocmd{\thebibliography}{%
 \raggedright
 \@rightskip=\z@ \@plus 3em
 \rightskip=\@rightskip
 \parfillskip=\z@ \@plus 1fil
 \frenchspacing
}{}{\PackageWarning{preamble}{Could not patch thebibliography}}
\makeatother

\makeatletter
\patchcmd{\@tocline}
 {\hfil}
 {\leaders\hbox{$\m@th\mkern 4.5mu\hbox{.}\mkern 4.5mu$}\hfill}
 {}{\PackageWarning{preamble}{Could not patch \string\@tocline}}
\makeatother

\usepackage{hyperref}
\hypersetup{
 unicode=true,
 pdftitle={A tyro's approach to the tilted sieve: beyond the Erdős–Rankin bound},
 pdfauthor={Tristan Freiberg},
 pdfsubject={Number theory},
 pdfstartview={FitH},
 pdfmenubar=false,
 pdffitwindow=false,
 pdfnewwindow=true,
 bookmarksnumbered=true,
 linktoc=all,
 colorlinks=true,
 linkcolor={black},
 citecolor={black},
 filecolor={black},
 urlcolor={black},
}
\newcommand{\DOI}[1]{\href{https://doi.org/#1}{doi:#1}}

\title[A tyro's approach to the tilted sieve]{%
A tyro's approach to the tilted sieve:\\[0.4ex]
Beyond the Erd\H{o}s--Rankin bound}
\author[T. Freiberg]{Tristan Freiberg}
\address{Montr\'eal, Canada}
\thanks{The use of AI in developing the proof and preparing the exposition is described in Appendix~\ref{app:ai-provenance}}
\date{\today}

\begin{document}

\begin{abstract}
We give an elementary exposition of the tilted sieve introduced by GPT-5.6~Sol \cite{GPT2026}, showing that one residue class modulo each prime $p \le x$ can cover an interval of length
\begin{equation*}
\gg \frac{x\log x}{(\log_{2} x)\log_{3} x}.
\end{equation*}
This improves the classical Erd\H{o}s--Rankin bound by a factor of $(\log_{2} x)/(\log_{3} x)^{2}$. Although weaker than the strongest known bounds, it shows what the tilt and its associated covering of composite survivors achieve without Maynard sieve weights or a hypergraph covering theorem. The proof uses the prime number theorem, Mertens' reciprocal-prime formula, and elementary probability. The appendices provide a historical survey and self-contained proofs of the classical Erd\H{o}s--Rankin bound.
\end{abstract}

\maketitle

\tableofcontents
\clearpage

\section{Introduction}
\label{sec:intro}

Let $Y(x)$ be the largest integer $y$ for which one can choose one residue class $a_{p} \bmod p$ for each prime $p \le x$ whose union contains $\{1, \ldots, y\}$. The Chinese remainder theorem converts such a covering into an interval of consecutive composite integers, so lower bounds for $Y(x)$ give lower bounds for large gaps between primes. Rankin \cite{RAN1938} proved that
\begin{equation}
\label{eq:intro-rankin-covering-bound}
Y(x) \ge cx\frac{(\log x)\log_{3} x}{(\log_{2} x)^{2}}
\end{equation}
for all sufficiently large $x$, with any fixed $0 < c < 1/3$. Here $\log_{j} = \log(\log_{j - 1})$ denotes the $j\,$th iterated logarithm.

For more than seventy years, improvements concerned the constant $c$. Work of Sch\"onhage \cite{SCH1963}, Rankin \cite{RAN1963}, Maier and Pomerance \cite{MP1990}, and Pintz \cite{PIN1997} eventually allowed any constant $c < 2e^{\gamma}$ ($= 3.56214\ldots$). Maynard \cite{MAY2016} and Ford, Green, Konyagin and Tao \cite{FGKT2016} independently showed that $c$ could be taken arbitrarily large. Their subsequent joint work \cite{FGKMT2018} improved Rankin's bound by a factor of $\log_{2} x$. More recently, GPT-5.6~Sol \cite{GPT2026} obtained a further factor of $(\log_{2} x)/(\log_{3} x)^{2}$ by introducing a tilted preliminary sieve.

The full argument in \cite{GPT2026} combines this tilt with Maynard sieve weights and a hypergraph covering theorem. Our purpose is to show how much can be gained from the tilt and its associated covering of composite survivors alone. We give a self-contained treatment of this part of the construction, using only the prime number theorem and Mertens' reciprocal-prime formula as external analytic inputs.

\begin{theorem}
\label{thm:tilted-covering-bound}
For all sufficiently large $x$, 
\begin{equation*}
Y(x) \gg \frac{x\log x}{(\log_{2} x)\log_{3} x}.
\end{equation*}
The implicit constant is absolute.
\end{theorem}

Write $p_{n}$ for the $n$th prime and
\begin{equation*}
G(X) := \max_{p_{n + 1} \, \le \, X}(p_{n + 1} - p_{n}).
\end{equation*}
The Chinese remainder theorem and the prime number theorem give the following consequence.

\begin{corollary}
\label{cor:tilted-prime-gap-bound}
For all sufficiently large $X$,
\begin{equation*}
G(X) \gg \frac{(\log X)\log_{2} X}{(\log_{3} X)\log_{4} X}.
\end{equation*}
The implicit constant is absolute.
\end{corollary}

Theorem~\ref{thm:tilted-covering-bound} improves the scale of \eqref{eq:intro-rankin-covering-bound} by the unbounded factor $(\log_{2} x)/(\log_{3} x)^{2}$. For comparison, the three covering bounds are
\begin{equation*}
\begin{aligned}
Y(x) &\gg \frac{x\log x}{(\log_{2} x)\log_{3} x} && \text{(this paper)}, \\[1ex]
Y(x) &\gg \frac{x(\log x)\log_{3} x}{\log_{2} x} && \text{\cite{FGKMT2018}}, \\[1ex]
Y(x) &\gg \frac{x\log x}{\log_{3} x} && \text{\cite{GPT2026}}.
\end{aligned}
\end{equation*}
Our bound is weaker than the latter two by factors of $(\log_{3} x)^{2}$ and $\log_{2} x$, respectively. Its interest lies in the simpler argument and in what it reveals about the contribution of the tilt.

The Erd\H{o}s--Rankin construction uses some of the primes to cover most of the interval, leaving sufficiently few integers that each can be assigned a distinct unused prime. In a probabilistic formulation, it is enough that the expected number of survivors be no greater than the number of primes held in reserve. A way to go further is to make a reserve prime cover several survivors at once. Maier and Pomerance \cite{MP1990}, and later Pintz \cite{PIN1997}, achieved this by covering suitable pairs. The work of Ford, Green, Konyagin, Maynard and Tao \cite{FGKMT2018} uses sieve weights to find residue classes containing many prime survivors, and a hypergraph covering theorem to coordinate their selection. This allows the number covered per reserve prime to tend to infinity.

The tilted sieve of \cite{GPT2026} makes composite survivors available for a different kind of group covering. The preliminary choices are biased toward the zero residue class, with a bias that varies with the prime. For the principal composite candidates, the resulting survival probabilities have a simple power law, and the probabilities that several candidates survive together can also be controlled. The preliminary sieve can therefore leave more survivors than there are reserve primes: the useful information is their joint survival law, as well as their number.

We use this information to choose residue classes containing whole groups of composite survivors. Following \cite{GPT2026}, we consider candidate groups whose members all survive the preliminary sieve, weighting each by the reciprocal of its probability of surviving in full. Estimates for the joint survival probabilities then show that these choices cover almost all the composite survivors. At the interval length considered here, an entire residue class of composite candidates is a sufficiently small group, so the subdivision into blocks used in \cite{GPT2026} is unnecessary. The prime survivors and the remaining exceptions are covered individually. Thus the improvement over Rankin's bound already appears before one uses any method for covering primes in groups. 

In the classical construction, the zero residue choices force composite survivors to be smooth, and a smooth-number estimate bounds their number. Here the small-prime zero choices make survivors rough, while the tilt controls their joint survival probabilities. These probabilities allow us to cover composite survivors in groups even when there are too many to cover individually.

Finally, the strongest prime-gap bound need not come from the strongest residue-class covering. OpenAI \cite{OAI2026}, in a proof attributed to GPT-6~Astra, obtains
\begin{equation*}
G(X) \gg \frac{(\log X)(\log_{2} X)^{2}\log_{4} X}{(\log_{3} X)^{2}}.
\end{equation*}
The argument first imposes congruences modulo the primes less than or equal to $x$, then finds a short translate by a multiple of their product in which all remaining integers are composite. Primes larger than $x$ provide the additional divisors. This gives a stronger bound for $G(X)$, but does not give a corresponding improvement for $Y(x)$, whose covering primes are restricted to $p \le x$.

\subsection*{Acknowledgments}
\label{subsec:acknowledgments}

The author thanks Jori Merikoski for locating and making available a copy of Westzynthius's 1931 dissertation \cite{WES1931}.

\section{Notation}
\label{sec:notation}

The $n$th prime is denoted by $p_{n}$, and $\pi(x)$ is the number of primes less than or equal to $x$. Sums and products indexed by $p$, $q$, $r$ or $\ell$ run over primes unless otherwise indicated. For positive integers $a$ and $b$, $(a, b)$ denotes their greatest common divisor. We write $\#\mathcal{D}$ for the cardinality of a finite set $\mathcal{D}$ of integers, and $\mathcal{D}\bmod p$ for the set of residue classes modulo $p$ represented by its members. Empty sums and products have values $0$ and $1$, respectively.

For a positive integer $n$, the number of distinct prime divisors is denoted by $\omega(n)$, and their product by $\operatorname{rad}(n)$; thus $\omega(1) = 0$ and $\operatorname{rad}(1) = 1$. A positive integer is $w$-rough if all its prime factors exceed $w$, and $z$-smooth if all its prime factors are less than or equal to $z$. The integer $1$ satisfies both conditions.

We write $\log_{j} = \log(\log_{j - 1})$ for the $j$th iterated logarithm. The Euler--Mascheroni constant is denoted by $\gamma = 0.577215\ldots$.

The covering function $Y(x)$ is the largest integer $y$ for which one can choose a residue class $a_{p} \bmod p$ for every prime $p \le x$ such that
\begin{equation*}
\{1, \ldots, y\} \subseteq \bigcup_{p \, \le \, x}(a_{p} + p\mathbb{Z}).
\end{equation*}
The largest gap between consecutive primes less than or equal to $X$ is
\begin{equation*}
G(X) := \max_{p_{n + 1} \, \le \, X}(p_{n + 1} - p_{n}).
\end{equation*}

We write $\mathbb{P}(F)$ for the probability of an event $F$, $\mathbf{1}_{F}$ for its indicator, and $\mathbb{E}Z$ and $\operatorname{Var}(Z)$ for the expectation and variance of a random variable $Z$. Conditional probabilities and expectations are written $\mathbb{P}(F \mid E)$ and $\mathbb{E}(Z \mid E)$. The underlying probability space is specified in each argument. A probability distribution $\mu_{p}$ on $\mathbb{Z}/p\mathbb{Z}$ is identified with its mass function, and $\mu_{p}(n)$ means its value at the residue class of the integer $n$ modulo $p$.

In the construction, $\mathcal{N}$ denotes the set of preliminary survivors, while $\mathcal{C}$ is the deterministic set of composite candidates. For $n \in \mathcal{C}$ and $\mathcal{D} \subseteq \mathcal{C}$, we write
\begin{equation*}
q_{n} := \mathbb{P}(n \in \mathcal{N}), \qquad \Pi(\mathcal{D}) := \mathbb{P}(\mathcal{D} \subseteq \mathcal{N}).
\end{equation*}
In particular, $\Pi(\varnothing) = 1$. Conditioning on $\mathcal{N}$ means conditioning on the entire survivor set.

The reserve primes are divided into $\mathcal{R}$, used to cover composite survivors in groups, and $\mathcal{R}'$, used to mop up the remaining integers individually. The set still uncovered after the group covering is denoted by $\mathcal{M}$. The covering framework and terminology are introduced in Section~\ref{sec:covering-framework-and-strategy}; the parameters and the candidate set are specified in Section~\ref{sec:intermediate-tilted-sieve}.

For $q \in \mathcal{R}$ and an integer $a$, the corresponding group is
\begin{equation*}
E_{q,a} := \{n \in \mathcal{C} : n \equiv a \bmod q\}.
\end{equation*}
This depends only on the class of $a$ modulo $q$. In particular, $E_{q,n}$ is the group containing a candidate $n$.

For functions $f$ and $g$, with $g$ positive, we write $f = O(g)$, or equivalently $f \ll g$, if $|f| \le Cg$ throughout the range under consideration for some constant $C > 0$. The notation $g \gg f$ has the same meaning. For positive functions, $f \asymp g$ means that $f \ll g$ and $g \ll f$. Implied constants may depend on parameters explicitly held fixed, but not on variables tending to infinity; further dependence and uniformity are specified where needed.

We write $f = o(g)$ if $f/g \to 0$, and $f \sim g$ if $f/g \to 1$. Unless otherwise indicated, limits in the covering argument are taken as $x \to \infty$.

\section{Covering framework and strategy}
\label{sec:covering-framework-and-strategy}

The classical Erd\H{o}s--Rankin construction and the construction considered here begin by choosing residue classes that cover most of an interval. They differ in what they require of the integers left uncovered. In the classical construction, sufficiently few must remain for each to be assigned a separate unused prime. Here we allow more to remain, provided that many can subsequently be covered together. A probabilistic framework makes both their number and their joint survival probabilities available for study.

\subsection{The covering problem}
\label{subsec:covering-problem-and-terminology}

Let $H$ be a positive integer, and let
\begin{equation*}
2 < w < v < x < H.
\end{equation*}
Our aim is to choose one residue class $a_{p} \bmod p$ for every prime $p \le x$ such that
\begin{equation*}
\{1, \ldots, H\} \subseteq \bigcup_{p \, \le \, x}(a_{p} + p\mathbb{Z}).
\end{equation*}
Such a choice is a \emph{covering} of $\{1, \ldots, H\}$. An integer $n$ is \emph{covered} by the chosen class modulo $p$ if $n \equiv a_{p} \bmod p$, and is \emph{uncovered} after a stage of the construction if none of the classes chosen so far covers it.

We call the primes less than or equal to $v$ the \emph{preliminary primes}. Their chosen classes are the \emph{preliminary classes}, and removing the integers they cover from $\{1, \ldots, H\}$ is the \emph{preliminary sieve}. Its \emph{survivors} form the set
\begin{equation*}
\mathcal{N} := \{1, \ldots, H\} \setminus \bigcup_{p \, \le \, v}(a_{p} + p\mathbb{Z}).
\end{equation*}
A survivor means a member of this set, even if it is covered at a later stage. The primes in $(v, x]$ are the \emph{reserve primes}; their classes are chosen after the preliminary sieve.

A preliminary prime for which we prescribe $a_{p} = 0$ deterministically is called a \emph{clamping prime}. Every survivor is coprime to every clamping prime. We always clamp the primes $p \le w$, so every survivor is $w$-rough, meaning that all its prime factors exceed $w$.

In the classical construction there is a further parameter $z$, with $w < z < v$, and the primes in $(z, v]$ are also clamping primes. If $H < wv$, these two ranges force every surviving integer greater than $1$ either to have all its prime factors in $(w, z]$, or to be a prime greater than $v$. Indeed, a surviving composite with a prime factor greater than $v$ would have a complementary factor greater than $w$, and hence would exceed $H$. Thus the composite survivors are $z$-smooth, meaning that all their prime factors are less than or equal to $z$. The tilted construction uses no second clamping range.

\subsection{Random residue classes}
\label{subsec:random-residue-classes}

For each preliminary prime $p \le v$, let $\mu_{p}$ be a probability distribution on $\mathbb{Z}/p\mathbb{Z}$. We identify $\mu_{p}$ with its probability mass function and write $\mu_{p}(n)$ for its value at the residue class of an integer $n$ modulo $p$. Choose the classes $a_{p}$ independently, with
\begin{equation*}
\mathbb{P}(a_{p} \equiv n \bmod p) = \mu_{p}(n) \qquad (n \in \mathbb{Z}).
\end{equation*}
The sample space is the finite product $\prod_{p \, \le \, v}\mathbb{Z}/p\mathbb{Z}$, equipped with the product of these distributions; an outcome specifies one class for every preliminary prime. Clamping $p$ means taking $\mu_{p}(0) = 1$. We make this choice for every $p \le w$.

The survivor set $\mathcal{N}$ depends on the chosen classes, and its cardinality is a random variable. For each integer $1 \le n \le H$, independence gives
\begin{equation*}
\mathbb{P}(n \in \mathcal{N}) = \prod_{p \, \le \, v}\bigl(1 - \mu_{p}(n)\bigr).
\end{equation*}
By linearity of expectation,
\begin{equation*}
\mathbb{E}\,\#\mathcal{N} = \sum_{n \, = \, 1}^{H}\prod_{p \, \le \, v}\bigl(1 - \mu_{p}(n)\bigr).
\end{equation*}
The survival events for different integers need not be independent. Their joint probabilities are given by the following identity.

\begin{lemma}[Joint survival]
\label{lem:probabilistic-joint-survival}
For every $\mathcal{D} \subseteq \{1, \ldots, H\}$,
\begin{equation*}
\mathbb{P}(\mathcal{D} \subseteq \mathcal{N}) = \prod_{p \, \le \, v}\left(1 - \sum_{a \, \in \, \mathcal{D}\bmod p}\mu_{p}(a)\right),
\end{equation*}
where $\mathcal{D}\bmod p$ is the set of residue classes modulo $p$ represented by members of $\mathcal{D}$, each class counted once.
\end{lemma}

\begin{proof}
Every member of $\mathcal{D}$ survives the choice at $p$ precisely when the chosen class lies outside $\mathcal{D}\bmod p$. The factor indexed by $p$ is the probability of this event. Independence of the choices at different primes gives the product.
\end{proof}

Uniform selection and greedy selection are related by a simple averaging argument. Given a finite set $\mathcal{S}$ of integers still to be covered and an unused prime $p$, a uniformly chosen class modulo $p$ covers $\#\mathcal{S}/p$ members on average. A \emph{greedy choice} selects a class covering as many members of $\mathcal{S}$ as possible, with ties resolved arbitrarily, and therefore leaves at most $(1 - 1/p)\#\mathcal{S}$ members uncovered. Successive greedy choices give the same product bound for the number remaining as independent uniform choices give in expectation. The greedy choices depend on the integers left by earlier choices; they are not themselves independent selections from fixed distributions. This explains how the usual greedy step can be replaced by averaging after the clamping classes have been prescribed.

The distribution used in this paper instead biases the preliminary choices toward zero. We call this bias a \emph{tilt}, following \cite{GPT2026}. Its particular dependence on $p$ will give a power law for the survival probabilities of composite candidates. We introduce that distribution in the next section; the identities above apply to it as they do to uniform or deterministic choices.

\subsection{Mopping up the survivors}
\label{subsec:survivors-to-complete-covering}

If the number of survivors is no greater than the number of reserve primes, a covering can be completed by assigning a distinct reserve prime $p$ to each survivor $n$ and choosing $a_{p} \equiv n \bmod p$. We call this operation \emph{mopping up}. Any unused primes may be given arbitrary classes. Thus, for the classical construction, it suffices to arrange that
\begin{equation*}
\mathbb{E}\,\#\mathcal{N} \le \pi(x) - \pi(v).
\end{equation*}
Some outcome then leaves no more survivors than this expectation, and they can all be mopped up.

Our construction uses some reserve primes to cover several composite survivors at once before mopping up. We now partition the reserve primes into two fixed disjoint sets $\mathcal{R}$ and $\mathcal{R}'$. The primes in $\mathcal{R}$ are used for the group covering; those in $\mathcal{R}'$ are the \emph{final reserve primes}, held for mopping up. Their ranges will be specified with the other parameters.

After the preliminary classes have been chosen, we choose a class modulo each $p \in \mathcal{R}$ according to a distribution depending on those preliminary choices. Conditional on the preliminary choices, the choices at different primes in $\mathcal{R}$ are independent. The groups to be covered consist of composite candidates lying in the same residue class modulo a reserve prime. The joint-survival identity above will allow us to estimate the probability that all the members of a group survive, and these probabilities will guide the selection of the classes.

Let
\begin{equation*}
\mathcal{M} := \mathcal{N} \setminus \bigcup_{p \, \in \, \mathcal{R}}(a_{p} + p\mathbb{Z})
\end{equation*}
be the set of integers still uncovered after the group covering. We shall choose the distributions and parameters so that
\begin{equation*}
\mathbb{E}\,\#\mathcal{M} < \#\mathcal{R}',
\end{equation*}
where the expectation includes both stages of random choices. Some outcome therefore leaves fewer than $\#\mathcal{R}'$ integers uncovered. Fix such an outcome and mop up these integers with the final reserve primes.

The estimates that follow will account separately for the composite survivors covered in groups, the prime survivors, and the exceptional integers left for mopping up. The preliminary sieve need not leave few survivors in total. It must leave survivors whose joint probabilities permit an effective group covering, after which sufficiently few integers remain for the final reserve primes.

\section{The tilted preliminary sieve}
\label{sec:intermediate-tilted-sieve}

We now specify the preliminary distributions. Following \cite[Section~3]{GPT2026}, we give equal probability to all nonzero classes modulo each prime $w < p \le v$, while allowing the probability of zero to vary. This makes the survival probability of a $w$-rough integer a product of factors associated with its prime divisors in $(w, v]$. We choose these factors so that, for a squarefree integer whose prime factors all lie in this range, the survival probability depends only on its size.

\subsection{The power tilt}
\label{subsec:power-tilt}

Keep the clamping choices $\mu_{p}(0) = 1$ for $p \le w$. For each prime satisfying $w < p \le v$, choose a number $0 \le \beta_{p} \le 1$ and put
\begin{equation}
\label{eq:preliminary-beta-distributions}
\mu_{p}(0) := 1 - \beta_{p}, \qquad \mu_{p}(a) := \frac{\beta_{p}}{p - 1} \quad (a \not\equiv 0 \bmod p).
\end{equation}
Thus $\beta_{p}$ is the probability of choosing a nonzero class. The choice $\beta_{p} = 0$ clamps $p$, while $\beta_{p} = (p - 1)/p$ gives the uniform distribution on all classes modulo $p$. Values between these two choices bias the distribution toward zero.

A multiple of $p$ survives this choice with probability $\beta_{p}$. A nonmultiple survives with probability
\begin{equation*}
B_{p} := 1 - \frac{\beta_{p}}{p - 1}.
\end{equation*}
Since $p > w > 2$, we have $B_{p} > 0$. Write
\begin{equation*}
\lambda_{p} := \frac{\beta_{p}}{B_{p}},
\end{equation*}
so that $\lambda_{p}$ is the ratio of the survival probability for a multiple of $p$ to that for a nonmultiple. Finally, put
\begin{equation*}
A := \prod_{w \, < \, p \, \le \, v}B_{p}.
\end{equation*}
These definitions do not yet impose a particular dependence on $p$.

\begin{lemma}[Multiplicative survival law]
\label{lem:probabilistic-multiplicative-survival}
Under the independent choices \eqref{eq:preliminary-beta-distributions}, every $w$-rough integer $1 \le n \le H$ satisfies
\begin{equation}
\label{eq:probabilistic-multiplicative-law}
\mathbb{P}(n \in \mathcal{N}) = A\prod_{\substack{w \, < \, p \, \le \, v \\ p \, \mid \, n}}\lambda_{p}.
\end{equation}
Integers in $\{1, \ldots, H\}$ that are not $w$-rough have survival probability zero.
\end{lemma}

\begin{proof}
The clamping primes cover every integer with a prime factor less than or equal to $w$, and cover no $w$-rough integer. For $w < p \le v$, the local survival probability is $B_{p}$ if $p \nmid n$ and $\beta_{p} = B_{p}\lambda_{p}$ if $p \mid n$. Multiplying these factors proves \eqref{eq:probabilistic-multiplicative-law}.
\end{proof}

The factor contributed by a prime divisor appears once, regardless of its multiplicity. For squarefree integers, we can make this product depend only on the size of the integer by choosing the same power at every prime. Let $\tau > 0$ be a parameter, to be chosen later, and take $\lambda_{p} = p^{-\tau}$. Then, for every squarefree integer $n$ whose prime factors all lie in $(w, v]$,
\begin{equation*}
\prod_{p \, \mid \, n}\lambda_{p} = \prod_{p \, \mid \, n}p^{-\tau} = n^{-\tau}.
\end{equation*}
Solving $\beta_{p}/B_{p} = p^{-\tau}$ gives the \emph{power tilt}
\begin{equation}
\label{eq:intermediate-tilted-law}
\beta_{p} := \frac{(p - 1)p^{-\tau}}{p - 1 + p^{-\tau}} \qquad (w < p \le v),
\end{equation}
with
\begin{equation}
\label{eq:intermediate-local-parameters}
B_{p} = \frac{p - 1}{p - 1 + p^{-\tau}}, \qquad \lambda_{p} = p^{-\tau}.
\end{equation}
The choice \eqref{eq:intermediate-tilted-law} gives the uniform distribution at $\tau = 0$. For $\tau > 0$, the probability of zero is greater than $1/p$, and every nonzero class has probability less than $1/p$. The word \emph{tilt} refers to this bias, and \emph{power} to the prescribed ratio $p^{-\tau}$.

Increasing $\tau$ decreases the factor $n^{-\tau}$, but also increases $A$. In particular, it increases the survival probability of primes greater than $v$, which have no prime divisor in the product in \eqref{eq:probabilistic-multiplicative-law}. The eventual choice of $\tau$ must balance the covering of composite survivors against the number of prime survivors left for mopping up.

\subsection{Parameters and composite candidates}
\label{subsec:intermediate-parameters}

Let $0 < c \le 1$ be fixed, to be chosen sufficiently small. For sufficiently large $x$, take
\begin{equation}
\label{eq:target-interval-length}
H := \left\lfloor\frac{cx\log x}{(\log_{2} x)\log_{3} x}\right\rfloor,
\end{equation}
and set
\begin{equation}
\label{eq:intermediate-length-parameters}
w := (\log x)^{100}, \qquad v := \frac{x}{2}.
\end{equation}
The exponent $100$ gives ample room for the estimates involving small prime factors. No optimization of this fixed exponent is needed. The choice $v = x/2$ leaves a fixed proportion of the available primes in reserve. We divide them into
\begin{equation*}
\mathcal{R} := \{p : x/2 < p \le 3x/4\}, \qquad \mathcal{R}' := \{p : 3x/4 < p \le x\}.
\end{equation*}
The prime number theorem gives
\begin{equation}
\label{eq:intermediate-prime-budgets}
\#\mathcal{R} \sim \#\mathcal{R}' \sim \frac{x}{4\log x}.
\end{equation}
We will also use the sum of the reciprocals of the group-covering primes. By partial summation and the prime number theorem,
\begin{equation}
\label{eq:intermediate-reserve-reciprocals}
M := \sum_{p \, \in \, \mathcal{R}}\frac{1}{p} \sim \int_{x/2}^{3x/4}\frac{du}{u\log u} \sim \frac{\log(3/2)}{\log x}.
\end{equation}

We will use the primes in $\mathcal{R}$ to cover groups of squarefree $w$-rough composite integers. Put
\begin{equation}
\label{eq:initial-interval-cutoff}
T := \left\lfloor\frac{x}{(\log x)^{2}}\right\rfloor,
\end{equation}
and define
\begin{equation*}
\mathcal{C} := \{n \in \mathbb{Z} : T < n \le H,\ n \text{ is squarefree, } w\text{-rough and composite}\}.
\end{equation*}
We call the members of $\mathcal{C}$ \emph{composite candidates}. For each prime $p \in \mathcal{R}$, a group consists of the members of $\mathcal{C}$ in one residue class modulo $p$; choosing that class covers every member of the group. The set $\mathcal{C}$ is fixed before the preliminary classes are chosen, so a candidate need not be a survivor.

The lower cutoff $T$ serves two purposes. First, there are only $T = o(x/\log x)$ integers in $\{1, \ldots, T\}$, a negligible number compared with the final reserve count in \eqref{eq:intermediate-prime-budgets}. Any that remain uncovered can therefore be left for mopping up. Second, the survival probability of each member $n$ of $\mathcal{C}$ is $An^{-\tau}$, as we shall prove below. With the tilt used below, $\tau\log(H/T) = o(1)$, so these probabilities vary by a factor tending to $1$ throughout $\mathcal{C}$. The parameter $T$ thus restricts the integers considered for group covering; unlike the classical parameter $z$, it does not separate ranges of primes receiving different residue choices.

Every survivor greater than $T$ is either a prime, a member of $\mathcal{C}$, or a nonsquarefree $w$-rough composite. The last type will also be left for mopping up; its number will be estimated when we count the remaining integers.

For all sufficiently large $x$, the parameters satisfy
\begin{equation}
\label{eq:intermediate-parameter-relations}
2 < w < T < v < x < H < wv < \frac{x^{2}}{4}.
\end{equation}
Indeed, these comparisons follow from \eqref{eq:target-interval-length}, \eqref{eq:intermediate-length-parameters} and \eqref{eq:initial-interval-cutoff}. In particular,
\begin{equation*}
\frac{H}{wv} \ll \frac{1}{(\log x)^{99}(\log_{2} x)\log_{3} x} \to 0.
\end{equation*}
The inequality $H < wv$ will constrain the prime factors of composite candidates; $H < x^{2}/4$ will control intersections between their residue classes modulo reserve primes.

For the sizes of these classes, put
\begin{equation*}
U := \left\lceil\frac{2H}{x}\right\rceil + 2.
\end{equation*}
A residue class modulo any $p \in \mathcal{R}$ contains at most $(H/p) + 1 < U$ integers in $\{1, \ldots, H\}$. Moreover,
\begin{equation}
\label{eq:intermediate-group-size-bound}
U \ll \frac{\log x}{(\log_{2} x)\log_{3} x}, \qquad 2U < w.
\end{equation}
The second inequality holds for sufficiently large $x$ and will ensure that these groups are small compared with every preliminary prime greater than $w$.

\subsection{Individual survival probabilities}
\label{subsec:intermediate-one-point-probabilities}

We now restrict the tilt to a range that will contain its final choice. Write
\begin{equation}
\label{eq:intermediate-tilt-range}
\tau := \frac{t}{\log x}, \qquad \log_{3} x \le t \le 3\log_{3} x.
\end{equation}
This parametrization is convenient because $n^{-\tau} = (1 + o(1))e^{-t}$ uniformly for $T < n \le H$, as the estimates below will show. Until $t$ is chosen, all asymptotic estimates are uniform in this range. Implied constants may depend on the fixed $c$, but not on $x$, $t$, or the primes and integers under consideration.

\begin{lemma}[Individual survival]
\label{lem:intermediate-one-point-law}
Every prime factor of a $w$-rough composite less than or equal to $H$ lies in $(w, v]$. Consequently, for every $n \in \mathcal{C}$,
\begin{equation}
\label{eq:intermediate-composite-probability}
q_{n} := \mathbb{P}(n \in \mathcal{N}) = An^{-\tau}.
\end{equation}
For a prime $\ell \le H$,
\begin{equation*}
\mathbb{P}(\ell \in \mathcal{N}) =
\begin{cases}
0, & \ell \le w, \\
A\ell^{-\tau}, & w < \ell \le v, \\
A, & v < \ell \le H.
\end{cases}
\end{equation*}
\end{lemma}

\begin{proof}
If a $w$-rough composite $n \le H$ had a prime factor $p > v$, its complementary factor $n/p$ would be greater than $w$. This would give $n > wv > H$, contrary to \eqref{eq:intermediate-parameter-relations}. All its prime factors therefore lie in $(w, v]$. If $n \in \mathcal{C}$, it is also squarefree, so Lemma~\ref{lem:probabilistic-multiplicative-survival} and \eqref{eq:intermediate-local-parameters} give $q_{n} = An^{-\tau}$. The same lemma gives the assertions about primes.
\end{proof}

Before estimating $A$, we record the logarithmic comparisons needed for uniformity. The definitions of $H$ and $T$ give
\begin{equation}
\label{eq:intermediate-logarithmic-lengths}
\log H = \log x + O(\log_{2} x), \qquad \log T = \log x + O(\log_{2} x).
\end{equation}
Since $\log w = 100\log_{2} x$, the range \eqref{eq:intermediate-tilt-range} implies
\begin{equation}
\label{eq:intermediate-small-lower-endpoint}
\tau\log w = \frac{100t\log_{2} x}{\log x} = o(1).
\end{equation}
Similarly,
\begin{equation}
\label{eq:intermediate-small-probability-variation}
0 < \tau\log(H/T) \ll \frac{(\log_{2} x)\log_{3} x}{\log x} = o(1).
\end{equation}
Finally, \eqref{eq:intermediate-group-size-bound} gives the useful group estimate
\begin{equation*}
Ut \ll \frac{\log x}{\log_{2} x}.
\end{equation*}
These estimates separate the two effects of the tilt: it has negligible variation across the candidate interval, while $t \to \infty$ makes the common factor $e^{-t}$ tend to zero.

\subsection{Sizes of the survival probabilities}
\label{subsec:intermediate-survival-probability-sizes}

It remains to estimate $A$, which is also the survival probability of every prime greater than $v$. Mertens' reciprocal-prime formula gives its leading term.

\begin{lemma}[The normalizing product]
\label{lem:intermediate-normalizing-product}
Uniformly in the range \eqref{eq:intermediate-tilt-range},
\begin{equation}
\label{eq:intermediate-product-estimate}
A = (e^{\gamma} + o(1))\tau\log w = (100e^{\gamma} + o(1))\frac{t\log_{2} x}{\log x}.
\end{equation}
\end{lemma}

\begin{proof}
By \eqref{eq:intermediate-local-parameters},
\begin{equation*}
-\log A = \sum_{w \, < \, p \, \le \, v}\log\left(1 + \frac{p^{-\tau}}{p - 1}\right) = \sum_{w \, < \, p \, \le \, v}\frac{p^{-\tau}}{p} + O(1/w).
\end{equation*}
Here the error follows from $p^{-\tau} \le 1$ and $\sum_{p \, > \, w}p^{-2} \ll 1/w$.

Mertens' formula states that
\begin{equation*}
\sum_{p \, \le \, u}\frac{1}{p} = \log\log u + b + o(1),
\end{equation*}
where $b$ is an absolute constant. Partial summation against $u^{-\tau}$ therefore gives
\begin{equation}
\label{eq:intermediate-product-integral}
\sum_{w \, < \, p \, \le \, v}\frac{p^{-\tau}}{p} = \int_{w}^{v}\frac{du}{u^{1 + \tau}\log u} + o(1) = \int_{\tau\log w}^{\tau\log v}\frac{e^{-s}}{s}\,ds + o(1).
\end{equation}
The error is uniform in $t$: the error in Mertens' formula tends uniformly to zero for $u \ge w$, while $u^{-\tau}$ is bounded by $1$ and has total variation at most $1$ on $[w, v]$.

We use the elementary integral identity
\begin{equation*}
\gamma = \int_{0}^{1}\frac{1 - e^{-s}}{s}\,ds - \int_{1}^{\infty}\frac{e^{-s}}{s}\,ds.
\end{equation*}
Splitting the integral at $1$ shows that, as $a \downarrow 0$,
\begin{equation}
\label{eq:intermediate-exponential-integral}
\int_{a}^{\infty}\frac{e^{-s}}{s}\,ds = -\log a - \gamma + O(a).
\end{equation}
The lower endpoint in \eqref{eq:intermediate-product-integral} tends to zero by \eqref{eq:intermediate-small-lower-endpoint}. At the upper endpoint,
\begin{equation*}
\tau\log v = t + o(1) \to \infty,
\end{equation*}
so extending the integral to infinity introduces an error tending to zero uniformly in $t$. Equations~\eqref{eq:intermediate-product-integral} and \eqref{eq:intermediate-exponential-integral} now give
\begin{equation*}
-\log A = -\log(\tau\log w) - \gamma + o(1).
\end{equation*}
Exponentiating and substituting $\log w = 100\log_{2} x$ proves \eqref{eq:intermediate-product-estimate}.
\end{proof}

\begin{lemma}[A common scale for composite survival]
\label{lem:intermediate-probability-comparison}
Put
\begin{equation*}
Q := AT^{-\tau}.
\end{equation*}
Uniformly for $n \in \mathcal{C}$ and $t$ in \eqref{eq:intermediate-tilt-range},
\begin{equation}
\label{eq:intermediate-probability-comparison}
q_{n} \le Q, \qquad q_{n} = (1 + o(1))Q.
\end{equation}
Moreover,
\begin{equation}
\label{eq:intermediate-common-probability-size}
Q = (1 + o(1))Ae^{-t} = (100e^{\gamma} + o(1))\frac{te^{-t}\log_{2} x}{\log x}.
\end{equation}
\end{lemma}

\begin{proof}
By \eqref{eq:intermediate-composite-probability}, for every $n \in \mathcal{C}$,
\begin{equation*}
\exp\bigl(-\tau\log(H/T)\bigr) \le \frac{q_{n}}{Q} = \left(\frac{T}{n}\right)^{\tau} \le 1.
\end{equation*}
Equation~\eqref{eq:intermediate-small-probability-variation} proves \eqref{eq:intermediate-probability-comparison}. Also, \eqref{eq:intermediate-logarithmic-lengths} gives $\tau\log T = t + o(1)$ uniformly in $t$. Hence $T^{-\tau} = (1 + o(1))e^{-t}$, and \eqref{eq:intermediate-product-estimate} proves \eqref{eq:intermediate-common-probability-size}.
\end{proof}

The two relevant scales are now explicit: primes greater than $v$ survive with probability $A$, while composite candidates survive with probability asymptotic to $Ae^{-t}$. We will choose $t$ after estimating how effectively the reserve primes cover composite survivors in groups. That step requires joint survival probabilities, to which we now turn.

\section{Joint survival and groups of composite candidates}
\label{sec:intermediate-joint-probabilities}

A group will consist of all the composite candidates in one residue class modulo a prime in $\mathcal{R}$. If every member of a group survives the preliminary sieve, choosing that class covers them all at once. In the next section we will assign such a group a weight equal to the reciprocal of its probability of surviving in full. To control the sums of these weights, we need to know how the survival of one group affects the survival of another. We follow the method of \cite[Sections~4--5]{GPT2026}, using entire residue classes as groups.

There are two kinds of overlap to distinguish. Groups for the same reserve prime are disjoint, and groups for different reserve primes have at most one integer in common (two common integers would differ by a multiple of the product of those primes, which exceeds $H$). But even disjoint groups can contain integers with common prime factors, and these cause dependence between their survival events. The purpose of this section is to express that dependence through greatest common divisors and show that its contribution is small on average. We will not need to select a disjoint family of groups.

Throughout this section the parameters are those of Section~\ref{sec:intermediate-tilted-sieve}, with $t$ in the range \eqref{eq:intermediate-tilt-range}. Recall that
\begin{equation*}
U = \left\lceil\frac{2H}{x}\right\rceil + 2
\end{equation*}
bounds the number of integers in $(T, H]$ in one residue class modulo a reserve prime, and satisfies \eqref{eq:intermediate-group-size-bound}. All estimates below are uniform in $t$ and in the indicated candidates and reserve primes.

\subsection{Joint survival probabilities}
\label{subsec:intermediate-small-set-estimate}

For $\mathcal{D} \subseteq \mathcal{C}$, write
\begin{equation*}
\Pi(\mathcal{D}) := \mathbb{P}(\mathcal{D} \subseteq \mathcal{N}).
\end{equation*}
Thus $\Pi(\mathcal{D})$ is the probability that every member of $\mathcal{D}$ survives, and $\Pi(\varnothing) = 1$. We write $\operatorname{rad}(m)$ for the product of the distinct prime divisors of $m$, with $\operatorname{rad}(1) = 1$.

The radical appears because a prime dividing several members of $\mathcal{D}$ imposes only one requirement: the chosen class at that prime must avoid zero. Its factor $p^{-\tau}$ therefore occurs once in the joint probability, rather than once for each divisible member. The following lemma makes this observation precise, with an error accounting for the nonzero residue classes.

\begin{lemma}[Joint survival estimate]
\label{lem:intermediate-joint-law}
If $\mathcal{D} \subseteq \mathcal{C}$ has $k \le 2U$ members, then
\begin{equation}
\label{eq:intermediate-joint-radical-law}
\Pi(\mathcal{D}) = A^{k}\operatorname{rad}\left(\prod_{n \, \in \, \mathcal{D}}n\right)^{-\tau}\exp\left(O\left(\frac{k^{2}\log H}{w\log w}\right)\right).
\end{equation}
In particular, $\Pi(\mathcal{D}) > 0$.
\end{lemma}

\begin{proof}
The assertion is immediate for the empty set, so suppose $k \ge 1$ and put $N := \prod_{n \, \in \, \mathcal{D}}n$. For $w < p \le v$, let $s_{p}$ be the number of distinct nonzero classes modulo $p$ represented by members of $\mathcal{D}$. The general joint-survival identity in Lemma~\ref{lem:probabilistic-joint-survival} gives
\begin{equation*}
\Pi(\mathcal{D}) = \prod_{w \, < \, p \, \le \, v}f_{p}(\mathcal{D}),
\end{equation*}
where
\begin{equation}
\label{eq:intermediate-exact-joint-factors}
f_{p}(\mathcal{D}) =
\begin{cases}
1 - s_{p}\beta_{p}/(p - 1), & p \nmid N, \\
\beta_{p}\left(1 - s_{p}/(p - 1)\right), & p \mid N.
\end{cases}
\end{equation}
In the second case the choice must avoid zero as well as the $s_{p}$ occupied nonzero classes. The clamping primes contribute factors of $1$, since every candidate is $w$-rough. All these factors are positive: $0 < \beta_{p} < 1$ and $s_{p} \le p - 1$, with $s_{p} \le k - 1 < p - 1$ when $p \mid N$.

Recall that 
\begin{equation*}
A := \prod_{w \, < \, p \, \le \, v}B_{p}
\end{equation*}
and that every prime divisor of $N$ lies in $(w, v]$ by Lemma~\ref{lem:intermediate-one-point-law}. Thus the proposed main term $A^{k}\operatorname{rad}(N)^{-\tau}$ is the product of the factors $B_{p}^{k}$, with an additional factor $p^{-\tau}$ precisely when $p \mid N$. We estimate the logarithm of the ratio of each actual local factor to its proposed main factor; summing these logarithms will give
\begin{equation*}
\log\frac{\Pi(\mathcal{D})}{A^{k}\operatorname{rad}(N)^{-\tau}}.
\end{equation*}

If $p \nmid N$, expansion of $\log(1 - u)$ gives
\begin{equation*}
\log\frac{f_{p}(\mathcal{D})}{B_{p}^{k}} = O\left(\frac{k - s_{p}}{p} + \frac{k^{2}}{p^{2}}\right).
\end{equation*}
If $p \mid N$, use $\beta_{p} = B_{p}p^{-\tau}$ in \eqref{eq:intermediate-exact-joint-factors} to obtain
\begin{equation*}
\log\frac{f_{p}(\mathcal{D})}{B_{p}^{k}p^{-\tau}} = \log\left(1 - \frac{s_{p}}{p - 1}\right) - (k - 1)\log B_{p} = O(k/p).
\end{equation*}
Since $k/p \le 2U/w = o(1)$, all arguments of these logarithmic expansions lie in $[0, 1/2]$ for sufficiently large $x$, so the implied constants are absolute.

To sum the errors, observe that for every nonzero integer $d$,
\begin{equation}
\label{eq:intermediate-divisor-reciprocals}
\sum_{\substack{p \, \mid \, d \\ p \, > \, w}}\frac{1}{p} \le \frac{\log|d|}{w\log w}.
\end{equation}
Indeed, there are at most $(\log|d|)/\log w$ such prime divisors, and each contributes less than $1/w$. Applying \eqref{eq:intermediate-divisor-reciprocals} to $N \le H^{k}$ gives
\begin{equation*}
k\sum_{p \, \mid \, N}\frac{1}{p} \le \frac{k^{2}\log H}{w\log w}.
\end{equation*}
To bound the terms $(k - s_{p})/p$ arising when $p \nmid N$, observe that $k - s_{p}$ is at most the number of unordered pairs of distinct members of $\mathcal{D}$ that are congruent modulo $p$. Interchanging the sums and applying \eqref{eq:intermediate-divisor-reciprocals} to each positive difference $m - n < H$ therefore gives
\begin{equation*}
\sum_{\substack{w \, < \, p \, \le \, v \\ p \, \nmid \, N}}\frac{k - s_{p}}{p}
\le \sum_{\substack{n, m \, \in \, \mathcal{D} \\ n \, < \, m}}\ \sum_{\substack{p \, \mid \, m - n \\ p \, > \, w}}\frac{1}{p} 
\le \frac{k^{2}\log H}{w\log w}.
\end{equation*}
Finally, $k^{2}\sum_{p \, > \, w}p^{-2} \ll k^{2}/w$, which is bounded by the same quantity. Adding the local logarithmic errors therefore gives
\begin{equation*}
\log\frac{\Pi(\mathcal{D})}{A^{k}\operatorname{rad}(N)^{-\tau}}
= O\left(\frac{k^{2}\log H}{w\log w}\right).
\end{equation*}
Exponentiating proves \eqref{eq:intermediate-joint-radical-law}.
\end{proof}

\subsection{Residue classes as groups}
\label{subsec:intermediate-residue-groups}

In calculations involving both kinds of prime, we use $p$ for a preliminary prime and $q$ or $r$ for a reserve prime. For $q \in \mathcal{R}$ and an integer $a$, define
\begin{equation*}
E_{q,a} := \{n \in \mathcal{C} : n \equiv a \bmod q\}, \qquad N_{q,a} := \prod_{n \, \in \, E_{q,a}}n.
\end{equation*}
The set $E_{q,a}$ is the group associated with the class $a \bmod q$. It depends only on that class. In particular, $E_{q,n}$ denotes the group containing a candidate $n$. For a fixed $q$, the nonempty groups partition $\mathcal{C}$. We retain all $q$ residue labels, including those whose groups are empty, and take the corresponding empty products to be $1$.

These groups are defined before the preliminary sieve. Their members need not survive, and the groups for different primes need not be disjoint. The following properties follow from their spacing and the prime factors allowed in $\mathcal{C}$.

\begin{lemma}[Properties of the groups]
\label{lem:intermediate-group-geometry}
Each $E_{q,a}$ has at most $U$ members. Its members occupy distinct residue classes modulo every preliminary prime $w < p \le v$ and are pairwise coprime. Consequently $N_{q,a}$ is squarefree. For distinct reserve primes $q$ and $r$, any group modulo $q$ and any group modulo $r$ have at most one member in common. In particular,
\begin{equation}
\label{eq:group-intersection-at-candidate}
E_{q,n} \cap E_{r,n} = \{n\} \qquad (n \in \mathcal{C},\ q \ne r).
\end{equation}
Furthermore, uniformly in $q$ and $a$,
\begin{equation}
\label{eq:intermediate-inverse-group-probability}
\Pi(E_{q,a})^{-1} = x^{o(1)}.
\end{equation}
\end{lemma}

\begin{proof}
The size bound follows from $q > x/2$ and the definition of $U$. Distinct members of $E_{q,a}$ have difference $hq$, where
\begin{equation*}
0 < |h| < H/q < U < w.
\end{equation*}
No preliminary prime $p > w$ divides this difference: it divides neither $h$ nor $q$. The members therefore have distinct residues modulo $p$. All prime factors of all candidates are preliminary primes greater than $w$, so two members of the same group cannot have a common prime factor. Each member is squarefree, and hence so is their product.

If two distinct integers belonged to groups modulo both $q$ and $r$, their nonzero difference would be divisible by $qr > x^{2}/4 > H$. Since both integers lie in $(T, H]$, this is impossible. This proves the intersection assertion and \eqref{eq:group-intersection-at-candidate}.

For the probability bound, Lemma~\ref{lem:intermediate-joint-law} gives
\begin{equation*}
\log\Pi(E_{q,a})^{-1} \le U\log A^{-1} + \tau U\log H + O\left(\frac{U^{2}\log H}{w\log w}\right).
\end{equation*}
By \eqref{eq:intermediate-product-estimate}, $\log A^{-1} = O(\log_{2} x)$. Also $\tau\log H = O(\log_{3} x)$. Using \eqref{eq:intermediate-group-size-bound}, the preceding bound is therefore
\begin{equation*}
O\left(\frac{\log x}{\log_{3} x}\right) = o(\log x).
\end{equation*}
Exponentiating proves \eqref{eq:intermediate-inverse-group-probability}.
\end{proof}

Thus the reciprocal probabilities used as weights later grow more slowly than every fixed positive power of $x$. The size bound $U$ is essential here: we can use the entire group without its survival probability becoming too small for this estimate.

\subsection{Correlations and common prime factors}
\label{subsec:intermediate-correlation-ratios}

For two groups $E$ and $F$, the ratio 
\begin{equation*}
\frac{\Pi(E \cup F)}{\Pi(E)\Pi(F)}
\end{equation*}
compares their joint survival probability with the product of their separate survival probabilities. A ratio of $1$ means that the two survival events are independent. We need two versions of this comparison.

First, fix a reserve prime $q$ and consider distinct labels $a$ and $b$ modulo $q$. The groups $E_{q,a}$ and $E_{q,b}$ are disjoint as sets of integers, but their products can have common prime factors. Put
\begin{equation*}
G_{q}(a, b) := (N_{q,a}, N_{q,b}).
\end{equation*}

Second, fix a candidate $n$ and consider its groups modulo distinct reserve primes $q$ and $r$. Both contain $n$, so their survival events cannot be treated as independent without first accounting for the survival of $n$. We will condition on $n \in \mathcal{N}$. By \eqref{eq:group-intersection-at-candidate}, removing $n$ leaves disjoint sets. Define
\begin{equation*}
G_{n}(q, r) := \left(\prod_{m \, \in \, E_{q,n} \setminus \{n\}}m,\ \prod_{m \, \in \, E_{r,n} \setminus \{n\}}m\right).
\end{equation*}
Both kinds of greatest common divisor are squarefree. To state the error in the correlation estimates, put
\begin{equation*}
\Delta := \frac{U^{2}\log H}{w\log w}.
\end{equation*}
Since $U \ll \log x$, $\log H \asymp \log x$ and $w = (\log x)^{100}$, we have
\begin{equation}
\label{eq:joint-correlation-error-bound}
\Delta \ll (\log x)^{-90}.
\end{equation}

\begin{lemma}[Correlation ratios]
\label{lem:intermediate-correlation-ratios}
For $q \in \mathcal{R}$ and distinct labels $a$ and $b$ modulo $q$,
\begin{equation}
\label{eq:intermediate-unconditioned-ratio}
\frac{\Pi(E_{q,a} \cup E_{q,b})}{\Pi(E_{q,a})\Pi(E_{q,b})} = G_{q}(a, b)^{\tau}\exp(O(\Delta)).
\end{equation}
For $n \in \mathcal{C}$ and distinct $q, r \in \mathcal{R}$,
\begin{equation}
\label{eq:intermediate-conditioned-ratio}
\frac{q_{n}\Pi(E_{q,n} \cup E_{r,n})}{\Pi(E_{q,n})\Pi(E_{r,n})} = G_{n}(q, r)^{\tau}\exp(O(\Delta)).
\end{equation}
\end{lemma}

\begin{proof}
Each group has at most $U$ members, so Lemma~\ref{lem:intermediate-joint-law} applies to the two groups separately and to their union. For the first assertion, both group products are squarefree and the groups are disjoint. Thus
\begin{equation*}
\operatorname{rad}\left(\prod_{m \, \in \, E_{q,a} \cup E_{q,b}}m\right) = \frac{N_{q,a}N_{q,b}}{G_{q}(a, b)}.
\end{equation*}
Substitution in \eqref{eq:intermediate-joint-radical-law} cancels the powers of $A$ and leaves the factor $G_{q}(a, b)^{\tau}$. The errors in the three applications are all $O(\Delta)$, proving \eqref{eq:intermediate-unconditioned-ratio}.

For the second assertion, write $N_{q,n} = nN_{1}$ and $N_{r,n} = nN_{2}$. The group properties imply that $n$ is coprime to $N_{1}N_{2}$ and that both $N_{1}$ and $N_{2}$ are squarefree. Since the groups intersect only at $n$,
\begin{equation*}
\operatorname{rad}\left(\prod_{m \, \in \, E_{q,n} \cup E_{r,n}}m\right) = \frac{nN_{1}N_{2}}{G_{n}(q, r)}.
\end{equation*}
The union has one fewer member than the sum of the two group sizes. Hence the unconditioned ratio is $A^{-1}n^{\tau}G_{n}(q, r)^{\tau}\exp(O(\Delta))$. Multiplication by $q_{n} = An^{-\tau}$ gives \eqref{eq:intermediate-conditioned-ratio}.
\end{proof}

To see the conditional meaning of the second formula, note that any group $E$ containing $n$ satisfies
\begin{equation*}
\mathbb{P}(E \subseteq \mathcal{N} \mid n \in \mathcal{N}) = \frac{\Pi(E)}{q_{n}}.
\end{equation*}
Dividing the conditional joint probability by the product of the two conditional probabilities therefore gives exactly the left-hand side of \eqref{eq:intermediate-conditioned-ratio}.

The ratios need not be close to $1$ for every pair: a large common factor can increase the chance of joint survival. We next show that the factors $G^{\tau}$ are close to $1$ on average, for each of the two averages needed in the selection argument.

\subsection{Averaging the common factors}
\label{subsec:intermediate-common-factor-average}

We use two auxiliary probability spaces, solely to express finite averages over the groups. In the first, fix $q \in \mathcal{R}$ and choose an ordered pair of distinct labels $a, b$ uniformly modulo $q$; each pair has probability $1/(q(q - 1))$. Set $G := G_{q}(a, b)$. In the second, fix $n \in \mathcal{C}$ and choose an ordered pair of distinct primes $q, r \in \mathcal{R}$ with probability proportional to $1/(qr)$; set $G := G_{n}(q, r)$. Write $\mathbb{P}_{*}$ and $\mathbb{E}_{*}$ for probability and expectation in whichever of these two spaces is being used. These averages do not involve the preliminary random choices.

The weights $1/(qr)$ in the second average arise because the later selection rule assigns weights containing $1/q$ and $1/r$ to the two reserve primes. Their normalizing sum is
\begin{equation}
\label{eq:intermediate-prime-pair-normalizer}
\sum_{\substack{q, r \, \in \, \mathcal{R} \\ q \, \ne \, r}}\frac{1}{qr} = M^{2} - \sum_{q \, \in \, \mathcal{R}}\frac{1}{q^{2}} \asymp \frac{1}{(\log x)^{2}}.
\end{equation}
Here we used \eqref{eq:intermediate-reserve-reciprocals} and the bound $\sum_{q \, \in \, \mathcal{R}}q^{-2} = O(1/x) = o(M^{2})$. Thus the probability assigned to an individual ordered pair $(q, r)$ is
\begin{equation*}
\frac{1/(qr)}{\displaystyle\sum_{\substack{q', r' \, \in \, \mathcal{R} \\ q' \, \ne \, r'}}1/(q'r')}
\asymp \frac{(\log x)^{2}}{x^{2}},
\end{equation*}
since both primes lie in $(x/2, 3x/4]$.

We first bound the probability that a prescribed integer $d$ divides $G$. The integer $d$ and its prime factors remain fixed: we count the pairs of labels or reserve primes for which the divisibility holds. Here $\omega(d)$ denotes the number of distinct prime divisors of $d$.

\begin{lemma}[Counting common divisors]
\label{lem:intermediate-divisor-count}
For either auxiliary probability space and every squarefree $w$-rough integer $1 < d \le x$,
\begin{equation}
\label{eq:intermediate-divisor-count}
\mathbb{P}_{*}(d \mid G) \ll (\log x)^{2}\frac{(4U^{2})^{\omega(d)}}{d^{2}}.
\end{equation}
\end{lemma}

\begin{proof}
Write $d = p_{1}\cdots p_{j}$, where $j := \omega(d)$ and the primes $p_{i}$ are distinct. Every prime factor of $G$ lies in $(w, v]$. Thus the probability is zero if $d$ has a prime factor greater than $v$, and we may suppose that all $p_{i}$ lie in $(w, v]$.

For the first average, let $b_{a}$ be the unique representative of the class $a \bmod q$ in $(T, T + q]$. Every member of $E_{q,a}$ can be written as $b_{a} + hq$, where $0 \le h < U$ is an integer. If $d \mid G_{q}(a, b)$, each prime $p_{i}$ divides one member of each group. Hence there are integer offsets satisfying
\begin{equation*}
0 \le h_{i}, \ell_{i} < U, \qquad
p_{i} \mid b_{a} + h_{i}q, \qquad
p_{i} \mid b_{b} + \ell_{i}q
\qquad (1 \le i \le j).
\end{equation*}
There are at most $U^{2j}$ possible vectors $(h_{1}, \ldots, h_{j}, \ell_{1}, \ldots, \ell_{j})$.

For each fixed vector, the congruences
\begin{equation*}
b_{a} \equiv -h_{i}q \bmod p_{i}, \qquad
b_{b} \equiv -\ell_{i}q \bmod p_{i}
\qquad (1 \le i \le j)
\end{equation*}
determine one class modulo $d$ for each of $b_{a}$ and $b_{b}$ by the Chinese remainder theorem. Each representative lies in an interval of length $q$, so the number of possible pairs is at most
\begin{equation*}
\left(\frac{q}{d} + 1\right)^{2} \le \frac{4x^{2}}{d^{2}}.
\end{equation*}
This also bounds the number of pairs of distinct labels satisfying the conditions for that vector. There are $q(q - 1) \gg x^{2}$ equally likely ordered pairs of distinct labels in total. Summing the counts over all offset vectors therefore gives
\begin{equation*}
\mathbb{P}_{*}(d \mid G)
\le U^{2j}\frac{4x^{2}/d^{2}}{q(q - 1)}
\ll \frac{U^{2j}}{d^{2}}.
\end{equation*}
This is stronger than \eqref{eq:intermediate-divisor-count}. Any overcounting in the sum over offset vectors is harmless for an upper bound.

For the second average, a member of $E_{q,n} \setminus \{n\}$ has the form $n + hq$ with $0 < |h| < U$. Thus $d \mid G_{n}(q, r)$ implies the existence of integer offsets satisfying
\begin{equation*}
0 < |h_{i}|, |\ell_{i}| < U, \qquad
p_{i} \mid n + h_{i}q, \qquad
p_{i} \mid n + \ell_{i}r
\qquad (1 \le i \le j).
\end{equation*}
There are fewer than $2U$ choices for each offset, and hence at most $(2U)^{2j} = (4U^{2})^{j}$ possible offset vectors.

If these divisibility conditions hold, then $p_{i} \nmid n$ by pairwise coprimality within each group. Also $h_{i}$ and $\ell_{i}$ are invertible modulo $p_{i}$, since their absolute values are positive and less than $U < w < p_{i}$. For a fixed offset vector, the conditions therefore give
\begin{equation*}
q \equiv -nh_{i}^{-1} \bmod p_{i}, \qquad
r \equiv -n\ell_{i}^{-1} \bmod p_{i}
\qquad (1 \le i \le j).
\end{equation*}
The Chinese remainder theorem determines one class modulo $d$ for each of $q$ and $r$. Ignoring the requirements that they be distinct primes in $\mathcal{R}$, and counting integers in $(0, x]$ instead, gives at most
\begin{equation*}
\left(\frac{x}{d} + 1\right)^{2} \le \frac{4x^{2}}{d^{2}}
\end{equation*}
possible pairs for each vector.

Each admissible prime pair has probability $O((\log x)^{2}/x^{2})$ by \eqref{eq:intermediate-prime-pair-normalizer}. Multiplying the number of offset vectors by the bound for the number of pairs per vector and the bound for the probability of each pair gives
\begin{equation*}
\mathbb{P}_{*}(d \mid G)
\ll (4U^{2})^{j}\frac{x^{2}}{d^{2}}\frac{(\log x)^{2}}{x^{2}}
= (\log x)^{2}\frac{(4U^{2})^{j}}{d^{2}},
\end{equation*}
as required.
\end{proof}

We now use the common-divisor estimate to show that the mean contribution of $G^{\tau}$ is close to $1$.

\begin{lemma}[Mean common-factor contribution] \label{lem:intermediate-gcd-average} 
In either auxiliary probability space, 
\begin{equation} 
\label{eq:intermediate-gcd-moment} 
\mathbb{E}_{*}G^{\tau} = 1 + O\left((\log x)^{-90}\right). 
\end{equation} 
The bound is uniform in the fixed reserve prime $q$ in the first case and in the fixed candidate $n$ in the second. 
\end{lemma}

\begin{proof}
The divisor estimate is available only for $d \le x$, so we first bound the contribution from $G > x$. Every prime factor of $G$ is less than or equal to $v < x$. By \eqref{eq:intermediate-divisor-count} and the union bound, the probability that $G$ has a prime factor greater than $x^{1/3}$ is at most
\begin{equation*}
O\left((\log x)^{2}U^{2}\sum_{p \, > \, x^{1/3}}\frac{1}{p^{2}}\right) \ll (\log x)^{2}U^{2}x^{-1/3}.
\end{equation*}
If $G > x$ and all its prime factors are less than or equal to $x^{1/3}$, multiply distinct prime factors until their product first exceeds $x^{1/3}$. The resulting divisor $d$ satisfies $x^{1/3} < d \le x^{2/3}$. Applying \eqref{eq:intermediate-divisor-count} to these possible divisors gives
\begin{equation*}
\mathbb{P}_{*}(G > x) \ll (\log x)^{2}U^{2}x^{-1/3} + (\log x)^{2}\sum_{\substack{x^{1/3} \, < \, d \, \le \, x^{2/3} \\ d \text{ squarefree and } w\text{-rough}}}\frac{(4U^{2})^{\omega(d)}}{d^{2}}.
\end{equation*}
For $d > x^{1/3}$, we have $d^{-2} \le x^{-1/6}d^{-3/2}$. Extending the sum to all squarefree $w$-rough positive integers therefore bounds its contribution by
\begin{equation*}
(\log x)^{2}x^{-1/6}\prod_{p \, > \, w}\left(1 + \frac{4U^{2}}{p^{3/2}}\right).
\end{equation*}
The product is $1 + o(1)$, because its logarithm is $O(U^{2}/\sqrt{w}) = o(1)$. Since $U \ll \log x$, we obtain
\begin{equation}
\label{eq:large-common-factor-probability-bound}
\mathbb{P}_{*}(G > x) \le x^{-1/6 + o(1)}.
\end{equation}

In either average, $G \le H^{U}$. Also $\tau\log H = O(\log_{3} x)$, so \eqref{eq:intermediate-group-size-bound} gives
\begin{equation*}
G^{\tau} \le \exp(\tau U\log H) \le \exp\left(O\left(\frac{\log x}{\log_{2} x}\right)\right) = x^{o(1)}.
\end{equation*}
Multiplying this bound by \eqref{eq:large-common-factor-probability-bound} yields
\begin{equation}
\label{eq:intermediate-large-gcd-moment}
\mathbb{E}_{*}\left(G^{\tau}\mathbf{1}_{\{G > x\}}\right) \ll x^{-1/8}.
\end{equation}
Here $\mathbf{1}_{\{G > x\}}$ is the indicator of the event $G > x$.

We now bound the contribution from $G \le x$. Since $G$ is squarefree, expanding
\begin{equation*}
G^{\tau} = \prod_{p \, \mid \, G}\left(1 + (p^{\tau} - 1)\right)
\end{equation*}
gives
\begin{equation*}
G^{\tau} - 1 = \sum_{\substack{d \, \mid \, G \\ d \, > \, 1}}\ \prod_{p \, \mid \, d}(p^{\tau} - 1).
\end{equation*}
On the event $G \le x$, every divisor in this sum is less than or equal to $x$. Each is also squarefree and $w$-rough. Interchanging the finite sum with expectation therefore gives
\begin{equation*}
\mathbb{E}_{*}\left((G^{\tau} - 1)\mathbf{1}_{\{G \le x\}}\right) = \sum_{\substack{1 \, < \, d \, \le \, x \\ d \text{ squarefree and } w\text{-rough}}}
\mathbb{P}_{*}(d \mid G,\ G \le x)\prod_{p \, \mid \, d}(p^{\tau} - 1).
\end{equation*}
All coefficients are nonnegative, since $\tau > 0$. We may thus replace $\mathbb{P}_{*}(d \mid G,\ G \le x)$ by the larger probability $\mathbb{P}_{*}(d \mid G)$ and apply \eqref{eq:intermediate-divisor-count} to obtain
\begin{equation*}
\mathbb{E}_{*}\left((G^{\tau} - 1)\mathbf{1}_{\{G \le x\}}\right) \ll (\log x)^{2}
\sum_{\substack{1 \, < \, d \, \le \, x \\ d \text{ squarefree and } w\text{-rough}}}
\frac{(4U^{2})^{\omega(d)}}{d^{2}}\prod_{p \, \mid \, d}(p^{\tau} - 1).
\end{equation*}

Squarefreeness allows us to write each summand as
\begin{equation*}
\frac{(4U^{2})^{\omega(d)}}{d^{2}}\prod_{p \, \mid \, d}(p^{\tau} - 1)
= \prod_{p \, \mid \, d}\frac{4U^{2}(p^{\tau} - 1)}{p^{2}}.
\end{equation*}
Removing the restriction $d \le x$ adds only nonnegative terms. The resulting sum over all squarefree $w$-rough integers greater than $1$ is an Euler product with its constant term removed:
\begin{equation*}
\sum_{\substack{d \, > \, 1 \\ d \text{ squarefree and } w\text{-rough}}}
\prod_{p \, \mid \, d}\frac{4U^{2}(p^{\tau} - 1)}{p^{2}}
=
\prod_{p \, > \, w}\left(1 + \frac{4U^{2}(p^{\tau} - 1)}{p^{2}}\right) - 1.
\end{equation*}
The divisor estimate was used only for $d \le x$; extending the numerical sum afterward does not require that estimate for larger $d$.

To bound this product, use $\tau = o(1)$ and $w^{\tau} = 1 + o(1)$ to obtain
\begin{equation*}
\sum_{p \, > \, w}\frac{p^{\tau} - 1}{p^{2}}
\le \sum_{m \, > \, w}m^{-2 + \tau}
\ll w^{-1 + \tau}
\ll \frac{1}{w}.
\end{equation*}
In particular, the product converges, and its logarithm is $O(U^{2}/w) = o(1)$. Its value minus $1$ is therefore also $O(U^{2}/w)$. Consequently,
\begin{equation}
\label{eq:intermediate-small-gcd-moment}
\mathbb{E}_{*}\left((G^{\tau} - 1)\mathbf{1}_{\{G \le x\}}\right)
\ll \frac{(\log x)^{2}U^{2}}{w}
\ll (\log x)^{-90}.
\end{equation}

Finally, $G^{\tau} \ge 1$, so
\begin{equation*}
\begin{split}
0 \le \mathbb{E}_{*}G^{\tau} - 1
&= \mathbb{E}_{*}\left((G^{\tau} - 1)\mathbf{1}_{\{G \le x\}}\right)
 + \mathbb{E}_{*}\left((G^{\tau} - 1)\mathbf{1}_{\{G > x\}}\right) \\
&\le \mathbb{E}_{*}\left((G^{\tau} - 1)\mathbf{1}_{\{G \le x\}}\right)
 + \mathbb{E}_{*}\left(G^{\tau}\mathbf{1}_{\{G > x\}}\right).
\end{split}
\end{equation*}
The bounds \eqref{eq:intermediate-small-gcd-moment} and \eqref{eq:intermediate-large-gcd-moment} give $O((\log x)^{-90} + x^{-1/8}) = O((\log x)^{-90})$, proving \eqref{eq:intermediate-gcd-moment}.
\end{proof}

Together with \eqref{eq:joint-correlation-error-bound}, Lemma~\ref{lem:intermediate-gcd-average} shows that the ratio in \eqref{eq:intermediate-unconditioned-ratio} has mean $1 + O((\log x)^{-90})$ over distinct labels for each fixed reserve prime. The conditional ratio in \eqref{eq:intermediate-conditioned-ratio} has the same mean under the weighted average over distinct reserve primes for each fixed candidate. These are the two estimates needed to control the selection weights: one for the total weight at a reserve prime, and one for the total weight available to cover a surviving candidate.

\section{Covering composite survivors}
\label{sec:intermediate-composite-covering}

We now choose the classes modulo the primes in $\mathcal{R}$. Following \cite[Sections~4--5]{GPT2026}, we give a group weight zero unless all its members survive the preliminary sieve, and otherwise give it the reciprocal of its survival probability. The estimates of the preceding section will show that these weights can be used to cover almost all the composite survivors in expectation. The two quantities to control are the total weight at a reserve prime and the total weight available to cover a surviving candidate. The first must be small enough to define a probability distribution; the second must be large enough to make noncoverage unlikely.

Throughout this section, $t$ remains in the range \eqref{eq:intermediate-tilt-range}, and all estimates are uniform in this range. Recall that $E_{q,a}$ consists of the members of $\mathcal{C}$ congruent to $a$ modulo $q$, and that $\Pi(E_{q,a})$ is the probability that every member of this group survives. We will also use
\begin{equation*}
q_{n} = \mathbb{P}(n \in \mathcal{N}), \qquad q_{n} \le Q = AT^{-\tau}, \qquad M = \sum_{q \, \in \, \mathcal{R}}\frac{1}{q}.
\end{equation*}
The bound for $q_{n}$ holds for every $n \in \mathcal{C}$ by \eqref{eq:intermediate-probability-comparison}. Until the residue classes in $\mathcal{R}$ are chosen, all probabilities and expectations refer to the preliminary sieve. In particular, the auxiliary probability spaces used to average common factors in the preceding section have served their purpose; their estimates will enter here when we sum over pairs of groups.

\subsection{Inverse-probability weights}
\label{subsec:intermediate-inverse-probability-weights}

For $q \in \mathcal{R}$ and $0 \le a < q$, put
\begin{equation}
\label{eq:intermediate-selection-weights}
X_{q,a} := \frac{\mathbf{1}_{\{E_{q,a} \subseteq \mathcal{N}\}}}{\Pi(E_{q,a})}.
\end{equation}
Here the numerator is $1$ when the entire group survives and $0$ otherwise. The denominator is positive by Lemma~\ref{lem:intermediate-joint-law}. Thus a group that survives with small probability receives a large weight when it does survive, and the two factors cancel in expectation:
\begin{equation*}
\mathbb{E}X_{q,a} = \frac{\mathbb{P}(E_{q,a} \subseteq \mathcal{N})}{\Pi(E_{q,a})} = 1.
\end{equation*}
This also holds for an empty group, whose survival probability and weight are both $1$. Empty groups will cover no candidates, but retaining their labels makes the averages over all residue classes exact.

To measure the total weight at a reserve prime, define
\begin{equation*}
Z_{q} := \frac{1}{q}\sum_{a \, = \, 0}^{q - 1}X_{q,a}, \qquad \mathbb{E}Z_{q} = 1.
\end{equation*}
The factor $1/q$ averages over the $q$ available classes. We will use the individual quantities $X_{q,a}/q$ to assign selection probabilities, so $Z_{q}$ is their total mass.

For $n \in \mathcal{C}$, write $X_{q,n}$ for the weight of its group $E_{q,n}$ and put
\begin{equation}
\label{eq:intermediate-total-weight}
W_{n} := \sum_{q \, \in \, \mathcal{R}}\frac{X_{q,n}}{q}.
\end{equation}
This sums the weights available to cover $n$ over all reserve primes. If $n$ does not survive, none of its groups survives in full, so $W_{n} = 0$. We are therefore interested in its size conditional on $n \in \mathcal{N}$. Since $n \in E_{q,n}$,
\begin{equation*}
\mathbb{E}(X_{q,n} \mid n \in \mathcal{N}) = \frac{\mathbb{P}(E_{q,n} \subseteq \mathcal{N} \mid n \in \mathcal{N})}{\Pi(E_{q,n})} = \frac{1}{q_{n}}.
\end{equation*}
Indeed, the conditional probability in the numerator is $\Pi(E_{q,n})/q_{n}$. Summing over $q$ gives
\begin{equation}
\label{eq:composite-weight-conditional-mean}
\mathbb{E}(W_{n} \mid n \in \mathcal{N}) = \frac{M}{q_{n}}, \qquad \mathbb{E}\left(\left.\frac{q_{n}W_{n}}{M}\,\right|\,n \in \mathcal{N}\right) = 1.
\end{equation}
The reciprocal weighting has therefore given us two exact means: $1$ for the total mass at a prime, and $M/q_{n}$ for the weight available to a surviving candidate. We next show that both quantities are usually close to their means.

\subsection{Second-moment estimates}
\label{subsec:intermediate-second-moments}

An expectation alone does not provide this assurance: a nonnegative random variable may usually be zero and occasionally take a large value. The variance measures its mean squared deviation from its expectation. For a random variable $Y$ with $\mathbb{E}Y = 1$,
\begin{equation*}
\operatorname{Var}(Y) = \mathbb{E}(Y - 1)^{2} = \mathbb{E}Y^{2} - 1.
\end{equation*}
Thus it suffices to bound the second moment close to $1$. Expanding the square of a sum brings in products of two weights, whose expectations are precisely the correlation ratios estimated in the preceding section.

\begin{lemma}[Two second moments]
\label{lem:intermediate-second-moments}
Uniformly for $q \in \mathcal{R}$ and $n \in \mathcal{C}$,
\begin{equation}
\label{eq:intermediate-second-moment-bounds}
\begin{split}
\mathbb{E}(Z_{q} - 1)^{2} &\ll (\log x)^{-90}, \\
\mathbb{E}\left(\left.\left(\frac{q_{n}W_{n}}{M} - 1\right)^{2}\,\right|\,n \in \mathcal{N}\right) &\ll (\log x)^{-90}.
\end{split}
\end{equation}
\end{lemma}

\begin{proof}
For distinct labels $a$ and $b$, multiplying the indicators in \eqref{eq:intermediate-selection-weights} gives the indicator that both groups survive. Consequently,
\begin{equation*}
\mathbb{E}(X_{q,a}X_{q,b}) = \frac{\Pi(E_{q,a} \cup E_{q,b})}{\Pi(E_{q,a})\Pi(E_{q,b})}.
\end{equation*}
By \eqref{eq:joint-correlation-error-bound}, \eqref{eq:intermediate-unconditioned-ratio} and Lemma~\ref{lem:intermediate-gcd-average}, the average of these terms over the $q(q - 1)$ ordered pairs of distinct labels is $1 + O((\log x)^{-90})$. Their contribution to $\mathbb{E}Z_{q}^{2}$ is therefore
\begin{equation*}
\frac{1}{q^{2}}\sum_{\substack{0 \, \le \, a, b \, < \, q \\ a \, \ne \, b}}\mathbb{E}(X_{q,a}X_{q,b}) = \left(1 - \frac{1}{q}\right)\left(1 + O\left((\log x)^{-90}\right)\right).
\end{equation*}
The terms with $a = b$ require a separate bound. Squaring an indicator leaves it unchanged, so
\begin{equation*}
\frac{1}{q^{2}}\sum_{a \, = \, 0}^{q - 1}\mathbb{E}X_{q,a}^{2} = \frac{1}{q^{2}}\sum_{a \, = \, 0}^{q - 1}\Pi(E_{q,a})^{-1} \le x^{-1 + o(1)}
\end{equation*}
by \eqref{eq:intermediate-inverse-group-probability} and $q > x/2$. This is $o((\log x)^{-90})$. Adding the two contributions gives
\begin{equation*}
\mathbb{E}Z_{q}^{2} \le 1 + O((\log x)^{-90}).
\end{equation*}
Since $\mathbb{E}Z_{q} = 1$, the first assertion follows.

For the second assertion, we condition throughout on $n \in \mathcal{N}$. Every group under consideration contains $n$, so the same indicator calculations give
\begin{equation}
\label{eq:intermediate-conditional-weight-identities}
\begin{split}
q_{n}^{2}\mathbb{E}(X_{q,n}^{2} \mid n \in \mathcal{N}) &= \frac{q_{n}}{\Pi(E_{q,n})}, \\
q_{n}^{2}\mathbb{E}(X_{q,n}X_{r,n} \mid n \in \mathcal{N}) &= \frac{q_{n}\Pi(E_{q,n} \cup E_{r,n})}{\Pi(E_{q,n})\Pi(E_{r,n})} \qquad (q \ne r).
\end{split}
\end{equation}
For example, the second identity follows by dividing the conditional joint-survival probability $\Pi(E_{q,n} \cup E_{r,n})/q_{n}$ by the two denominators in the weights, and then multiplying by $q_{n}^{2}$.

Using \eqref{eq:intermediate-total-weight}, we expand the conditional second moment as
\begin{equation*}
\mathbb{E}\left(\left.\left(\frac{q_{n}W_{n}}{M}\right)^{2}\,\right|\,n \in \mathcal{N}\right) = \frac{1}{M^{2}}\sum_{q, r \, \in \, \mathcal{R}}\frac{q_{n}^{2}\mathbb{E}(X_{q,n}X_{r,n} \mid n \in \mathcal{N})}{qr}.
\end{equation*}
By \eqref{eq:intermediate-conditional-weight-identities}, for $q \ne r$ the numerator inside the sum is the conditional correlation ratio in \eqref{eq:intermediate-conditioned-ratio}. Its average with weights proportional to $1/(qr)$ is $1 + O((\log x)^{-90})$, by the second average in Lemma~\ref{lem:intermediate-gcd-average} and \eqref{eq:joint-correlation-error-bound}. Hence the off-diagonal contribution is
\begin{equation*}
\left(\frac{1}{M^{2}}\sum_{\substack{q, r \, \in \, \mathcal{R} \\ q \, \ne \, r}}\frac{1}{qr}\right)\left(1 + O\left((\log x)^{-90}\right)\right) \le 1 + O\left((\log x)^{-90}\right).
\end{equation*}
The factor in parentheses accounts for the normalization of that auxiliary average; it is less than or equal to $1$ because $M^{2}$ includes the diagonal terms as well.

For the diagonal terms, $q_{n} \le 1$ and \eqref{eq:intermediate-inverse-group-probability} give $q_{n}/\Pi(E_{q,n}) \le x^{o(1)}$. Also $q > x/2$, so
\begin{equation*}
\frac{1}{M^{2}}\sum_{q \, \in \, \mathcal{R}}\frac{1}{q^{2}} \le \frac{2}{xM} \ll \frac{\log x}{x},
\end{equation*}
where we used \eqref{eq:intermediate-reserve-reciprocals}. The diagonal contribution is therefore
\begin{equation*}
\frac{1}{M^{2}}\sum_{q \, \in \, \mathcal{R}}\frac{1}{q^{2}}\frac{q_{n}}{\Pi(E_{q,n})} \le x^{-1 + o(1)} = o\left((\log x)^{-90}\right).
\end{equation*}
The conditional second moment is thus less than or equal to $1 + O((\log x)^{-90})$. Its conditional mean is $1$ by \eqref{eq:composite-weight-conditional-mean}; subtracting the square of that mean proves the second assertion.
\end{proof}

For later use, recall how a squared-deviation bound controls exceptional events. If $\mathbb{E}Y = 1$, then for any $a > 0$,
\begin{equation*}
\mathbb{P}(|Y - 1| \ge a) \le \frac{\mathbb{E}(Y - 1)^{2}}{a^{2}},
\end{equation*}
because $(Y - 1)^{2} \ge a^{2}$ on the indicated event. This is Chebyshev's inequality, and the same argument applies under a conditional probability distribution. The second estimate in \eqref{eq:intermediate-second-moment-bounds} will therefore show that a surviving candidate rarely has much less than its expected weight $M/q_{n}$.

\subsection{Choosing the residue classes}
\label{subsec:intermediate-truncated-selection}

Fix the preliminary choices for the moment, so that all the weights are now known numbers. For each $q \in \mathcal{R}$ with $Z_{q} \le 2$, assign probability $X_{q,a}/(2q)$ to each label $a$, and add all the remaining probability to the zero label. This is possible because
\begin{equation*}
\sum_{a \, = \, 0}^{q - 1}\frac{X_{q,a}}{2q} = \frac{Z_{q}}{2} \le 1.
\end{equation*}
If $Z_{q} > 2$, choose the zero label with probability $1$. Choose independently for different $q$, conditional on the preliminary choices, and take $a_{q}$ to be the chosen class. The constants $2$ and $1/2$ ensure that the assigned probabilities have total mass less than or equal to $1$ whenever we use the group weights.

The zero class contains no candidates: every prime factor of a member of $\mathcal{C}$ is less than or equal to $v$, whereas $q > v$. Thus the extra probability assigned to zero contributes nothing to covering $\mathcal{C}$.

We write $\mathbb{P}(\,\cdot \mid \mathcal{N})$ for conditional probability given the entire survivor set: for each possible value $S$ of $\mathcal{N}$, this means conditioning on the event $\mathcal{N} = S$. Knowing this set determines all the weights $X_{q,a}$ and $Z_{q}$, and hence the distributions for the second-stage choices. For each candidate $n$, the conditional probability that the choice at $q$ covers it is therefore
\begin{equation}
\label{eq:composite-candidate-selection-probability}
\theta_{q,n} := \mathbb{P}(a_{q} \equiv n \bmod q \mid \mathcal{N}) = \frac{X_{q,n}}{2q}\mathbf{1}_{\{Z_{q} \le 2\}}.
\end{equation}
A nonempty group can receive positive selection weight only when all its members survive, in which case choosing its class covers them all.

Discarding the group weights at primes with $Z_{q} > 2$ might reduce the weight available to a candidate. Write
\begin{equation*}
D_{n} := \sum_{q \, \in \, \mathcal{R}}\frac{X_{q,n}}{q}\mathbf{1}_{\{Z_{q} > 2\}}
\end{equation*}
for this discarded weight. Both $W_{n}$ and $D_{n}$ vanish if $n \notin \mathcal{N}$, and \eqref{eq:composite-candidate-selection-probability} gives
\begin{equation}
\label{eq:usable-composite-covering-weight}
\sum_{q \, \in \, \mathcal{R}}\theta_{q,n} = \frac{W_{n} - D_{n}}{2}.
\end{equation}
The sum on the left is the conditional expected number of reserve primes that cover $n$; it can be greater than $1$.

We need to bound how much weight is discarded, not just the probability of discarding. The following estimate controls the total loss over all candidates.

\begin{lemma}[Discarded weight]
\label{lem:intermediate-discarded-weight}
With expectation over the preliminary choices,
\begin{equation}
\label{eq:intermediate-discarded-weight-bound}
\sum_{n \, \in \, \mathcal{C}}\mathbb{E}D_{n} \ll \frac{U\#\mathcal{R}}{(\log x)^{90}}.
\end{equation}
\end{lemma}

\begin{proof}
For $z \ge 0$,
\begin{equation*}
z\mathbf{1}_{\{z > 2\}} \le 2(z - 1)^{2}.
\end{equation*}
For $z \le 2$ this is immediate, and for $z > 2$ it follows from $2(z - 1)^{2} - z = (2z - 1)(z - 2) > 0$. The first estimate in \eqref{eq:intermediate-second-moment-bounds} therefore gives
\begin{equation*}
\mathbb{E}\left(Z_{q}\mathbf{1}_{\{Z_{q} > 2\}}\right) \le 2\mathbb{E}(Z_{q} - 1)^{2} \ll (\log x)^{-90}.
\end{equation*}
For a fixed $q$, the same label weight $X_{q,a}$ occurs in $D_{n}$ once for every member $n$ of $E_{q,a}$. Each group has size less than or equal to $U$ by Lemma~\ref{lem:intermediate-group-geometry}. Interchanging the sums and grouping candidates by their classes modulo $q$ gives
\begin{equation*}
\sum_{n \, \in \, \mathcal{C}}D_{n} = \sum_{q \, \in \, \mathcal{R}}\frac{\mathbf{1}_{\{Z_{q} > 2\}}}{q}\sum_{a \, = \, 0}^{q - 1}\#E_{q,a}\,X_{q,a} 
\le U\sum_{q \, \in \, \mathcal{R}}Z_{q}\mathbf{1}_{\{Z_{q} > 2\}}.
\end{equation*}
Taking expectations and using the preceding estimate proves \eqref{eq:intermediate-discarded-weight-bound}.
\end{proof}

\subsection{The remaining composites}
\label{subsec:intermediate-uncovered-composites}

We can now estimate the number of composite candidates missed by both stages. There are three contributions: survivors with too little initial weight, survivors losing too much weight through discarding, and survivors with enough usable weight that nevertheless escape every second-stage choice. The first two will be bounded by the moment estimates; the third will be bounded using the conditional independence of the choices at the reserve primes.

\begin{proposition}[Composite remainder]
\label{prop:intermediate-composite-covering}
Let
\begin{equation*}
R_{\mathcal{C}} := \#\left(\mathcal{C} \cap \mathcal{N} \setminus \bigcup_{q \, \in \, \mathcal{R}}(a_{q} + q\mathbb{Z})\right).
\end{equation*}
Then, uniformly for $t$ in \eqref{eq:intermediate-tilt-range},
\begin{equation}
\label{eq:intermediate-composite-remainder}
\mathbb{E}R_{\mathcal{C}} \ll \frac{HQ}{(\log x)^{90}} + HQ\exp\left(-\frac{M}{8Q}\right).
\end{equation}
The expectation includes both the preliminary choices and the subsequent choices at primes in $\mathcal{R}$.
\end{proposition}

\begin{proof}
We first count survivors for which $W_{n} < M/(2q_{n})$. Conditional on $n \in \mathcal{N}$, this implies that $q_{n}W_{n}/M$ differs from its mean $1$ by more than $1/2$. Chebyshev's inequality and \eqref{eq:intermediate-second-moment-bounds} give
\begin{equation*}
\mathbb{P}\left(\left.W_{n} < \frac{M}{2q_{n}}\,\right|\,n \in \mathcal{N}\right) \ll (\log x)^{-90}.
\end{equation*}
Multiplying by the survival probability $q_{n}$ and summing over candidates therefore gives
\begin{equation*}
\sum_{n \, \in \, \mathcal{C}}\mathbb{P}\left(n \in \mathcal{N},\ W_{n} < \frac{M}{2q_{n}}\right) \ll \frac{1}{(\log x)^{90}}\sum_{n \, \in \, \mathcal{C}}q_{n} \le \frac{HQ}{(\log x)^{90}}.
\end{equation*}
We used $\#\mathcal{C} \le H$ and $q_{n} \le Q$ (see \eqref{eq:intermediate-probability-comparison}). These exceptional survivors can all be counted as uncovered, regardless of the second-stage choices.

Next consider survivors for which $D_{n} > M/(4q_{n})$. For any nonnegative random variable $Y$ and $a > 0$, the inequality $Y \ge a\mathbf{1}_{\{Y > a\}}$ gives Markov's inequality $\mathbb{P}(Y > a) \le \mathbb{E}Y/a$. Apply it to $D_{n}$, noting that $D_{n} = 0$ when $n$ does not survive. By \eqref{eq:intermediate-discarded-weight-bound},
\begin{equation*}
\sum_{n \, \in \, \mathcal{C}}\mathbb{P}\left(n \in \mathcal{N},\ D_{n} > \frac{M}{4q_{n}}\right) \le \frac{4}{M}\sum_{n \, \in \, \mathcal{C}}q_{n}\mathbb{E}D_{n} \\
\le \frac{4Q}{M}\sum_{n \, \in \, \mathcal{C}}\mathbb{E}D_{n} \ll \frac{QU\#\mathcal{R}}{M(\log x)^{90}}.
\end{equation*}
Since $q \le x$ for every $q \in \mathcal{R}$, we have $M \ge \#\mathcal{R}/x$. Also the definition $U = \lceil 2H/x\rceil + 2$ and $H > x$ give $Ux \le 2H + 3x < 5H$. Hence
\begin{equation*}
\frac{QU\#\mathcal{R}}{M} \le QUx \ll QH,
\end{equation*}
so the expected number of these exceptional survivors is also $O(HQ/(\log x)^{90})$.

It remains to consider a preliminary outcome and a surviving $n$ satisfying both
\begin{equation*}
W_{n} \ge \frac{M}{2q_{n}}, \qquad D_{n} \le \frac{M}{4q_{n}}.
\end{equation*}
For this outcome, \eqref{eq:usable-composite-covering-weight} gives
\begin{equation*}
\sum_{q \, \in \, \mathcal{R}}\theta_{q,n} = \frac{W_{n} - D_{n}}{2} \ge \frac{M}{8q_{n}} \ge \frac{M}{8Q}.
\end{equation*}
Conditional on $\mathcal{N}$, the choices at different reserve primes are independent. Thus the probability that all of them miss this $n$ is
\begin{equation*}
\prod_{q \, \in \, \mathcal{R}}(1 - \theta_{q,n}) \le \exp\left(-\sum_{q \, \in \, \mathcal{R}}\theta_{q,n}\right) \le \exp\left(-\frac{M}{8Q}\right),
\end{equation*}
where we used $1 - u \le e^{-u}$ for $u \ge 0$. This is a bound for each preliminary outcome satisfying the two inequalities above. Averaging over those outcomes, whose total probability is less than or equal to $q_{n}$, bounds the probability that $n$ survives, satisfies both inequalities and remains uncovered by $q_{n}\exp(-M/(8Q))$. Summing over candidates gives
\begin{equation*}
\sum_{n \, \in \, \mathcal{C}}q_{n}\exp\left(-\frac{M}{8Q}\right) \le HQ\exp\left(-\frac{M}{8Q}\right).
\end{equation*}
Adding this contribution to the two exceptional contributions proves \eqref{eq:intermediate-composite-remainder}.
\end{proof}

The factor $HQ$ bounds the expected number of preliminary survivors in $\mathcal{C}$. The first term in \eqref{eq:intermediate-composite-remainder} accounts for inadequate or discarded weight, and the second for survivors missed despite having sufficient usable weight. Independence was needed only for the second-stage choices at different reserve primes, conditional on the preliminary sieve; no independence between the coverage events for different integers was assumed. In the final section we choose $t$ so that both terms are negligible compared with the number of mopping primes, and count the survivors outside $\mathcal{C}$ as well.

\section{Completing the covering}
\label{sec:intermediate-completion}

We now choose the tilt so that the expected number of composite candidates left uncovered is negligible compared with the number of mopping primes. We then count the survivors outside the candidate set. The prime survivors give the main contribution, and choosing the constant $c$ sufficiently small makes this contribution less than the number of primes in $\mathcal{R}'$.

\subsection{Choosing the tilt}
\label{subsec:intermediate-final-parameters}

Recall that $R_{\mathcal{C}}$ counts the members of $\mathcal{C}$ left uncovered after both stages of random choices. Proposition~\ref{prop:intermediate-composite-covering} gives
\begin{equation*}
\mathbb{E}R_{\mathcal{C}} \ll \frac{HQ}{(\log x)^{90}} + HQ\exp\left(-\frac{M}{8Q}\right).
\end{equation*}
The first term is already negligible throughout the working range of $t$. Indeed, $Q = AT^{-\tau} \le 1$ and $H \ll x\log x$, so
\begin{equation*}
\frac{HQ}{(\log x)^{90}} \ll \frac{x}{(\log x)^{89}} = o\left(\frac{x}{\log x}\right).
\end{equation*}
We choose $t$ to make the second term negligible as well.

Equations~\eqref{eq:intermediate-reserve-reciprocals} and \eqref{eq:intermediate-common-probability-size} give
\begin{equation}
\label{eq:final-covering-intensity-size}
\frac{M}{Q} \asymp \frac{e^{t}}{t\log_{2} x}.
\end{equation}
Also, the definition of $H$ in \eqref{eq:target-interval-length} and the estimate for $Q$ give $HQ \asymp xe^{-t}$, with $c$ fixed and $t$ in \eqref{eq:intermediate-tilt-range}. Consequently,
\begin{equation}
\label{eq:final-composite-prefactor-size}
\log\left(\frac{HQ\log x}{x}\right) = \log_{2} x - t + O(1) = \log_{2} x + O(\log_{3} x).
\end{equation}
Thus it suffices to make $M/Q$ grow faster than $\log_{2} x$: the negative exponent in the remainder bound will then dominate the logarithm of the prefactor. By \eqref{eq:final-covering-intensity-size}, a sufficient condition is
\begin{equation*}
\frac{e^{t}}{t(\log_{2} x)^{2}} \to \infty.
\end{equation*}
We therefore choose
\begin{equation}
\label{eq:intermediate-final-tilt}
t := 2\log_{3} x + 2\log_{4} x, \qquad \tau := \frac{t}{\log x}.
\end{equation}
This lies in the range \eqref{eq:intermediate-tilt-range} for sufficiently large $x$, since $\log_{4} x = o(\log_{3} x)$. Moreover,
\begin{equation*}
e^{t} = (\log_{2} x)^{2}(\log_{3} x)^{2}, \qquad \frac{e^{t}}{t(\log_{2} x)^{2}} = \frac{(\log_{3} x)^{2}}{t} \sim \frac{\log_{3} x}{2}.
\end{equation*}
In particular, \eqref{eq:final-covering-intensity-size} now gives $M/Q \asymp (\log_{2} x)\log_{3} x$. Using \eqref{eq:final-composite-prefactor-size}, we obtain
\begin{equation*}
\log\left(\frac{HQ\log x}{x}\exp\left(-\frac{M}{8Q}\right)\right) = \log_{2} x + O(\log_{3} x) - \frac{M}{8Q} \to -\infty.
\end{equation*}
Both terms in \eqref{eq:intermediate-composite-remainder} are therefore negligible on the required scale:
\begin{equation}
\label{eq:final-composite-remainder-small}
\mathbb{E}R_{\mathcal{C}} = o\left(\frac{x}{\log x}\right).
\end{equation}

\subsection{Counting the remaining integers}
\label{subsec:intermediate-remaining-integers}

Every preliminary survivor is $w$-rough. Thus a survivor greater than $T$ is either a prime, a member of $\mathcal{C}$, or a nonsquarefree $w$-rough composite. We have counted the uncovered members of $\mathcal{C}$. The initial interval contributes less than or equal to
\begin{equation*}
T = \left\lfloor\frac{x}{(\log x)^{2}}\right\rfloor = o\left(\frac{x}{\log x}\right)
\end{equation*}
uncovered integers, regardless of the residue choices. It remains to count the primes and the nonsquarefree exceptions.

\begin{lemma}[Survivors outside the candidate set]
\label{lem:intermediate-final-counts}
With the tilt \eqref{eq:intermediate-final-tilt},
\begin{equation}
\label{eq:intermediate-prime-expectation}
\begin{split}
\mathbb{E}\,\#\{\ell \in \mathcal{N} : \ell > T,\ \ell \text{ is prime}\} &= (1 + o(1))\frac{AH}{\log H} \\
&= (200e^{\gamma}c + o(1))\frac{x}{\log x}.
\end{split}
\end{equation}
The number of nonsquarefree $w$-rough integers less than or equal to $H$ is $o(x/\log x)$.
\end{lemma}

\begin{proof}
By Lemma~\ref{lem:intermediate-one-point-law}, a prime $\ell$ in $(T, v]$ survives with probability $A\ell^{-\tau}$, while a prime in $(v, H]$ survives with probability $A$. Linearity of expectation therefore gives the exact expression
\begin{equation*}
A\bigl(\pi(H) - \pi(v)\bigr) + A\sum_{T \, < \, \ell \, \le \, v}\ell^{-\tau},
\end{equation*}
where the sum runs over primes. Since $\ell^{-\tau} \le 1$, this differs from $A\pi(H)$ by $O(A\pi(v))$. The prime number theorem and the relations $H/x \to \infty$ and $\log H \sim \log x$ give
\begin{equation*}
\pi(v) = o\left(\frac{H}{\log H}\right), \qquad \pi(H) \sim \frac{H}{\log H}.
\end{equation*}
This proves the first equality in \eqref{eq:intermediate-prime-expectation}. Substituting \eqref{eq:target-interval-length} and \eqref{eq:intermediate-product-estimate}, and using $t \sim 2\log_{3} x$, gives the second. The subsequent choices at primes in $\mathcal{R}$ can only remove survivors, so the same expectation bounds the expected number of primes still requiring mopping up.

Every nonsquarefree $w$-rough integer is divisible by $p^{2}$ for some prime $p > w$. Their number is therefore less than or equal to
\begin{equation*}
\sum_{w \, < \, p \, \le \, \sqrt{H}}\left\lfloor\frac{H}{p^{2}}\right\rfloor \le H\sum_{p \, > \, w}\frac{1}{p^{2}} \ll \frac{H}{w} = o\left(\frac{x}{\log x}\right).
\end{equation*}
This is a bound for all such integers, so it also bounds those remaining after either stage of the sieve.
\end{proof}

\subsection{Mopping up and the main theorem}
\label{subsec:intermediate-deterministic-covering}

We now combine these counts and apply the expectation criterion from Subsection~\ref{subsec:survivors-to-complete-covering} to complete the covering with the primes in $\mathcal{R}'$.

\begin{proof}[Proof of Theorem~\ref{thm:tilted-covering-bound}]
Fix $c > 0$ sufficiently small that $200e^{\gamma}c < 1/4$, and use the tilt \eqref{eq:intermediate-final-tilt}. Recall that
\begin{equation*}
\mathcal{M} = \mathcal{N} \setminus \bigcup_{q \, \in \, \mathcal{R}}(a_{q} + q\mathbb{Z})
\end{equation*}
is the set of integers still uncovered after both stages of random choices. The initial interval and the nonsquarefree exceptions together contribute $o(x/\log x)$ integers. By \eqref{eq:final-composite-remainder-small}, the expected contribution from the remaining composite candidates is also $o(x/\log x)$. Adding the prime expectation from \eqref{eq:intermediate-prime-expectation} gives
\begin{equation*}
\mathbb{E}\,\#\mathcal{M} \le (200e^{\gamma}c + o(1))\frac{x}{\log x} < \#\mathcal{R}'
\end{equation*}
for all sufficiently large $x$, since \eqref{eq:intermediate-prime-budgets} gives $\#\mathcal{R}' \sim x/(4\log x)$. The expectation includes both stages of random choices.

We can now apply the expectation criterion from Subsection~\ref{subsec:survivors-to-complete-covering}. Some outcome therefore has $\#\mathcal{M} < \#\mathcal{R}'$. Fix such an outcome, assign a distinct prime $p \in \mathcal{R}'$ to each $n \in \mathcal{M}$, and choose $a_{p} \equiv n \bmod p$. Give arbitrary classes to any unused primes. These choices complete the covering:
\begin{equation*}
\{1, \ldots, H\} \subseteq \bigcup_{p \, \le \, x}(a_{p} + p\mathbb{Z}).
\end{equation*}
Thus $Y(x) \ge H$, and the definition \eqref{eq:target-interval-length} gives the asserted bound.
\end{proof}

\subsection{From coverings to prime gaps}
\label{subsec:tilted-covering-prime-gaps}

It remains to convert the covering into a gap between consecutive primes. We also check that the larger prime is less than or equal to $X$, as required by the definition of $G(X)$.

\begin{proof}[Proof of Corollary~\ref{cor:tilted-prime-gap-bound}]
For sufficiently large $X$, put $x := (\log X)/2$ and write
\begin{equation*}
P(x) := \prod_{p \, \le \, x}p.
\end{equation*}
The prime number theorem gives
\begin{equation}
\label{eq:final-primorial-size}
\log P(x) = \sum_{p \, \le \, x}\log p = (1 + o(1))x, \qquad P(x) = X^{1/2 + o(1)}.
\end{equation}
Take the residue classes just constructed. By the Chinese remainder theorem, there is an integer $m$ satisfying
\begin{equation*}
P(x) \le m < 2P(x), \qquad m \equiv -a_{p} \bmod p \quad (p \le x).
\end{equation*}
For each $1 \le n \le H$, the covering supplies a prime $p \le x$ with $n \equiv a_{p} \bmod p$. Hence $p \mid m + n$. Since $m + n > P(x) > x \ge p$, this divisor is proper. All the integers $m + 1, \ldots, m + H$ are therefore composite.

Let $p_{-}$ be the largest prime less than or equal to $m$, and $p_{+}$ the smallest prime greater than $m + H$. These primes are consecutive, and $p_{+} - p_{-} > H$. The prime number theorem guarantees a prime in $(u, 2u]$ for all sufficiently large $u$. Since $H = o(P(x))$, it follows from \eqref{eq:final-primorial-size} that
\begin{equation*}
p_{+} \le 2(m + H) < 6P(x) = X^{1/2 + o(1)} < X.
\end{equation*}
This gap is therefore counted by $G(X)$. Substituting $x = (\log X)/2$ into the definition of $H$ now gives
\begin{equation*}
G(X) \ge H \gg \frac{x\log x}{(\log_{2} x)\log_{3} x} \asymp \frac{(\log X)\log_{2} X}{(\log_{3} X)\log_{4} X},
\end{equation*}
as required.
\end{proof}

\clearpage

\appendix

\section{The classical Erd\H{o}s--Rankin construction}
\label{app:traditional-erdos-rankin}

We give a detailed proof of the classical Erd\H{o}s--Rankin bound, using an elementary smooth-number estimate and a greedy choice of residue classes. We first state the covering bound and deduce its consequence for prime gaps, then establish the estimate needed for the construction. We also explain how the parameter choices arise.

Recall that $Y(x)$ is the largest nonnegative integer $y$ for which one can choose a residue class $a_p \bmod p$ for each prime $p \le x$ so that these classes \emph{cover} $\{1,\ldots,y\}$: every integer in this interval belongs to at least one of the chosen classes. We call such a family of residue classes a \emph{covering} of the interval. Writing $p_n$ for the $n$th prime, we also put
\begin{equation*}
G(X) := \max_{p_{n + 1} \, \le \, X}(p_{n + 1} - p_n),
\end{equation*}
the largest gap between consecutive primes both at most $X$.

\subsection{The covering bound and its consequence for prime gaps}
\label{subsec:classical-covering-and-prime-gaps}

The construction gives the following lower bound for the length of an interval that can be covered by one residue class modulo each prime up to $x$.

\begin{theorem}
\label{thm:classical-erdos-rankin-bound}
For all sufficiently large $x$,
\begin{equation*}
Y(x) \gg \frac{x(\log x)(\log_{3} x)}{(\log_{2} x)^2}.
\end{equation*}
\end{theorem}

The Chinese remainder theorem converts such a covering into an interval of composite integers. Before constructing the covering, we make this passage to prime gaps explicit.

\begin{corollary}
\label{cor:classical-erdos-rankin-prime-gaps}
For all sufficiently large $X$,
\begin{equation*}
G(X) \gg \frac{(\log X)(\log_{2} X)(\log_{4} X)}{(\log_{3} X)^2}.
\end{equation*}
Moreover, for infinitely many $n$,
\begin{equation*}
p_{n + 1} - p_n \gg \frac{(\log p_n)(\log_{2} p_n)(\log_{4} p_n)}{(\log_{3} p_n)^2}.
\end{equation*}
\end{corollary}

We deduce the corollary from a general transfer principle, which also preserves the leading constant in more precise covering bounds.

\begin{proposition}[From coverings to prime gaps]
\label{prop:covering-to-prime-gaps}
Let $F$ be an eventually positive, nondecreasing function. Suppose that, for some $\alpha \ge 0$ and every fixed $a > 0$,
\begin{equation*}
\frac{F(at)}{F(t)} \to a^{\alpha} \qquad (t \to \infty).
\end{equation*}
If $Y(x) \ge F(x)$ for all sufficiently large $x$, then
\begin{equation*}
G(X) \ge (1 - o(1))F(\log X) \qquad (X \to \infty).
\end{equation*}
Moreover, there is a sequence of indices $n_j \to \infty$ such that
\begin{equation*}
p_{n_j + 1} - p_{n_j} \ge (1 - o(1))F(\log p_{n_j}) \qquad (j \to \infty).
\end{equation*}
For this latter conclusion, it suffices that $Y(x) \ge F(x)$ for arbitrarily large $x$.
\end{proposition}

The condition on $F$ is called \emph{regular variation of index $\alpha$}. For example, an eventually increasing function of the form $c t^{\alpha}\prod_{j = 1}^{k}(\log_j t)^{\beta_j}$, with fixed $c > 0$, $k$ and real exponents $\beta_j$, satisfies this condition.

\begin{proof}[Proof of Proposition \ref{prop:covering-to-prime-gaps}]
Write
\begin{equation*}
P(x) := \prod_{p \, \le \, x}p.
\end{equation*}
We first record the elementary relation between a covering and a prime gap. Put $y = Y(x)$, and choose residue classes $a_p \bmod p$, one for each prime $p \le x$, covering $\{1,\ldots,y\}$.

For each $p \le x$, choose $b_p \not\equiv a_p \bmod p$. The Chinese remainder theorem gives an integer $b \in \{1,\ldots,P(x)\}$ with $b \equiv b_p \bmod p$ for every $p \le x$. Since $b$ belongs to none of the covering classes, we have $y < b \le P(x)$.

Another application of the Chinese remainder theorem gives an integer $M$ satisfying
\begin{equation*}
P(x) \le M < 2P(x), \qquad M \equiv -a_p \bmod p \quad (p \le x).
\end{equation*}
For sufficiently large $x$, we have $M > x$, so every integer $M + 1,\ldots,M + y$ is composite. Let $p_n$ be the largest prime not exceeding $M$. Then $p_{n + 1} > M + y$, and hence
\begin{equation*}
p_{n + 1} - p_n > y.
\end{equation*}
Since $y < P(x)$, Bertrand's postulate gives
\begin{equation*}
\frac{P(x)}{2} < p_n < p_{n + 1} < 6P(x).
\end{equation*}
Consequently,
\begin{equation}
\label{eq:transfer-crt-bound}
G(6P(x)) > Y(x),
\end{equation}
and the prime number theorem gives $\log p_n \sim \log P(x) \sim x$.

Now fix $0 < \epsilon < 1$ and put $x = (1 - \epsilon)\log X$. The prime number theorem gives
\begin{equation*}
6P(x) = X^{1 - \epsilon + o(1)} < X
\end{equation*}
for sufficiently large $X$. Since $G$ is nondecreasing, \eqref{eq:transfer-crt-bound} and the assumed covering bound give
\begin{equation*}
G(X) \ge G(6P(x)) > Y(x) \ge F((1 - \epsilon)\log X).
\end{equation*}
Regular variation therefore implies
\begin{equation*}
\liminf_{X \to \infty}\frac{G(X)}{F(\log X)} \ge (1 - \epsilon)^{\alpha}.
\end{equation*}
Letting $\epsilon$ tend to zero proves the assertion about $G(X)$.

For the individual gaps, let $x$ tend to infinity through values for which $Y(x) \ge F(x)$, and use the gaps constructed above. Their left endpoints tend to infinity and satisfy $\log p_n \sim x$. Monotonicity and regular variation imply $F(\log p_n) \sim F(x)$: for every fixed $0 < \epsilon < 1$ and sufficiently large $x$,
\begin{equation*}
F((1 - \epsilon)x) \le F(\log p_n) \le F((1 + \epsilon)x),
\end{equation*}
and division by $F(x)$ followed by passage to the limit proves the claim. Thus
\begin{equation*}
p_{n + 1} - p_n > Y(x) \ge F(x) = (1 + o(1))F(\log p_n).
\end{equation*}
Passing to a subsequence with strictly increasing indices completes the proof.
\end{proof}

\begin{proof}[Proof of Corollary~\ref{cor:classical-erdos-rankin-prime-gaps}]
Theorem~\ref{thm:classical-erdos-rankin-bound} gives $Y(x) \ge F(x)$ for all sufficiently large $x$, where
\begin{equation*}
F(x) := \frac{cx(\log x)(\log_{3} x)}{(\log_{2} x)^2}
\end{equation*}
for some fixed $c > 0$. This function is eventually increasing and regularly varying of index $1$. Proposition~\ref{prop:covering-to-prime-gaps} gives both conclusions.
\end{proof}

\begin{remark}
\label{rem:transfer-weaker-hypothesis}
Under the stated positivity and monotonicity assumptions, regular variation can be replaced by the weaker condition
\begin{equation*}
\lim_{\epsilon \downarrow 0}\liminf_{t \to \infty}\frac{F((1 - \epsilon)t)}{F(t)} = 1.
\end{equation*}
This is sufficient both for the passage to every sufficiently large height $X$ and for replacing $F(x)$ by $F(\log p_n)$ when $\log p_n \sim x$.
\end{remark}

\subsection{An elementary smooth-number bound}
\label{subsec:classical-smooth-number-bound}

The construction will leave two kinds of integers to be covered: primes and smooth numbers. For the latter, we need an upper bound that remains uniform as $u = \log H/\log z$ grows. The following estimate suffices.

The proof follows Constantinescu \cite[Lemma~2.3]{CON2018}, combining Rankin's method \cite{RAN1938} with Mertens' estimates and convexity. The same choice of exponent, $\sigma = 1 - \log u/\log z$, appears in Pomerance's exposition \cite[p.~139]{POM1989}, where the prime number theorem gives a sharper estimate for the Euler product. In the range needed below, the bound also follows from \cite[Lemma~7.5]{MV2006}.

Write $\Psi(H,z)$ for the number of $z$-smooth positive integers less than or equal to $H$.

\begin{lemma}
\label{lem:rankin-smooth-numbers}
Let $H \ge z \ge 2$, put
\begin{equation*}
u := \frac{\log H}{\log z},
\end{equation*}
and suppose that $u \ge 2$, $u \le \log z$ and $3\log u \le \log z$. Then
\begin{equation*}
\Psi(H,z) \ll H(\log z)\exp(-u\log u + u).
\end{equation*}
The implied constant is absolute.
\end{lemma}

\begin{proof}
For $0 < \sigma < 1$, Rankin's trick gives
\begin{equation*}
\Psi(H,z) \le \sum_{\substack{n \, \ge \, 1 \\ p \mid n \,\Rightarrow\, p \, \le \, z}}\left(\frac{H}{n}\right)^{\sigma} = H^{\sigma}\prod_{p \, \le \, z}\left(1 - p^{-\sigma}\right)^{-1}.
\end{equation*}
Choose
\begin{equation*}
\sigma := 1 - \frac{\log u}{\log z}.
\end{equation*}

The assumption $3\log u \le \log z$ gives $\sigma \ge 2/3$. Consequently,
\begin{equation*}
\log\prod_{p \, \le \, z}\left(1 - p^{-\sigma}\right)^{-1} = \sum_{p \, \le \, z}\frac{1}{p^{\sigma}} + O(1),
\end{equation*}
since the contribution from the second and higher powers is bounded uniformly.

We write
\begin{equation*}
\sum_{p \, \le \, z}\frac{1}{p^{\sigma}} = \sum_{p \, \le \, z}\frac{1}{p} + \sum_{p \, \le \, z}\frac{p^{1 - \sigma} - 1}{p}.
\end{equation*}
Mertens' estimates give
\begin{equation*}
\sum_{p \, \le \, z}\frac{1}{p} = \log_{2} z + O(1), \qquad \sum_{p \, \le \, z}\frac{\log p}{p} = \log z + O(1).
\end{equation*}
For $0 \le t \le 1$ and $c > 0$, convexity gives
\begin{equation*}
e^{ct} - 1 \le (e^{c} - 1)t.
\end{equation*}
Taking $c = \log u$ and $t = \log p/\log z$, we obtain
\begin{equation*}
\sum_{p \, \le \, z}\frac{p^{1 - \sigma} - 1}{p} \le \frac{u - 1}{\log z}\sum_{p \, \le \, z}\frac{\log p}{p} = u + O(1),
\end{equation*}
where the last error is bounded because $u \le \log z$. It follows that
\begin{equation*}
\prod_{p \, \le \, z}\left(1 - p^{-\sigma}\right)^{-1} \ll e^{u}\log z.
\end{equation*}
Finally, since $\log H = u\log z$,
\begin{equation*}
H^{\sigma} = H\exp\left(-\frac{\log u}{\log z}\log H\right) = H\exp(-u\log u).
\end{equation*}
Combining the last two estimates with Rankin's trick proves the result.
\end{proof}

\subsection{Construction of the covering}
\label{subsec:classical-covering-construction}

We divide the primes up to $x$ into four ranges,
\begin{equation*}
[2,w], \qquad (w,z], \qquad (z,v], \qquad (v,x],
\end{equation*}
and use them in three stages. An integer is \emph{uncovered} after a stage if none of the residue classes chosen so far covers it. The primes $p \le v$ are the \emph{preliminary primes}, their chosen classes are the \emph{preliminary classes}, and removing the integers they cover from $\{1,\ldots,H\}$ is the \emph{preliminary sieve}. Its \emph{survivors} are the integers left uncovered after all preliminary classes have been chosen; they retain this designation even if covered later. The primes in $(v,x]$ are the \emph{reserve primes}, whose classes are chosen after the preliminary sieve.

A preliminary prime for which we prescribe $a_p = 0$ deterministically is called a \emph{clamping prime}, and its prescribed zero class a \emph{clamping class}. Here we clamp the primes in the first and third ranges, restricting the possible prime factors of the integers left uncovered. The condition $H \le wv$ will ensure that every such integer is either $z$-smooth or a prime greater than $v$.

The second range is used \emph{greedily}: for each prime in turn, we choose a residue class containing as many currently uncovered integers as possible. Since these integers are partitioned among $p$ residue classes, such a choice covers at least a proportion $1/p$ of them. We then repeat this procedure with the integers still uncovered. This completes the preliminary sieve.

Finally, we assign a distinct reserve prime to each survivor and choose its class to cover that survivor. We call this operation \emph{mopping up}. In the classical construction all reserve primes are held for this purpose, and we also call them \emph{mopping primes}. The estimates below show that there are enough of them.

\begin{proof}[Proof of Theorem \ref{thm:classical-erdos-rankin-bound}]
Fix a constant $0 < \theta < 1/2$. Throughout this proof, asymptotic statements refer to $x \to \infty$, with $\theta$ and the constant $c > 0$ chosen below held fixed. All other parameters are functions of $x$. 

\proofstep{Choice of parameters} We shall prove the result with
\begin{align*}
w & := \frac{(\log x)(\log_{3} x)}{(\log_{2} x)^2}, \\
z & := \exp\left(\frac{\theta(\log x)(\log_{3} x)}{\log_{2} x}\right), \\
v & := \frac{x}{2}, \\
H & := \left\lfloor \frac{cx(\log x)(\log_{3} x)}{(\log_{2} x)^2}\right\rfloor,
\end{align*}
where $c > 0$ is a sufficiently small constant depending only on $\theta$. For sufficiently large $x$,
\begin{equation*}
2 < w < z < v < x.
\end{equation*}
Moreover, if $c \le 1/2$, then
\begin{equation*}
H \le wv.
\end{equation*}

\proofstep{Clamping stage}
Choose $a_p = 0$ for $p \le w$ and $z < p \le v$, and let $\mathcal{N}_{0}$ denote the integers in $\{1,\ldots,H\}$ left uncovered after these choices.

Every prime factor of an integer in $\mathcal{N}_{0}$ lies in $(w,z]$ or exceeds $v$. If $n \in \mathcal{N}_{0}$ is not $z$-smooth, it therefore has a prime factor $q > v$. Write $n = qm$. If $m > 1$, then $m$ has a prime factor exceeding $w$, so
\begin{equation*}
n = qm > vw \ge H,
\end{equation*}
a contradiction. Thus $m = 1$ and $n = q$ is prime. Every member of $\mathcal{N}_{0}$ is therefore either $z$-smooth or a prime in $(v,H]$, giving
\begin{equation*}
\#\mathcal{N}_{0} \le \Psi(H,z) + \pi(H) - \pi(v) \le \Psi(H,z) + \pi(H).
\end{equation*}

\proofstep{Greedy stage}
For any finite set $\mathcal{S}$ of integers and any prime $p$,
\begin{equation*}
\#\mathcal{S} = \sum_{a \bmod p}\#\{n \in \mathcal{S} : n \equiv a \bmod p\}.
\end{equation*}
There are $p$ summands, so at least one is at least their average, $\#\mathcal{S}/p$. A residue class containing as many elements of $\mathcal{S}$ as possible therefore leaves at most $(1 - 1/p)\#\mathcal{S}$ uncovered.

Starting with $\mathcal{N}_{0}$, we make such a choice successively for each prime $w < p \le z$, applying it each time to the set of integers still uncovered. This completes the preliminary sieve, whose survivor set $\mathcal{N}$ thus satisfies
\begin{equation*}
\#\mathcal{N} \le \#\mathcal{N}_{0}\prod_{w \, < \, p \, \le \, z}\left(1 - \frac{1}{p}\right) \le \left(\Psi(H,z) + \pi(H)\right)\prod_{w \, < \, p \, \le \, z}\left(1 - \frac{1}{p}\right).
\end{equation*}
By Mertens' theorem,
\begin{equation*}
\prod_{w \, < \, p \, \le \, z}\left(1 - \frac{1}{p}\right) \sim \frac{\log w}{\log z} \sim \frac{(\log_{2} x)^2}{\theta(\log x)(\log_{3} x)}.
\end{equation*}

\proofstep{Counting the survivors}
We first estimate the contribution from the primes. Since $\log H \sim \log x$,
\begin{equation}
\label{eq:classical-prime-contribution}
\pi(H)\prod_{w \, < \, p \, \le \, z}\left(1 - \frac{1}{p}\right) \ll \frac{H}{\log H}\frac{\log w}{\log z} \ll \frac{cx}{\log x}.
\end{equation}
The implied constant may depend on $\theta$, but can be taken independent of $c$; the threshold for $x$ may depend on both.

It remains to control the smooth integers. Put
\begin{equation*}
u := \frac{\log H}{\log z}.
\end{equation*}
Our choices give
\begin{equation*}
u \sim \frac{\log_{2} x}{\theta\log_{3} x}, \qquad \log u \sim \log_{3} x,
\end{equation*}
and therefore
\begin{equation*}
u\log u = \left(\frac{1}{\theta} + o(1)\right)\log_{2} x.
\end{equation*}
The hypotheses of Lemma~\ref{lem:rankin-smooth-numbers} hold for sufficiently large $x$, so
\begin{equation*}
\Psi(H,z) \ll H(\log z)(\log x)^{-1/\theta + o(1)}.
\end{equation*}
After the greedy stage, the contribution of these integers is at most
\begin{equation}
\label{eq:classical-smooth-contribution}
\begin{aligned}
\Psi(H,z)\prod_{w \, < \, p \, \le \, z}\left(1 - \frac{1}{p}\right) 
& \ll H(\log w)(\log x)^{-1/\theta + o(1)} \\
& \ll \frac{x\log_{3} x}{\log_{2} x}(\log x)^{1 - 1/\theta + o(1)} \\
& = o\left(\frac{x}{\log x}\right),
\end{aligned}
\end{equation}
because $\theta < 1/2$.

\proofstep{Mopping up}
By the prime number theorem, the number of mopping primes satisfies
\begin{equation*}
\pi(x) - \pi(v) \sim \frac{x}{2\log x}.
\end{equation*}
Choose $c$ sufficiently small in terms of $\theta$. Equations~\eqref{eq:classical-prime-contribution} and \eqref{eq:classical-smooth-contribution} then give
\begin{equation*}
\#\mathcal{N} \le \pi(x) - \pi(v)
\end{equation*}
for all sufficiently large $x$. Assign a distinct prime $p \in (v,x]$ to each $n \in \mathcal{N}$ and choose $a_p \equiv n \bmod p$. Give any unused primes in this range arbitrary residue classes. The resulting classes cover $\{1,\ldots,H\}$, and hence $Y(x) \ge H$.
\end{proof}

\subsection{Choosing the parameters}
\label{subsec:classical-parameter-balance}

The proof specifies the parameters at the outset. We now explain how their sizes arise from balancing the prime and smooth contributions against the supply of mopping primes. Write
\begin{equation*}
L_1 := \log x, \qquad L_2 := \log L_1, \qquad L_3 := \log L_2.
\end{equation*}
Put $H = xh$, where the enlargement factor $h$ is to be chosen. We seek an enlargement factor $h \to \infty$ in the range $h = x^{o(1)}$, so that $\log H \sim L_1$. Taking $v = x/2$ and $w = 2h$ ensures that $H = wv$ and leaves $\asymp x/L_1$ mopping primes. With
\begin{equation*}
u := \frac{\log H}{\log z},
\end{equation*}
we have $\log z \sim L_1/u$ and $\log w \sim \log h$.

For parameters satisfying the hypotheses of Lemma~\ref{lem:rankin-smooth-numbers}, the estimates in the proof give
\begin{align*}
\frac{\pi(H)\prod_{w \, < \, p \, \le \, z}\,(1 - 1/p)}{x/L_1} & \ll \frac{hu\log h}{L_1}, \\
\frac{\Psi(H,z)\prod_{w \, < \, p \, \le \, z}\,(1 - 1/p)}{x/L_1} & \ll hL_1(\log h)\exp(-u\log u + u).
\end{align*}
To leave enough mopping primes for a distinct assignment to every survivor, we keep the prime contribution below a sufficiently small fixed proportion of $x/L_1$. For a given $u$, this permits us to choose $h$ by
\begin{equation*}
h\log h = \frac{\delta L_1}{u},
\end{equation*}
where $\delta > 0$ is a sufficiently small fixed constant. Substituting this choice into the second bound gives
\begin{equation*}
\frac{\Psi(H,z)\prod_{w \, < \, p \, \le \, z}\,(1 - 1/p)}{x/L_1} \ll \delta\frac{L_1^{2}}{u}\exp(-u\log u + u).
\end{equation*}
Apart from the fixed factor $\delta$, the logarithm of the expression on the right is
\begin{equation*}
2L_2 - u\log u + u - \log u.
\end{equation*}
We therefore want $u\log u$ to exceed $2L_2$ by a sufficient margin, while keeping $u$ small enough to allow a large value of $h$.

For fixed $\delta$, the relation $h\log h = \delta L_1/u$ makes $h$ decrease as $u$ increases. We therefore seek the smallest $u$ for which the smooth contribution is controlled. The threshold suggested by our bound is
\begin{equation*}
u\log u - u + \log u = 2L_2,
\end{equation*}
whose solution satisfies $u \sim 2L_2/L_3$. Thus the scale of $u$ is determined by the estimates. To stay above this threshold by a fixed margin, we take
\begin{equation*}
u = \frac{L_2}{\theta L_3}, \qquad 0 < \theta < \frac{1}{2},
\end{equation*}
with $\theta$ fixed. Then $\log u \sim L_3$, and the logarithmic expression above is $(2 - 1/\theta + o(1))L_2$, which tends to $-\infty$. The smooth contribution is therefore negligible compared with the number of mopping primes.

The equation defining $h$ now becomes
\begin{equation*}
h\log h = \delta\theta\frac{L_1L_3}{L_2}.
\end{equation*}
Taking logarithms gives $\log h + \log\log h = L_2 - L_3 + \log L_3 + O(1)$, and hence $\log h \sim L_2$. Consequently,
\begin{equation*}
H = xh \sim \delta\theta\frac{xL_1L_3}{L_2^{2}}, \qquad \log z \sim \frac{\theta L_1L_3}{L_2}.
\end{equation*}
These choices satisfy the assumed parameter restrictions for sufficiently large $x$. They recover the scale used in the proof: one factor $L_2$ in the denominator comes from $\log h \sim \log w$, while the other, together with $L_3$ in the numerator, comes from the size of $u$ required to control the smooth contribution.

This calculation explains the scale obtained from the estimates used here. Sharper smooth-number estimates improve the permissible constants without changing that scale. It does not establish an upper bound for what every construction using these four prime ranges and one mopping prime per survivor can achieve.

\subsection{A fixed choice of residue classes}
\label{subsec:fixed-residue-classes}

The second stage of the construction need not be greedy. Following Erd\H{o}s \cite{ERD1935}, we may choose $a_p = -1$ for every prime $w < p \le z$. The primes that survive then avoid one prescribed nonzero residue class modulo each of these primes. An upper-bound sieve shows that there are sufficiently few such integers.

Erd\H{o}s states the required estimate in \cite[Lemma~1]{ERD1935}, as a consequence of Brun's method \cite{BRU1920}. We derive the version needed here from the following standard consequence of Selberg's sieve, recorded in \cite[Theorem~3.13]{MV2006}. This separates the choice of residue classes from the estimate that justifies it.

\begin{lemma}[Upper-bound sieve]
\label{lem:fixed-upper-bound-sieve}
Let $T \ge 4$. For each prime $p \le \sqrt{T}$, let $\Omega_p$ be a set of residue classes modulo $p$, and write $\omega(p) := \#\Omega_p$. Suppose that $\omega(2) \le 1$ and $\omega(p) \le 2$ for $p > 2$. Then
\begin{equation*}
\#\{n \in \mathbb{N} : n \le T,\ n \bmod p \notin \Omega_p \text{ for every } p \le \sqrt{T}\,\} \ll T\prod_{p \, \le \, \sqrt{T}}\left(1 - \frac{\omega(p)}{p}\right),
\end{equation*}
with an absolute implied constant.
\end{lemma}

Indeed, Theorem~3.13 of \cite{MV2006} gives an upper bound of order
\begin{equation*}
\frac{T}{(\log T)^2}\prod_{p \, \le \, \sqrt{T}}\left(1 - \frac{\omega(p)}{p}\right)\left(1 - \frac{1}{p}\right)^{-2},
\end{equation*}
and Mertens' theorem gives the stated form.

For the fixed choice $a_p = -1$, the relevant consequence is the following.

\begin{lemma}[Primes avoiding the fixed classes]
\label{lem:fixed-prime-survivors}
Let $H \ge 16$ and $2 \le w < z \le \sqrt{H}$. Define
\begin{equation*}
\mathcal{Q}(H; w, z) := \{q \le H : q \text{ is prime},\ q \not\equiv -1 \bmod p \text{ for every prime } w < p \le z\}.
\end{equation*}
Then
\begin{equation*}
\#\mathcal{Q}(H;w,z) \ll \frac{H}{\log H}\prod_{w \, < \, p \, \le \, z}\left(1 - \frac{1}{p - 1}\right) \ll \frac{H\log w}{(\log H)\log z},
\end{equation*}
with absolute implied constants.
\end{lemma}

\begin{proof}
For primes $p \le \sqrt{H}$, take
\begin{equation*}
\Omega_p := \begin{cases} \{0,-1\} \bmod p, & w < p \le z, \\ \{0\} \bmod p, & \text{otherwise}. \end{cases}
\end{equation*}
Every member of $\mathcal{Q}(H;w,z)$ exceeding $\sqrt{H}$ avoids these classes. Since $w \ge 2$, the two classes are distinct whenever both are prescribed, and $\omega(2) = 1$. Lemma~\ref{lem:fixed-upper-bound-sieve} therefore gives
\begin{equation*}
\#\mathcal{Q}(H;w,z) \ll \sqrt{H} + H\prod_{p \, \le \, \sqrt{H}}\left(1 - \frac{1}{p}\right)\prod_{w \, < \, p \, \le \, z}\frac{1 - 2/p}{1 - 1/p}.
\end{equation*}
The first product is $\asymp 1/\log H$ by Mertens' theorem, while
\begin{equation*}
\frac{1 - 2/p}{1 - 1/p} = 1 - \frac{1}{p - 1}.
\end{equation*}
Moreover,
\begin{equation*}
\prod_{w \, < \, p \, \le \, z}\left(1 - \frac{1}{p - 1}\right) = \prod_{w \, < \, p \, \le \, z}\left(1 - \frac{1}{p}\right)\prod_{w \, < \, p \, \le \, z}\left(1 - \frac{1}{(p - 1)^2}\right) \asymp \frac{\log w}{\log z}.
\end{equation*}
Here the second product is bounded above and below by positive absolute constants, since all its primes exceed $2$ and $\sum_{p \, > \, 2}(p - 1)^{-2}$ converges. Finally, $\log w/\log z \gg 1/\log H$, so the term $\sqrt{H}$ is absorbed into the asserted bound.
\end{proof}

We can now give a second proof of Theorem~\ref{thm:classical-erdos-rankin-bound}. The clamping and mopping stages are unchanged. The fixed second stage is estimated only on the primes; every smooth integer is allowed to remain. We compensate by choosing a slightly smaller smoothness cutoff.

\begin{proof}[Second proof of Theorem~\ref{thm:classical-erdos-rankin-bound}]
Fix a constant $0 < \theta < 1/3$. Throughout this proof, asymptotic statements refer to $x \to \infty$, with $\theta$ and the constant $c > 0$ chosen below held fixed.

\proofstep{Choice of parameters}
Put
\begin{align*}
w & := \frac{(\log x)(\log_{3} x)}{(\log_{2} x)^2}, \\
z & := \exp\left(\frac{\theta(\log x)(\log_{3} x)}{\log_{2} x}\right), \\
v & := \frac{x}{2}, \\
H & := \left\lfloor \frac{cx(\log x)(\log_{3} x)}{(\log_{2} x)^2}\right\rfloor,
\end{align*}
where $0 < c \le 1/2$ will be chosen sufficiently small in terms of $\theta$. For sufficiently large $x$,
\begin{equation*}
2 < w < z < v < H, \qquad z \le \sqrt{H}, \qquad H \le wv.
\end{equation*}

\proofstep{Clamping stage}
Choose $a_p = 0$ for primes $p \le w$ and $z < p \le v$, and let $\mathcal{N}_{0}$ be the set of integers in $\{1,\ldots,H\}$ left uncovered after these choices. Every prime factor of a member of $\mathcal{N}_{0}$ lies in $(w,z]$ or exceeds $v$. If $n \in \mathcal{N}_{0}$ has a prime factor $q > v$ and is composite, its complementary factor has a prime factor exceeding $w$, giving
\begin{equation*}
n > vw \ge H.
\end{equation*}
Thus every member of $\mathcal{N}_{0}$ is either $z$-smooth or a prime in $(v,H]$.

\proofstep{Fixed-residue stage}
For every prime $w < p \le z$, choose $a_p = -1$. Write
\begin{equation*}
\mathcal{N} := \{n \in \mathcal{N}_{0} : n \not\equiv -1 \bmod p \text{ for every prime } w < p \le z\}.
\end{equation*}
Every prime member of $\mathcal{N}$ belongs to $\mathcal{Q}(H;w,z)$. Consequently,
\begin{equation*}
\#\mathcal{N} \le \Psi(H,z) + \#\mathcal{Q}(H;w,z).
\end{equation*}

\proofstep{Counting the survivors}
Since $\log H \sim \log x$ and $\log w \sim \log_{2} x$, Lemma~\ref{lem:fixed-prime-survivors} gives
\begin{equation}
\label{eq:fixed-prime-contribution}
\#\mathcal{Q}(H;w,z) \ll \frac{H\log w}{(\log H)\log z} \ll \frac{cx}{\log x}.
\end{equation}
The implied constant may depend on $\theta$, but can be taken independent of $c$; the threshold for $x$ may depend on both.

For the smooth integers, put
\begin{equation*}
u := \frac{\log H}{\log z}.
\end{equation*}
Our choices give
\begin{equation*}
u \sim \frac{\log_{2} x}{\theta\log_{3} x}, \qquad \log u \sim \log_{3} x, \qquad u\log u = \left(\frac{1}{\theta} + o(1)\right)\log_{2} x.
\end{equation*}
The hypotheses of Lemma~\ref{lem:rankin-smooth-numbers} hold for sufficiently large $x$. Since $u = o(\log_{2} x)$, that lemma yields
\begin{equation}
\label{eq:fixed-smooth-contribution}
\begin{aligned}
\Psi(H,z) & \ll H(\log z)\exp(-u\log u + u) \\
& \ll H(\log z)(\log x)^{-1/\theta + o(1)} \\
& \ll x(\log x)^{2 - 1/\theta + o(1)} \\
& = o\left(\frac{x}{\log x}\right),
\end{aligned}
\end{equation}
because $\theta < 1/3$.

\proofstep{Mopping up}
By the prime number theorem, the number of mopping primes satisfies
\begin{equation*}
\pi(x) - \pi(v) \sim \frac{x}{2\log x}.
\end{equation*}
Choose $c$ sufficiently small in terms of $\theta$. Equations~\eqref{eq:fixed-prime-contribution} and \eqref{eq:fixed-smooth-contribution} then give
\begin{equation*}
\#\mathcal{N} \le \pi(x) - \pi(v)
\end{equation*}
for all sufficiently large $x$. Assign a distinct prime $p \in (v,x]$ to each $n \in \mathcal{N}$ and choose $a_p \equiv n \bmod p$. Give any unused primes arbitrary residue classes. These choices cover $\{1,\ldots,H\}$, and hence $Y(x) \ge H$, as required.
\end{proof}

\subsection{Averaging over residue classes}
\label{subsec:averaging-residue-classes}

A third proof replaces the successive greedy choices by a single average over all choices of the second-stage residue classes. This is the argument used by Montgomery and Vaughan in \cite[p.~222, proof of Lemma~7.13]{MV2006}. The average number of survivors is exactly the upper bound obtained by the greedy procedure, so all the subsequent estimates remain unchanged.

\begin{lemma}[Averaging over residue classes]
\label{lem:averaging-residue-classes}
Let $\mathcal{S}$ be a finite set of integers and let $\mathcal{P}$ be a finite set of primes. Put
\begin{equation*}
Q := \prod_{p \, \in \, \mathcal{P}} p.
\end{equation*}
Then
\begin{equation*}
\frac{1}{Q}\sum_{b \bmod Q}\#\{n \in \mathcal{S} : (n - b,Q) = 1\} = \#\mathcal{S}\prod_{p \, \in \, \mathcal{P}}\left(1 - \frac{1}{p}\right).
\end{equation*}
Consequently, there are residue classes $a_p \bmod p$, one for each $p \in \mathcal{P}$, such that
\begin{equation*}
\#\{n \in \mathcal{S} : n \not\equiv a_p \bmod p \text{ for every } p \in \mathcal{P}\} \le \#\mathcal{S}\prod_{p \, \in \, \mathcal{P}}\left(1 - \frac{1}{p}\right).
\end{equation*}
\end{lemma}

\begin{proof}
For each fixed $n$, the difference $n - b$ runs through all residue classes modulo $Q$ as $b$ does. Exactly $\phi(Q)$ of these classes are coprime to $Q$. Interchanging the order of summation therefore gives
\begin{equation*}
\sum_{b \bmod Q}\#\{n \in \mathcal{S} : (n - b,Q) = 1\} = \sum_{n \, \in \, \mathcal{S}}\#\{b \bmod Q : (n - b,Q) = 1\} = \#\mathcal{S}\phi(Q).
\end{equation*}
Dividing by $Q$ proves the identity. For at least one residue class $b \bmod Q$, the number of $n \in \mathcal{S}$ satisfying $(n - b,Q) = 1$ is no greater than this average. Choosing $a_p \equiv b \bmod p$ for every $p \in \mathcal{P}\,$ proves the assertion. 
\end{proof}

By the Chinese remainder theorem, letting $b$ run through the classes modulo $Q$ runs through every possible tuple of classes modulo the primes in $\mathcal{P}$, exactly once. Thus the common representative $b$ places no restriction on the choices. The lemma selects all the classes together, without requiring any individual choice to be greedy.

\begin{proof}[Third proof of Theorem~\ref{thm:classical-erdos-rankin-bound}]
Fix a constant $0 < \theta < 1/2$. Throughout this proof, asymptotic statements refer to $x \to \infty$, with $\theta$ and the constant $c > 0$ chosen below held fixed.

\proofstep{Choice of parameters}
Put
\begin{align*}
w & := \frac{(\log x)(\log_{3} x)}{(\log_{2} x)^2}, \\
z & := \exp\left(\frac{\theta(\log x)(\log_{3} x)}{\log_{2} x}\right), \\
v & := \frac{x}{2}, \\
H & := \left\lfloor \frac{cx(\log x)(\log_{3} x)}{(\log_{2} x)^2}\right\rfloor,
\end{align*}
where $0 < c \le 1/2$ will be chosen sufficiently small in terms of $\theta$. For sufficiently large $x$,
\begin{equation*}
2 < w < z < v < H, \qquad H \le wv.
\end{equation*}

\proofstep{Clamping stage}
Choose $a_p = 0$ for primes $p \le w$ and $z < p \le v$, and let $\mathcal{N}_{0}$ be the set of integers in $\{1,\ldots,H\}$ left uncovered after these choices. Every prime factor of a member of $\mathcal{N}_{0}$ lies in $(w,z]$ or exceeds $v$. If $n \in \mathcal{N}_{0}$ has a prime factor $q > v$ and is composite, its complementary factor has a prime factor exceeding $w$, giving
\begin{equation*}
n > vw \ge H.
\end{equation*}
Thus every member of $\mathcal{N}_{0}$ is either $z$-smooth or a prime in $(v,H]$, and
\begin{equation*}
\#\mathcal{N}_{0} \le \Psi(H,z) + \pi(H).
\end{equation*}

\proofstep{Averaging stage}
Apply Lemma~\ref{lem:averaging-residue-classes} to $\mathcal{N}_{0}$ and the primes $w < p \le z$. Choose the residue classes supplied by the lemma, and let $\mathcal{N}$ be the survivor set of the completed preliminary sieve. Then
\begin{equation*}
\#\mathcal{N} \le \left(\Psi(H,z) + \pi(H)\right)\prod_{w < p \le z}\left(1 - \frac{1}{p}\right).
\end{equation*}
By Mertens' theorem,
\begin{equation*}
\prod_{w \, < \, p \, \le \, z}\left(1 - \frac{1}{p}\right) \sim \frac{\log w}{\log z} \sim \frac{(\log_{2} x)^2}{\theta(\log x)(\log_{3} x)}.
\end{equation*}

\proofstep{Counting the survivors}
Since $\log H \sim \log x$, the prime contribution satisfies
\begin{equation}
\label{eq:averaging-prime-contribution}
\pi(H)\prod_{w \, < \, p \, \le \, z}\left(1 - \frac{1}{p}\right) \ll \frac{H}{\log H}\frac{\log w}{\log z} \ll \frac{cx}{\log x}.
\end{equation}
The implied constant may depend on $\theta$, but can be taken independent of $c$; the threshold for $x$ may depend on both.

For the smooth contribution, put
\begin{equation*}
u := \frac{\log H}{\log z}.
\end{equation*}
Our choices give
\begin{equation*}
u \sim \frac{\log_{2} x}{\theta\log_{3} x}, \qquad \log u \sim \log_{3} x, \qquad u\log u = \left(\frac{1}{\theta} + o(1)\right)\log_{2} x.
\end{equation*}
The hypotheses of Lemma~\ref{lem:rankin-smooth-numbers} hold for sufficiently large $x$. Since $u = o(\log_{2} x)$, that lemma gives
\begin{equation*}
\Psi(H,z) \ll H(\log z)(\log x)^{-1/\theta + o(1)}.
\end{equation*}
Consequently,
\begin{equation}
\label{eq:averaging-smooth-contribution}
\begin{aligned}
\Psi(H,z)\prod_{w \, < \, p \, \le \, z}\left(1 - \frac{1}{p}\right) & \ll H(\log w)(\log x)^{-1/\theta + o(1)} \\
& \ll \frac{x\log_{3} x}{\log_{2} x}(\log x)^{1 - 1/\theta + o(1)} \\
& = o\left(\frac{x}{\log x}\right),
\end{aligned}
\end{equation}
because $\theta < 1/2$.

\proofstep{Mopping up}
By the prime number theorem, the number of mopping primes satisfies
\begin{equation*}
\pi(x) - \pi(v) \sim \frac{x}{2\log x}.
\end{equation*}
Choose $c$ sufficiently small in terms of $\theta$. Equations~\eqref{eq:averaging-prime-contribution} and \eqref{eq:averaging-smooth-contribution} then give
\begin{equation*}
\#\mathcal{N} \le \pi(x) - \pi(v)
\end{equation*}
for all sufficiently large $x$. Assign a distinct prime $p \in (v,x]$ to each $n \in \mathcal{N}$ and choose $a_p \equiv n \bmod p$. Give any unused primes arbitrary residue classes. These choices cover $\{1,\ldots,H\}$, and hence $Y(x) \ge H$, as required.
\end{proof}

\subsection{A probabilistic framework for the preliminary sieve}
\label{subsec:probabilistic-preliminary-sieve}

The three proofs above follow the same plan. The clamping primes constrain the factorizations of the integers left uncovered, the second stage reduces their number, and the reserve primes complete the covering by mopping up. The clamping and second stages together form the preliminary sieve, using all primes $p \le v$. In each proof, its survivor set $\mathcal{N}$ satisfies
\begin{equation*}
\#\mathcal{N} \le \pi(x) - \pi(v),
\end{equation*}
so that each survivor can be assigned a distinct reserve prime.

This condition is sufficient, but not necessary: a single residue class modulo a reserve prime may contain several survivors. To exploit this, we need both suitable groups of congruent survivors and a way to select their classes without repeatedly covering the same integers. A larger survivor set may therefore admit an efficient covering even when its cardinality alone gives no such assurance. Using more primes in the preliminary sieve leaves fewer in reserve; changing the preliminary residue classes affects both the number and the arrangement of the survivors. The broader objective is to leave a set that can be covered with the reserve primes.

We first place the preliminary choices in a common probabilistic framework. This recovers the averaging argument and its greedy counterpart, while allowing distributions beyond the fixed and uniform choices used above.

\subsubsection*{Independent residue choices}
\label{subsubsec:independent-residue-classes}

Let $\mathcal{S}$ be a finite set of integers and $\mathcal{P}$ a finite set of primes. For each $p \in \mathcal{P}$, let $\mu_p$ be a probability distribution on $\mathbb{Z}/p\mathbb{Z}$, so that
\begin{equation*}
\mu_p(a) \ge 0, \qquad \sum_{a \bmod p}\mu_p(a) = 1.
\end{equation*}
For an integer $n$, we write $\mu_p(n)$ for the mass assigned to its residue class modulo $p$. Choose the classes $a_p \bmod p$ independently according to these distributions. Thus the probability space and its measure are
\begin{equation*}
\Omega := \prod_{p \, \in \, \mathcal{P}}\mathbb{Z}/p\mathbb{Z}, \qquad \mathbb{P} := \bigotimes_{p \, \in \, \mathcal{P}}\mu_p.
\end{equation*}
For a complete choice $\mathbf{a} = (a_p)_{p \, \in \, \mathcal{P}}\,$, put
\begin{equation*}
\mathcal{N}(\mathbf{a}) := \{n \in \mathcal{S} : n \not\equiv a_p \bmod p \text{ for every } p \in \mathcal{P}\}.
\end{equation*}
We write $\mathcal{N}$ when the dependence on $\mathbf{a}$ is understood.

\begin{lemma}[Expected number of remaining integers]
\label{lem:probabilistic-remaining-count}
With the notation above,
\begin{equation}
\label{eq:probabilistic-expected-count}
\mathbb{E}\,\#\mathcal{N} = \sum_{n \, \in \, \mathcal{S}}\prod_{p \, \in \, \mathcal{P}}\left(1 - \mu_p(n)\right).
\end{equation}
In particular, there is a choice of the residue classes for which $\#\mathcal{N}$ is no greater than the right-hand side.
\end{lemma}

\begin{proof}
For a fixed $n \in \mathcal{S}$, the probability that $n \not\equiv a_p \bmod p$ is $1 - \mu_p(n)$. Independence of the choices for different primes gives
\begin{equation*}
\mathbb{P}(n \in \mathcal{N}) = \prod_{p \, \in \, \mathcal{P}}\left(1 - \mu_p(n)\right).
\end{equation*}
Summing over $n \in \mathcal{S}$ proves the identity. Some outcome has cardinality no greater than its expectation.
\end{proof}

For the covering construction, take $\mathcal{S} = \{1,\ldots,H\}$ and $\mathcal{P} = \{p : p \le v\}$. The primes in $(v,x]$ remain unassigned. If the expectation in \eqref{eq:probabilistic-expected-count} is at most $\pi(x) - \pi(v)$, the lemma supplies a realization that can be completed by individual mopping.

\subsubsection*{The fixed, uniform and greedy choices}
\label{subsubsec:fixed-uniform-and-greedy-choices}

A deterministic choice is represented by a distribution concentrated on one class. In particular, clamping a preliminary prime $p$ means taking $\mu_p(0) = 1$; in the classical construction we do this for $p \le w$ and $z < p \le v$. Every integer removed by these choices contributes zero to \eqref{eq:probabilistic-expected-count}. The remaining sum is over the set $\mathcal{N}_{0}$ left after clamping.

If the second-stage distributions are uniform, so that $\mu_p(a) = 1/p$ for $w < p \le z$, then
\begin{equation*}
\mathbb{E}\,\#\mathcal{N} = \#\mathcal{N}_{0}\prod_{w \, < \, p \, \le \, z}\left(1 - \frac{1}{p}\right).
\end{equation*}
This is the Montgomery--Vaughan averaging argument \cite[p.~222]{MV2006}, expressed through independent residue choices.

Greedy selection gives a deterministic way to attain this bound. Suppose that some second-stage classes have been fixed, leaving a set $\mathcal{T}$, and let $\mathcal{R}$ be the set of second-stage primes still unassigned. Keep the choices at these primes uniform and independent. For $p \in \mathcal{R}$, fixing $a_p = a$ makes the conditional expected final count
\begin{equation*}
\left(\#\mathcal{T} - \#\{n \in \mathcal{T} : n \equiv a \bmod p\}\right)\prod_{q \, \in \, \mathcal{R}\setminus\{p\}}\left(1 - \frac{1}{q}\right).
\end{equation*}
The product is independent of $a$. Choosing a class containing as many members of $\mathcal{T}$ as possible therefore minimizes this conditional expectation. At least one choice gives a value no greater than the average over $a$. Repeating the procedure fixes all the classes without increasing the conditional expectation, and the final count is no greater than the original expectation. Thus the greedy proof is a derandomization of the uniform model by successive conditional expectations.

The Erd\H{o}s-inspired proof instead takes $\mu_p(-1) = 1$ throughout the second range. In this case every choice is deterministic, and the right-hand side of \eqref{eq:probabilistic-expected-count} is the actual number of remaining integers. Lemma~\ref{lem:fixed-prime-survivors} supplies the arithmetic estimate needed to bound it. The probabilistic notation accommodates this proof, but does not replace its sieve input.

\subsubsection*{Beyond fixed and uniform choices}
\label{subsubsec:beyond-fixed-and-uniform-choices}

The classical constructions assign sharply separated roles to the primes. Some invariably select zero, while the averaging argument selects uniformly in the second range. What if the transition were more gradual? We could favor zero without selecting it invariably, and allow the strength of this preference to depend on the prime.

Such a choice involves a tradeoff. Selecting zero covers every multiple of the prime and no nonmultiple. Increasing the weight assigned to zero therefore makes each multiple less likely to survive, while making each nonmultiple more likely to survive, since the nonzero classes share the remaining weight equally. Decreasing the weight assigned to zero reverses these effects. The resulting set need not be smaller. The question is whether replacing some fixed zero classes by biased random choices can leave a set for which efficient final covering is easier to establish.

Consider the family
\begin{equation*}
\mu_p(0) = 1 - \beta_p, \qquad \mu_p(a) = \frac{\beta_p}{p - 1}\quad(a \ne 0), \qquad 0 \le \beta_p \le 1.
\end{equation*}
Here $\beta_p = 0$ gives clamping, $\beta_p = (p - 1)/p$ gives uniform selection among all classes, and $\beta_p = 1$ gives uniform selection among the nonzero classes. Values between $0$ and $(p - 1)/p$ favor zero relative to the uniform distribution.

For this family, the local survival probabilities are
\begin{equation*}
\mathbb{P}(n \not\equiv a_p \bmod p) = \begin{cases} \beta_p, & p \mid n, \\ 1 - \beta_p/(p - 1), & p \nmid n. \end{cases}
\end{equation*}
When its denominator is nonzero, the ratio
\begin{equation*}
\frac{\beta_p}{1 - \beta_p/(p - 1)}
\end{equation*}
therefore measures the relative survival of multiples and nonmultiples of $p$. Rather than requiring every multiple to disappear, we may prescribe how strongly its survival is suppressed.

\subsubsection*{A multiplicative survival law}
\label{subsubsec:a-multiplicative-survival-law}

A particularly useful prescription is to make the relative survival probability equal to $p^{-\tau}$, where $\tau > 0$. Solving
\begin{equation*}
\frac{\beta_p}{1 - \beta_p/(p - 1)} = p^{-\tau}
\end{equation*}
gives
\begin{equation*}
\beta_p = \frac{(p - 1)p^{-\tau}}{p - 1 + p^{-\tau}}.
\end{equation*}
This is the tilted residue law used in \cite[Section~3]{GPT2026}. Its form ensures that each prime divisor contributes a simple multiplicative factor to the survival probability.

\begin{lemma}[Multiplicative survival law]
\label{lem:app-probabilistic-multiplicative-survival}
Let $\mathcal{P}$ be a finite set of primes, and use the tilted distributions above independently for $p \in \mathcal{P}$. Put
\begin{equation*}
B_p := \frac{p - 1}{p - 1 + p^{-\tau}}, \qquad A := \prod_{p \, \in \, \mathcal{P}} B_p.
\end{equation*}
For every positive integer $n$,
\begin{equation*}
\mathbb{P}(n \not\equiv a_p \bmod p \text{ for every } p \in \mathcal{P}) = A\prod_{\substack{p \, \in \, \mathcal{P} \\ p \mid n}}p^{-\tau}.
\end{equation*}
In particular, if $n$ is squarefree and all its prime factors belong to $\mathcal{P}$, this probability is $An^{-\tau}$.
\end{lemma}

\begin{proof}
If $p \nmid n$, the local survival probability is
\begin{equation*}
1 - \frac{\beta_p}{p - 1} = 1 - \frac{p^{-\tau}}{p - 1 + p^{-\tau}} = B_p.
\end{equation*}
If $p \mid n$, it is $\beta_p = B_p p^{-\tau}$. Thus every prime contributes a factor $B_p$, with an additional factor $p^{-\tau}$ when it divides $n$. By independence,
\begin{equation*}
\mathbb{P}(n \not\equiv a_p \bmod p \text{ for every } p \in \mathcal{P})
= \prod_{\substack{p \, \in \, \mathcal{P} \\ p \nmid n}} B_p \prod_{\substack{p \, \in \, \mathcal{P} \\ p \mid n}}(B_p p^{-\tau}) \\
= A\prod_{\substack{p \, \in \, \mathcal{P} \\ p \mid n}}p^{-\tau}.
\end{equation*}
If $n$ is squarefree and all its prime factors belong to $\mathcal{P}$, the final product equals $n^{-\tau}$.
\end{proof}

The squarefree hypothesis in the last assertion is essential: the residue condition detects whether $p$ divides $n$, but not its multiplicity. Additional deterministic clamping can be imposed at primes outside $\mathcal{P}\,$; the displayed law then applies to integers that pass those fixed choices.

For fixed $\tau > 0$, the factor $p^{-\tau}$ imposes a stronger relative penalty at larger primes. At $\tau = 0$, the same formulas give uniform selection, while for each fixed prime the distribution approaches clamping as $\tau$ tends to infinity. The parameter $\tau$ may itself depend on the scale of the construction.

The benefit is more than an interpolation between familiar choices. A local preference for zero has produced an explicit weight on the factorizations of the remaining integers. In the relevant squarefree class, integers of comparable size have comparable survival probabilities regardless of how their prime factors are distributed. This gives a tractable starting point for counting the remaining composites. Covering survivors in groups, however, also requires information about their joint survival probabilities.

\subsubsection*{Groups of remaining integers}
\label{subsubsec:groups-of-remaining-integers}

The product model gives an equally direct formula for the probability that several integers all remain. For a finite set $\mathcal{T}$, write $\mathcal{T}\bmod p$ for the set of distinct residue classes represented by its members.

\begin{lemma}[Joint survival]
\label{lem:app-probabilistic-joint-survival}
Under the independent residue distributions $\mu_p$, for every $\mathcal{T} \subseteq \mathcal{S}$,
\begin{equation*}
\mathbb{P}(\mathcal{T} \subseteq \mathcal{N}) = \prod_{p \, \in \, \mathcal{P}}\left(1 - \sum_{a \, \in \, \mathcal{T}\bmod p}\mu_p(a)\right).
\end{equation*}
\end{lemma}

\begin{proof}
Every member of $\mathcal{T}$ escapes the selected class modulo $p$ precisely when that class lies outside $\mathcal{T}\bmod p$. This has probability $1 - \sum_{a \, \in \, \mathcal{T}\bmod p}\mu_p(a)$. The choices for different primes are independent.
\end{proof}

Under uniform selection, this becomes
\begin{equation*}
\mathbb{P}(\mathcal{T} \subseteq \mathcal{N}) = \prod_{p \, \in \, \mathcal{P}}\left(1 - \frac{\#(\mathcal{T}\bmod p)}{p}\right).
\end{equation*}
Thus independence of the residue choices does not mean independence of the survival events for different integers. Their joint survival depends on the residue classes they occupy together.

For a reserve prime $q$, a set $\mathcal{T}$ whose members are all congruent modulo $q$ can be covered by one class if it survives the preliminary sieve. The joint-survival formula allows us to estimate how often such candidate groups remain. A further argument must then show that sufficiently many groups can be selected, with at most one class per reserve prime and with controlled overlap. This is where the arrangement of the survivors enters the covering problem, beyond their number alone.

The preliminary sieve should therefore be judged by whether its remaining set can be covered with the reserve primes. A small cardinality guarantees this in the classical argument, but more efficient covering requires information about the arrangement of the survivors. The tilted distributions provide explicit survival laws with which to study that arrangement; a further covering argument is needed to turn this information into an improvement.

\clearpage

\section{Historical development of the Erdős--Rankin sieve}
\label{app:history}

Let $p_n$ be the $n$th term in the sequence of all prime numbers:
\begin{equation*}
p_{1} = 2 < p_{2} = 3 < p_{3} = 5 < \cdots.
\end{equation*}
How large can a gap $p_{n + 1} - p_{n}$ between consecutive primes be?

\subsection{From factorials to residue-class coverings}
\label{subsec:factorials}
One answer is: arbitrarily large, as can be seen from the run of composite integers
\begin{equation*}
m! + 2,\, m! + 3,\, \ldots,\, m! + m.
\end{equation*}
But this does not yet put the size of the gap in perspective. The prime number theorem states that the number of primes less than or equal to $X$, denoted $\pi(X)$, satisfies
\begin{equation*}
\pi(X) \sim \frac{X}{\log X}.
\end{equation*}
Thus the average spacing between primes up to $X$ is asymptotic to $\log X$. In particular, there are gaps below $X$ of size at least
\begin{equation*}
(1 - o(1))\log X.
\end{equation*}

This already beats the simple factorial construction by a logarithmic factor. Indeed, the construction above gives a gap of length about $m$ at height about $m!$, whereas the prime number theorem guarantees a gap below $m!$ of size at least
\begin{equation*}
(1 - o(1))\log(m!) \sim m\log m,
\end{equation*}
by Stirling's formula.

The factorial construction is wasteful: to make $M + k$ composite, it is enough that $M$ contain one prime divisor of $k$. Accordingly, let
\begin{equation*}
P_m := p_{1} p_{2} \cdots p_{m}.
\end{equation*}
If $2 \le k \le p_{m}$, then $k$ has a prime divisor among $p_{1},\ldots,p_{m}$, and hence
\begin{equation*}
P_{m} + 2,\, P_{m} + 3,\, \ldots,\, P_{m} + p_{m}
\end{equation*}
are all composite. The construction therefore produces a prime gap of length at least $p_{m}$. By the prime number theorem in Chebyshev's form,
\begin{equation*}
\log P_{m} = \sum_{p \, \le \, p_{m}}\log p \sim p_{m}.
\end{equation*}
Hence the primorial construction produces gaps on precisely the scale supplied by the prime number theorem. 

It is not, however, the most efficient way to use the small primes. In the construction above we have made the same choice
\begin{equation*}
M \equiv 0 \bmod p
\end{equation*}
for every prime $p \le p_{m}$. By choosing the residue class modulo each prime separately, one can do better. The following construction goes back to Legendre \cite{LEG1830}. Let $m \ge 3$. By the Chinese remainder theorem, choose
\begin{equation*}
P_{m} \le M< 2P_{m}
\end{equation*}
such that
\begin{alignat*}{3}
M & \equiv -p_{m - 1} &\quad& \bmod p &\qquad& (p < p_{m - 1}), \\
M & \equiv -(p_{m - 1} - 1) &\quad& \bmod p_{m - 1}, \\
M & \equiv -(p_{m - 1} + 1) &\quad& \bmod p_m.
\end{alignat*}

The last two congruences show that $M + p_{m - 1} - 1$ and $M + p_{m - 1} + 1$ are composite, while the first shows that $M + p_{m - 1}$ is divisible by every prime $p < p_{m - 1}$. For every other $k \in [1, 2p_{m - 1} - 1]$, we have
\begin{equation*}
2 \le |k - p_{m - 1}| \le p_{m - 1} - 1.
\end{equation*}
Hence any prime divisor $p$ of $|k - p_{m - 1}|$ satisfies $p < p_{m - 1}$, and
\begin{equation*}
M + k \equiv -p_{m - 1} + k \equiv 0 \bmod p.
\end{equation*}
Thus
\begin{equation*}
M + 1,\, M + 2,\, \ldots,\, M + 2p_{m - 1} - 1
\end{equation*}
is a run of $2p_{m - 1} - 1$ consecutive composite integers, lying at height comparable to $P_{m}$.

Legendre went further and asserted that this construction was essentially optimal \cite[p.~76, no.~410]{LEG1830}. This assertion formed part of a more general claim about arithmetic progressions avoiding prescribed prime divisors, on which Legendre based an attempted proof of the infinitude of primes in every reduced arithmetic progression. Dirichlet \cite{DIR1837} identified a gap in the argument and proved the latter theorem by different means in 1837. The validity of Legendre’s claim became the subject of the Académie des Sciences’ grand prize in mathematics for 1858; Dupré’s counterexamples appeared in 1859 \cite{DUP1859}. Further counterexamples were given independently by Moreau and Piltz; see Ricci \cite[pp.~190--191]{RIC1934} for this early history.

Curiously, Legendre had himself conjectured the asymptotic formula \cite[p.~65, \S VIII, no.~394]{LEG1830}
\begin{equation*}
\pi(X) \sim \frac{X}{\log X - 1.08366},
\end{equation*}
and hence, in particular, the first-order asymptotic $\pi(X) \sim X/\log X$.
Had this conjecture been available as a theorem, his construction would already have implied the existence of prime gaps asymptotically almost twice the average size. This conclusion was first obtained unconditionally by Backlund \cite{BAC1929}, who used Legendre's construction together with the prime number theorem to show that, for every $\epsilon > 0$, there are infinitely many $n$ for which
\begin{equation*}
p_{n + 1} - p_n > (2 - \epsilon)\log p_n.
\end{equation*}
Backlund's result appears to mark the beginning of the modern study of large prime gaps: it was the first known result showing that infinitely many gaps exceed the average spacing at their scale by a fixed factor greater than one.

In 1930, Brauer and Zeitz \cite{BZ1930, ZEI1930} strengthened these earlier counterexamples by proving that, for every $p_m \ge 43$, there are at least
\begin{equation*}
2p_{m - 1} + 1
\end{equation*}
consecutive integers, each divisible by at least one of the primes $p_{1},\ldots,p_{m}$. Thus Legendre's proposed maximum $2p_{m - 1} - 1$ is too small by at least two.

Their argument already contains an idea that became characteristic of later constructions: some primes are set aside to deal with the small exceptional set left by the main construction. In modernized notation, choose an integer $A < p_{m}$, and first arrange that an odd integer $M$ is divisible by every odd prime $p \le A$. Then, for $1 \le a \le A$,
\begin{equation*}
M \pm 2a \equiv 0 \bmod p
\end{equation*}
whenever $a$ has an odd prime divisor $p$. The only integers not thereby forced to be divisible by one of these primes are those for which $a$ is a power of $2$:
\begin{equation*}
M \pm 2,\, M \pm 2^{2},\, \ldots,\, M \pm 2^{k}.
\end{equation*}
There are only $2k = O(\log p_{m})$ such exceptional integers. Brauer and Zeitz reserve sufficiently many primes near the top of the range $p \le p_{m}$ and assign one of them to each exception, choosing the corresponding residue class so that the exceptional integer is divisible by that prime. The Chinese remainder theorem then combines these choices.

For every $\eta > 0$ and all sufficiently large $m$, they thereby obtain
\begin{equation*}
\left\lfloor (2 - \eta)p_m \right\rfloor
\end{equation*}
consecutive odd integers, each having a common prime factor with $P_m$. The even integers between them are divisible by $2$, so this gives a run of
\begin{equation*}
2\left\lfloor (2 - \eta)p_m \right\rfloor - 1
\end{equation*}
consecutive composite integers. It follows, on taking $\eta$ sufficiently small, that Backlund's $2 - \epsilon$ may be replaced by $4 - \epsilon$.

\subsection{The covering problem and the transfer to prime gaps}
\label{subsec:covering-problem}
At this point it is useful to give names to the two quantities that have been implicit in the discussion. Let
\begin{equation*}
G(X) := \max_{p_{n + 1} \, \le \, X} \, (p_{n + 1} - p_{n}),
\end{equation*}
and let $Y(x)$ be the largest integer $y$ for which one can choose a residue class $a_p \bmod p$ for every prime $p \le x$ such that
\begin{equation*}
\{1,\ldots,y\} \subseteq \bigcup_{p \, \le \, x} (a_p + p\mathbb{Z}).
\end{equation*}
Such a choice is a \emph{covering} of the interval. 

In this notation, Legendre's construction gives $Y(p_m) \ge 2p_{m - 1} - 1$, while his asserted extremality would give equality. Brauer and Zeitz proved instead that $Y(p_{m}) \ge 2p_{m - 1} + 1$ for every $m \ge 14$, and more strongly that $Y(p_{m}) \ge (4 - \epsilon)p_{m}$ for all sufficiently large $m$.

The connection between the two functions is supplied by the Chinese remainder theorem. A choice of residue classes as in the definition of $Y(x)$ may be realized simultaneously by choosing an integer $M$ such that, for each $1 \le k \le Y(x)$, the integer $M + k$ is divisible by some prime $p \le x$. Taking $M > x$ makes every integer in this run composite. By the prime number theorem in Chebyshev's form,
\begin{equation*}
\log \left(\prod_{p \, \le \, x} p\right) \sim x.
\end{equation*}
Thus a lower bound for $Y(x)$ at the scale $x$ gives a corresponding lower bound for $G(X)$ after the substitution $x \sim \log X$. We shall make this correspondence, including the relation between bounds for $G(X)$ and the older formulations in terms of infinitely many gaps $p_{n + 1} - p_n$, precise below.

\subsection{Westzynthius and the basic sieve architecture}
\label{subsec:basic-architecture}
Westzynthius \cite{WES1931} then made the decisive qualitative advance. He showed that the constant $4$ in the Brauer--Zeitz bound could be replaced by an arbitrarily large constant. More precisely, for every $\epsilon > 0$ and all sufficiently large $m$,
\begin{equation*}
Y(p_m) \ge (2 - \epsilon)e^{\gamma} \frac{p_m \log_{2} p_m}{\log_{3} p_m},
\end{equation*}
where $\log_{j} = \log(\log_{j - 1})$ denotes the $j$-fold iterated logarithm and $\gamma$ is the Euler--Mascheroni constant, with
\begin{equation*}
\gamma = 0.577215\ldots, \qquad e^\gamma = 1.781072\ldots.
\end{equation*}
The Chinese remainder theorem translates such a choice of residue classes into a run of consecutive composite integers at a height whose logarithm is asymptotic to the prime cutoff. We make this correspondence, including its formulation in terms of the record-gap function $G(X)$ and the relevant quantifiers, precise below. Consequently, for infinitely many $n$,
\begin{equation*}
p_{n + 1} - p_n > (2 - \epsilon)e^{\gamma} \frac{\log p_n \log_{3} p_n}{\log_{4} p_n}.
\end{equation*}

Westzynthius's proof combines restrictions on factorization with successive choices of residue classes. He assigns zero classes to a middle range of primes, leaving only offsets composed of smaller primes and offsets with a larger prime factor. On the latter set he uses the smallest primes in turn, choosing at each step a residue class containing as many of the remaining offsets as possible \cite[pp.~20--21]{WES1931}. Such a choice is called \emph{greedy}: it maximizes the number covered at the current step. Since there are $p$ residue classes modulo $p$, this removes at least a proportion $1/p$ of the offsets then remaining. The largest available primes are reserved to cover the remaining offsets individually. An elementary count bounds the offsets composed entirely of small primes; a sharper version of that count in the final section gives the constant $2e^{\gamma}$ \cite[pp.~33--37]{WES1931}.

The following modern presentation, due to Montgomery and Vaughan \cite[pp.~221--222]{MV2006}, gives a particularly simple proof of the qualitative conclusion. It assigns zero classes also to the smallest primes, so that any integer left uncovered by the zero classes with a large prime factor must itself be prime.

Fix a large constant $L$ and put
\begin{equation*}
H := \left\lfloor \frac{xL}{3} \right\rfloor.
\end{equation*}
Split the primes $p \le x$ into the four ranges
\begin{equation*}
p \le L, \qquad L < p \le L^{L}, \qquad L^{L} < p \le \frac{x}{3}, \qquad \frac{x}{3} < p \le x.
\end{equation*}
To prove a lower bound for $Y(x)$, we seek one residue class $a_p \bmod p$ for each prime $p \le x$ whose union contains $\{1,\ldots,H\}$. An integer is \emph{covered} by a chosen class if it belongs to that class, and is \emph{uncovered} after a stage if none of the classes chosen so far covers it. We use the primes up to $v = x/3$ for the \emph{preliminary sieve}, holding those in $(v,x]$ in reserve until that sieve is complete. We call these the \emph{preliminary primes} and the \emph{reserve primes}, respectively; the classes chosen for the preliminary primes are the \emph{preliminary classes}.

We begin by prescribing the zero class at every prime in the first and third ranges. Prescribing this class deterministically at a preliminary prime is called \emph{clamping}; the prime and its prescribed class are called a \emph{clamping prime} and a \emph{clamping class}.

For the primes in the first and third ranges, take
\begin{equation*}
a_p = 0,
\end{equation*}
and let $\mathcal{N}_{0}$ denote the set of integers $1 \le n \le H$ left uncovered after these choices. Suppose that $n \in \mathcal{N}_{0}$ has a prime divisor $q > x/3$. Since $n$ has no prime divisor at most $L$, either $n = q$, or
\begin{equation*}
n > qL > \frac{xL}{3} \ge H,
\end{equation*}
which is impossible. Hence any member of $\mathcal{N}_{0}$ having a prime factor greater than $x/3$ is itself such a prime, while every other member has all its prime factors in $(L,L^{L}]$.

Writing $\Psi(x,y)$ for the number of positive integers at most $x$ all of whose prime factors are at most $y$, we therefore have
\begin{equation*}
\#\mathcal{N}_{0} \le \Psi(H,L^{L}) + \pi(H) - \pi(x/3).
\end{equation*}
Since $L$ is fixed, an $L^{L}$-smooth integer is determined by finitely many prime exponents, each $O(\log H)$, so $\Psi(H,L^{L})$ grows at most like a fixed power of $\log H$. In particular, its contribution is negligible compared with the prime term:
\begin{equation*}
\Psi(H,L^{L}) = o\left(\frac{H}{\log H}\right).
\end{equation*}
By the prime number theorem,
\begin{equation*}
\#\mathcal{N}_{0} \le (1 + o(1))\frac{x(L - 1)}{3\log x}.
\end{equation*}

For the primes $L < p \le L^{L}$, choose the residue classes independently and uniformly at random. For any fixed $n \in \mathcal{N}_{0}$, the probability of surviving all these choices is
\begin{equation*}
\prod_{L < p \, \le \, L^{L}} \left(1 - \frac{1}{p}\right).
\end{equation*}
Hence, by linearity of expectation, there is a deterministic choice of residue classes leaving at most this proportion of $\mathcal{N}_{0}$. The greedy procedure gives the same upper bound: at each prime $p$, at most a proportion $1 - 1/p$ of the current set remains. 

These choices complete the preliminary sieve. We call the integers left uncovered its \emph{survivors}; they retain this designation even if covered later.

By Mertens' theorem,
\begin{equation*}
\prod_{L < p \, \le \, L^{L}} \left(1 - \frac{1}{p}\right) \sim \frac{1}{L}
\end{equation*}
as $L \to \infty$. Thus, on taking $L$ sufficiently large and then $x$ sufficiently large, the number of survivors of the completed preliminary sieve is less than
\begin{equation*}
\frac{x}{2\log x}.
\end{equation*}

The primes in the fourth range, $x/3 < p \le x$, have meanwhile been left in reserve, just as in the Brauer--Zeitz construction. By the prime number theorem there are asymptotically
\begin{equation*}
\frac{2x}{3\log x}
\end{equation*}
such primes. There are therefore enough reserve primes to assign a distinct one to each survivor and choose its residue class to contain that integer. We call this operation \emph{mopping up}. Thus
\begin{equation*}
Y(x) \ge H = \left\lfloor \frac{xL}{3} \right\rfloor.
\end{equation*}
Since $L$ may be taken arbitrarily large,
\begin{equation*}
\frac{Y(x)}{x} \to \infty.
\end{equation*}

In this modern form, two features of this argument recur through the Erdős--Rankin constructions and their later refinements. The smallest primes are assigned the zero residue class, and the interval length is chosen so that a survivor having a prime factor in the range of reserve primes must itself be prime; the reserve primes can then be used individually on the remaining integers. At the intermediate stage, averaging over residue classes guarantees a good deterministic choice, which may equivalently be realized by a greedy selection.

We can now make precise the correspondence between the covering function $Y(x)$, the record-gap function $G(X)$, and the older formulations in terms of individual prime gaps. Put
\begin{equation*}
P(x) := \prod_{p \, \le \, x} p.
\end{equation*}
Let $y = Y(x)$, and choose residue classes $a_p \bmod p$, for $p \le x$, whose union contains $\{1,\ldots,y\}$. By the Chinese remainder theorem there is an integer $M$ with
\begin{equation*}
x < M \le x + P(x), \qquad M \equiv -a_p \bmod p \quad (p \le x).
\end{equation*}
For each $1 \le k \le y$, some prime $p \le x$ satisfies $k \equiv a_p \bmod p$, and hence
\begin{equation*}
M + k \equiv 0 \bmod p.
\end{equation*}
Since $M > x$, the integers $M + 1,\ldots,M + y$ are therefore all composite.

Let $q$ be the first prime greater than $M + y$, and let $q'$ be the preceding prime. Then $q' \le M$, so
\begin{equation*}
q - q' > y.
\end{equation*}
By Bertrand's postulate,
\begin{equation*}
q < 2(M + y) \le 2\left(x + P(x) + Y(x)\right).
\end{equation*}
With our convention that $G(X)$ counts gaps whose upper prime is at most $X$, this gives
\begin{equation*}
G\left(2\left(x + P(x) + Y(x)\right)\right) > Y(x).
\end{equation*}

We shall also use the elementary bound
\begin{equation*}
Y(x) < P(x).
\end{equation*}
Indeed, consider any choice of residue classes $a_p \bmod p$, one for each prime $p \le x$. For each $p$, choose a residue class $b_p \bmod p$ distinct from $a_p \bmod p$. By the Chinese remainder theorem there is an integer $b$, with $1 \le b \le P(x)$, satisfying
\begin{equation*}
b \equiv b_p \bmod p \qquad (p \le x).
\end{equation*}
This integer belongs to none of the chosen residue classes. Hence no choice of the $a_p$ can contain all of $\{1,\ldots,P(x)\}$.

It follows that
\begin{equation*}
\log\left(2\left(P(x) + x + Y(x)\right)\right) \sim \log P(x) \sim x,
\end{equation*}
where the second asymptotic follows from the prime number theorem in Chebyshev's form. Thus the natural height of the gap produced at scale $x$ has logarithm asymptotic to $x$. To obtain a bound for $G(X)$ for every sufficiently large $X$, rather than only at these particular heights, we interpolate between consecutive prime values of $x$.

There is a small point here concerning the quantifiers. Apply the preceding construction with $x = p_m$, and put
\begin{equation*}
X_m := 2\left(p_m + P(p_m) + Y(p_m)\right).
\end{equation*}
Then
\begin{equation*}
\log X_m \sim p_m.
\end{equation*}
By the prime number theorem, $p_{m + 1} \sim p_m$. Since $\log X_m \sim p_m$ and $\log X_{m + 1} \sim p_{m + 1}$, it follows that, whenever $X_m \le X < X_{m + 1}$, 
\begin{equation*}
\log X \sim p_m.
\end{equation*}
Since $G(X)$ is nondecreasing,
\begin{equation*}
G(X) \ge G(X_m) > Y(p_m).
\end{equation*}
Consequently,
\begin{equation*}
\frac{Y(x)}{x} \to \infty \qquad \implies \qquad \frac{G(X)}{\log X} \to \infty.
\end{equation*}

The older formulation in terms of individual gaps follows from the same construction. This time choose the representative of the Chinese-remainder class so that
\begin{equation*}
P(x) \le M < 2P(x).
\end{equation*}
For sufficiently large $x$ we again have $M > x$, so $M + 1,\ldots,M + Y(x)$ are composite. If $p_n$ is the largest prime not exceeding $M$ and $p_{n + 1}$ is the first prime greater than $M + Y(x)$, then
\begin{equation*}
p_{n + 1} - p_n > Y(x).
\end{equation*}
By Bertrand's postulate and $Y(x) < P(x)$,
\begin{equation*}
\frac{P(x)}{2} < p_n < p_{n + 1} < 6P(x),
\end{equation*}
and hence
\begin{equation*}
\log p_n \sim x.
\end{equation*}
As $x \to \infty$ the resulting gaps become arbitrarily large, so infinitely many distinct gaps occur.

For example, Westzynthius's estimate in its $Y(x)$-form
\begin{equation*}
Y(x) \ge (2 - \epsilon)e^{\gamma} \frac{x\log_{2} x}{\log_{3} x}
\end{equation*}
therefore yields, after an arbitrarily small adjustment of $\epsilon$,
\begin{equation*}
p_{n + 1} - p_n > (2 - \epsilon)e^{\gamma} \frac{\log p_n \log_{3} p_n}{\log_{4} p_n}
\end{equation*}
for infinitely many $n$. The same interpolation at the primorial scales gives the corresponding uniform bound
\begin{equation*}
G(X) \ge (2 - \epsilon)e^{\gamma} \frac{\log X \log_{3} X}{\log_{4} X}
\end{equation*}
for all sufficiently large $X$.

\subsection{The Erdős--Rankin scale and smooth numbers}
\label{subsec:erdos-rankin-scale}
Ricci \cite{RIC1934} next combined Westzynthius's construction with estimates from Brun's sieve in a broader study of polynomial values. He described the sieve method as that of Brun and Rademacher; the estimate used in his argument is quoted from his own preceding memoir \cite[\S 14]{RIC1934}. In the special case relevant here, he claimed the improvement
\begin{equation*}
Y(x) \gg x\log_{2} x,
\end{equation*}
and consequently
\begin{equation*}
G(X) \gg (\log X)\log_{3} X.
\end{equation*}
This removes the factor $\log_{4} X$ from the denominator of Westzynthius's prime-gap bound, although the proof contains a gap in its classification of the surviving integers.

Ricci considers an integer-valued polynomial $F$ with no fixed prime divisor. His argument retains Westzynthius's combination of factorization restrictions, greedy residue selection, and individual covering with reserve primes. The additional ingredients are sieve estimates for arguments at which $F$ has no small prime factor, and estimates for the distribution of primes modulo which $F$ has a root.

The greedy step takes account of the polynomial's roots. Write $h(p)$ for the number of roots of $F$ modulo $p$. The candidate arguments at this stage avoid those roots and therefore occupy only $p - h(p)$ residue classes. For each prime $p$ used in this step, Ricci chooses a class containing as many of the remaining candidates as possible, leaving at most the proportion
\begin{equation*}
1 - \frac{1}{p - h(p)}.
\end{equation*}
A Chinese-remainder translation then moves each selected class onto a root class of $F$ modulo the corresponding prime \cite[pp.~202--203]{RIC1934}.

The difficulty occurs in passing from this selection to a bound for all the survivors \cite[pp.~203--204]{RIC1934}. Ricci distinguishes values composed entirely of primes less than or equal to a lower cutoff from values composed entirely of primes greater than or equal to an upper cutoff. These do not exhaust the values having no prime factor between the two cutoffs: a value may contain both small and large prime factors. Such mixed values are included in Westzynthius's large-prime-factor class, but are absent from Ricci's count. Thus the stated bound for the number of integers requiring individual covering does not follow from the argument as written.\footnote{This omission was identified by GPT-6~Astra during a review of Ricci's argument and subsequently checked by the author.}

Erd\H{o}s \cite{ERD1935} obtained the substantially stronger bound
\begin{equation*}
Y(x) \gg \frac{x\log x}{(\log_{2} x)^2},
\end{equation*}
and hence
\begin{equation*}
G(X) \gg \frac{(\log X)(\log_{2} X)}{(\log_{3} X)^2}.
\end{equation*}
Although Erd\H{o}s described his argument as an increase in the precision of the Brauer--Zeitz method, structurally it already resembles the later Erd\H{o}s--Rankin construction: the primes are divided into ranges with different roles, the surviving integers are separated according to their factorization, and the final exceptional set is made small enough to be covered one integer at a time.

Put
\begin{equation*}
z := x^{1/(20\log_{2} x)}
\end{equation*}
and consider an interval of length
\begin{equation*}
H := \frac{cx\log x}{(\log_{2} x)^2},
\end{equation*}
where $c > 0$ is sufficiently small. Split the primes $p \le x$ into the four ranges
\begin{equation*}
p \le \log x, \qquad
\log x < p \le z, \qquad
z < p \le \frac{x}{2}, \qquad
\frac{x}{2} < p \le x.
\end{equation*}
The resemblance to the preceding qualitative construction is already visible. The fixed cutoff $L$ is replaced by the growing cutoff $\log x$, while the fixed second cutoff $L^{L}$ is replaced by $z$. The four ranges otherwise play essentially the same roles: zero residue classes at the first and third stages, a sieve acting on the primes left uncovered in the second, and a final range of primes held in reserve. The new difficulty is that the first two cutoffs now grow with $x$, so the smooth survivors require a genuine uniform estimate.

For the first and third ranges Erd\H{o}s takes the zero residue class. For the second range he takes
\begin{equation*}
a_p = -1,
\end{equation*}
while the primes in the fourth range are left in reserve for the final exceptional set.

After the zero residue classes have been chosen, an uncovered composite integer $n \le H$ either has all its prime factors in the second range, or else would have a prime factor exceeding $x/2$. The latter is impossible unless $n$ is itself prime: any complementary factor would exceed $\log x$, and hence
\begin{equation*}
n > \frac{x}{2}\log x > H
\end{equation*}
for sufficiently large $x$.

Thus the remaining integers consist essentially of two types: integers composed entirely of primes in the second range, and primes exceeding $x/2$. Erd\H{o}s treats these two sets separately.

The primes left uncovered are treated using the second range. For every prime $\log x < p \le z$, Erd\H{o}s takes $a_p = -1$, so a prime $n$ is covered whenever
\begin{equation*}
n \equiv -1 \bmod p.
\end{equation*}
Thus the primes that survive this stage are precisely those for which $n + 1$ has no prime divisor in $(\log x,z]$. Heuristically, the proportion of primes escaping all these residue classes is of order
\begin{equation*}
\prod_{\log x < p \, \le \, z} \left(1 - \frac{1}{p - 1}\right)
\asymp \frac{\log_{2} x}{\log z}
\asymp \frac{(\log_{2} x)^2}{\log x},
\end{equation*}
where the final comparison uses Erd\H{o}s's choice of $z$. The product has the same order as the reduction obtained by a greedy choice of residue classes; Brun's sieve shows that, for this structured set of primes, the fixed choice $a_p = -1$ performs essentially as well.

The choice of $H$ is dictated by the number of exceptions that must be left for the reserve primes. Brun's sieve gives, schematically, at most
\begin{equation*}
\ll H\frac{(\log_{2} x)^2}{(\log x)^2}
\end{equation*}
prime exceptions. The fourth range contains $\asymp x/\log x$ primes, so in order that one unused prime may be assigned to each remaining integer it is enough to require
\begin{equation*}
H\frac{(\log_{2} x)^2}{(\log x)^2} \ll \frac{x}{\log x}.
\end{equation*}
This leads naturally to
\begin{equation*}
H \asymp \frac{x\log x}{(\log_{2} x)^2}.
\end{equation*}

The smooth-number estimate required here is still elementary. Erd\H{o}s divides the relevant integers according to whether they have at most or more than $10\log_{2} x$ distinct prime factors; in the first case they can be counted directly from their possible prime-power factors, while in the second case their divisor function is so large that a divisor-sum estimate makes them negligible. No uniform theory of smooth numbers is yet needed.

Rankin \cite{RAN1938} sharpened Erd\H{o}s's construction by allowing the smoothness threshold to grow much further. The elementary device he used, now known as Rankin's trick,\footnote{The device predates Rankin's paper: Ramar\'e \cite{RAM2022} points to its use by Heilbronn and Landau \cite[p.~10]{HL1933} and discusses its role in sieve methods.} bounds the number $\Psi(H,z)$ of $z$-smooth integers up to $H$ by
\begin{equation*}
\Psi(H,z) \le H^{\sigma} \prod_{p \, \le \, z} \left(1 - p^{-\sigma}\right)^{-1} \qquad (\sigma > 0).
\end{equation*}
Indeed, if $n \le H$, then $1 \le (H/n)^{\sigma}$; summing this inequality over the $z$-smooth integers and then removing the condition $n \le H$ produces the Euler product.

In the range needed for the construction, estimating this Euler product gives
\begin{equation*}
\Psi(H,z) \ll H(\log z)\exp(-u\log u + u), \qquad u := \frac{\log H}{\log z}.
\end{equation*}
Lemma~\ref{lem:rankin-smooth-numbers} states the precise hypotheses and gives a proof following Constantinescu \cite[Lemma~2.3]{CON2018}, using Mertens' estimates and convexity.

The point of this estimate is its uniformity as the smoothness threshold $z$ grows with $x$. It permits the second cutoff to be pushed to the scale
\begin{equation*}
\log z \asymp \frac{(\log x)(\log_{3} x)}{\log_{2} x},
\end{equation*}
while keeping the smooth-number contribution small enough for the final assignment step. Rankin proved, more precisely, that for every $\epsilon > 0$ and all sufficiently large $n$,
\begin{equation*}
Y(p_n) \ge
\left(\frac{1}{3} - \epsilon\right)
\frac{p_n(\log p_n)(\log_{3} p_n)}
     {(\log_{2} p_n)^2}.
\end{equation*}
Consequently, in the notation used here,
\begin{equation*}
Y(x) \ge
\left(\frac{1}{3} - \epsilon\right)
\frac{x(\log x)(\log_{3} x)}
     {(\log_{2} x)^2}
\end{equation*}
for all sufficiently large $x$, after adjusting $\epsilon$, and hence
\begin{equation*}
G(X) \ge
\left(\frac{1}{3} - \epsilon\right)
\frac{(\log X)(\log_{2} X)(\log_{4} X)}
     {(\log_{3} X)^2}
\end{equation*}
for all sufficiently large $X$.

The next advances left the shape of Rankin's bound unchanged and improved only its constant. Sch\"onhage \cite{SCH1963} replaced Rankin's constant $1/3$ by $e^{\gamma}/2$ through a more efficient selection of the prime moduli. In his construction the target interval has length
\begin{equation*}
H := \frac{c x(\log x)(\log_{3} x)}{(\log_{2} x)^2},
\end{equation*}
while the smoothness threshold is taken to be
\begin{equation*}
z := \exp\left(\frac{\theta(\log x)(\log_{3} x)}{\log_{2} x}\right).
\end{equation*}
The smooth-number estimate then available restricted Sch\"onhage to $\theta < 1/2$.

Rankin \cite{RAN1963} observed that de Bruijn's estimates \cite{DEB1951a, DEB1951b} allow the same argument to be run with any fixed $0 < \theta < 1$. These estimates involve Dickman's function $\rho$, the continuous function on $[0,\infty)$ determined by
\begin{equation*}
\rho(u) = 1 \quad(0 \le u \le 1), \qquad u\rho'(u) = -\rho(u - 1) \quad (u > 1).
\end{equation*}
For fixed $u > 0$, $\rho(u)$ is the limiting proportion of positive integers up to $t$ whose prime factors are all at most $t^{1/u}$, as $t \to \infty$ \cite[Theorem~7.2]{MV2006}. De Bruijn's estimates also apply when the smoothness parameter grows. In the range required here, writing
\begin{equation*}
u := \frac{\log H}{\log z},
\end{equation*}
they give
\begin{equation*}
\Psi(H,z) \sim H\rho(u), \qquad \log\rho(u) \sim -u\log u.
\end{equation*}
For the choices above,
\begin{equation*}
u \sim \frac{\log_{2}x}{\theta\log_{3}x},
\end{equation*}
and hence
\begin{equation*}
\rho(u) = (\log x)^{-1/\theta + o(1)}.
\end{equation*}
Thus
\begin{equation*}
\Psi(H,z) < \frac{H}{\log x}
\end{equation*}
for every fixed $0 < \theta < 1$ and all sufficiently large $x$. The remaining steps of Sch\"onhage's argument are unchanged, and letting $\theta$ approach $1$ raises the constant from $e^{\gamma}/2$ to $e^{\gamma}$.

\subsection{The classical construction in outline}
\label{subsec:classical-construction-outline}

We may now collect the main ingredients into a single construction, leaving the parameters unspecified. Given a prime cutoff $x$, we seek to cover $\{1,\ldots,H\}$ by one residue class modulo each prime $p \le x$, with $H$ as large as possible. A successful choice proves $Y(x) \ge H$. Choose cutoffs
\begin{equation*}
2 \le w < z < v < x,
\end{equation*}
and divide the available primes into four ranges:
\begin{equation*}
p \le w,\qquad
w < p \le z,\qquad
z < p \le v,\qquad
v < p \le x.
\end{equation*}
The primes up to $v$ form the preliminary sieve, while those in $(v,x]$ are held in reserve for mopping up. For the moment, suppose that
\begin{equation*}
H \le wv.
\end{equation*}

Clamp the primes in the first and third ranges by prescribing $a_p = 0$. Every prime factor of an uncovered integer then lies in $(w,z]$ or exceeds $v$. If an uncovered integer $n \le H$ has a prime factor $q > v$, write $n = qm$. Were $m > 1$, its prime factors would exceed $w$, giving $n > vw \ge H$. Thus $n = q$. The uncovered integers therefore consist of those whose prime factors all lie in $(w,z]$, together with primes in $(v,H]$; the first group includes $1$ and contains at most $\Psi(H,z)$ integers.

The second range reduces the number of uncovered integers further. There are two natural ways to make the choices. The first treats all uncovered integers together. For each prime $w < p \le z$, choose a residue class containing as many of the integers still remaining as possible. At least a proportion $1/p$ is removed at each step. Writing $\mathcal{N}$ for the survivor set of the completed preliminary sieve, we obtain, when $H \ge v$,
\begin{equation*}
\#\mathcal{N} \le \left(\Psi(H,z) + \pi(H) - \pi(v)\right) \prod_{w \, < \, p \, \le \, z}\left(1 - \frac1p\right).
\end{equation*}
The same bound follows by choosing the residue classes independently and uniformly at random: each integer left uncovered after clamping remains uncovered with probability equal to the displayed product, so some choice leaves no more than the expected number.

Alternatively, one may direct the second stage at the uncovered primes alone. These are all greater than $v$, and hence occupy only nonzero residue classes modulo each prime $p \le z$. A greedy choice among those classes removes at least a proportion $1/(p - 1)$ of the primes still remaining. Leaving the smooth contribution unreduced in the bound gives instead
\begin{equation*}
\#\mathcal{N} \le \Psi(H,z) + \left(\pi(H) - \pi(v)\right) \prod_{w \, < \, p \, \le \, z}\left(1 - \frac{1}{p - 1}\right).
\end{equation*}
Here the corresponding probabilistic choice is uniform among the nonzero residue classes. The improvement in the reduction of the prime survivors comes at the cost of leaving the smooth contribution unreduced in the bound.

Both versions suffice to obtain the classical Erd\H{o}s--Rankin order, with suitable choices of parameters. As in Erd\H{o}s's argument, a suitable sieve estimate can also justify a fixed choice of nonzero residue classes.

The primes in $(v,x]$ have been held in reserve. If
\begin{equation*}
\#\mathcal{N} \le \pi(x) - \pi(v),
\end{equation*}
assign each remaining integer $n$ a distinct reserve prime $p$ and choose $a_p \equiv n \bmod p$. Every integer is now covered.

The parameters balance the two contributions to the survivor bound. Increasing $z$ supplies more primes for the second stage, but also enlarges the smooth-number contribution. Increasing $w$ permits a larger $H$ under the condition $H \le wv$, but shortens the second range. Increasing $v$ likewise relaxes that condition, while leaving fewer primes in reserve. The condition itself is sufficient to make every nonsmooth survivor prime; it is not intrinsic to the covering problem. Without it, additional composite survivors must be taken into account.

This outline brings together features of the preceding arguments: the separation by factorization and final individual covering in the qualitative construction, the greedy selection used by Ricci, the growing cutoffs and sieve estimate in Erd\H{o}s's argument, and the smooth-number estimates sharpened by Rankin and de Bruijn. The classical final step requires only that the survivors be few enough. The next developments exploit their arrangement to cover more than one with each reserve prime.

\subsection{Short intervals and more efficient use of the reserve primes}
\label{subsec:efficient-use-reserve-primes}
A related development concerns intervals containing unusually few or unusually many primes, rather than none at all. Maier \cite{MAI1985} showed that the prime-number-theorem heuristic cannot hold uniformly even for intervals of polylogarithmic length. More precisely, for every fixed $\lambda > 1$ there is a constant $\eta = \eta(\lambda) > 0$ such that, for arbitrarily large $x$,
\begin{equation*}
\pi\left(x + (\log x)^\lambda\right) - \pi(x)
<
(1 - \eta)(\log x)^{\lambda - 1},
\end{equation*}
while for arbitrarily large $x$ the reverse fluctuation
\begin{equation*}
\pi\left(x + (\log x)^\lambda\right) - \pi(x)
>
(1 + \eta)(\log x)^{\lambda - 1}
\end{equation*}
also occurs.

By this time the shape of Rankin's bound had resisted improvement for nearly half a century. Erd\H{o}s \cite{ERD1986} offered \$10,000 for a proof that the constant in Rankin's bound could be taken arbitrarily large.

Maier and Pomerance \cite{MP1990} introduced a new refinement of the final use of the reserve primes. Their arithmetic input, an averaged form of a generalized twin-prime problem, supplies many pairs of surviving integers congruent modulo a reserve prime $p$. Finding such pairs separately for many primes is not enough: the same integers might occur repeatedly. A graph-theoretic matching argument allows disjoint pairs to be assigned to distinct reserve primes, for a positive proportion of those primes. Each selected prime therefore covers two previously uncovered integers. This raised the constant in Rankin's bound from $e^{\gamma}$ to
\begin{equation*}
c_{0}e^{\gamma},
\end{equation*}
where $c_{0} = 1.31256\ldots$ is the positive solution of
\begin{equation*}
\frac{4}{c_{0}} - e^{-4/c_{0}} = 3.
\end{equation*}

Pintz \cite{PIN1997} strengthened the combinatorial step by a probabilistic argument, obtaining an essentially optimal matching in which almost all of the relevant reserve primes cover two survivors. He thereby raised the constant further to $2e^{\gamma}$.

Maier and Stewart \cite{MS2007} later gave a result interpolating between these short-interval irregularities and the classical construction of a prime-free interval. In particular, for every $\epsilon > 0$ there are arbitrarily large $x$ for which the interval $(x,x + \log x\log_{2} x]$ contains at most
\begin{equation*}
(1 + \epsilon)\frac{(\log_{3}x)^2}{\log_{4}x}
\end{equation*}
primes. Consequently, after an arbitrarily small adjustment of $\epsilon$, the prime gaps meeting such an interval have average length at least
\begin{equation*}
(1 - \epsilon)\frac{(\log x)(\log_{2}x)(\log_{4}x)}{(\log_{3}x)^2}.
\end{equation*}
In particular, at least one of these gaps has Erd\H{o}s--Rankin size. Thus the classical large-gap phenomenon sits inside a broader irregularity: on suitable intervals substantially longer than a single exceptional gap, primes may occur with an average spacing of the same order as the largest gaps produced by the Erd\H{o}s--Rankin construction.

\subsection{The 2014 breakthrough: many survivors per prime}
\label{subsec:many-survivors-per-prime}
In 2014, seventy-six years after Rankin's paper, the order of his bound was finally surpassed. Maynard \cite{MAY2016} and, independently, Ford, Green, Konyagin and Tao \cite{FGKT2016} proved that the constant multiplying the Erd\H{o}s--Rankin scale can be made arbitrarily large. Both extended the principle of Maier and Pomerance beyond pairs, using residue classes containing many surviving primes. Their arithmetic inputs differed: Ford, Green, Konyagin and Tao used results on arithmetic progressions of primes, while Maynard used multidimensional sieve weights to find classes containing many primes.

Later that year, Ford, Green, Konyagin, Maynard and Tao combined these ideas with an efficient hypergraph covering argument \cite{FGKMT2018}. They proved
\begin{equation*}
G(X) \gg
\frac{(\log X)(\log_{2} X)(\log_{4} X)}
     {\log_{3} X}.
\end{equation*}
Their construction gives, at the level of the finite covering problem,
\begin{equation*}
Y(x) \gg
\frac{x(\log x)(\log_{3} x)}
     {\log_{2} x}.
\end{equation*}
Here the combinatorial selection becomes decisive. Regard the survivors as vertices and the sets covered by individual residue classes as edges of a hypergraph. One must choose at most one class for each reserve prime, while keeping overlaps small enough that most of the coverage is useful. The authors achieve this by adapting the R\"odl nibble method underlying a theorem of Pippenger and Spencer \cite{PS1989}: classes are selected in successive small random rounds, with later choices adjusted to account for earlier coverage. Their version accommodates classes containing different numbers of survivors and reduces the overlap losses in the earlier independent random selections.

\subsection{Structured survivors and short translates}
\label{subsec:structured-survivors}
Two recent unrefereed manuscripts, discussed on the Erdős Problems forum \cite{EP4DISC}, claim further improvements, and both depart more sharply from the preceding Erdős--Rankin strategy. Two features had persisted through the earlier constructions. First, the preliminary sieve was designed largely to leave as few survivors as possible, with composite survivors treated as an exceptional set to be made negligible. Second, the construction remained within the finite covering problem $Y(x)$: every divisibility condition ultimately came from one chosen residue class modulo a prime $p \le x$.

The first manuscript, attributed to GPT-5.6~Sol \cite{GPT2026}, relaxes the first of these principles. Its preliminary random sieve gives a bias to the zero residue, chosen so that the survival probabilities of rough composites retain a particularly simple multiplicative structure. This makes it possible to treat the surviving composites as a resource rather than merely as an exceptional set to be eliminated: they can be organized into residue-class blocks and many of them covered at once. The surviving primes are handled separately using the Maynard weight and the hypergraph covering machinery of Ford, Green, Konyagin, Maynard and Tao. In this sense, the number of preliminary survivors is no longer the only quantity to optimize: a somewhat larger but more structured survivor set may be easier to cover than a smaller unstructured one. The manuscript claims
\begin{equation*}
Y(x) \gg \frac{x\log x}{\log_{3}x}, \qquad G(X) \gg \frac{(\log X)(\log_{2}X)}{\log_{4}X}.
\end{equation*}

Theorem~\ref{thm:tilted-covering-bound} isolates the contribution of the tilt and its associated covering of composite survivors. Covering the surviving primes individually already gives
\begin{equation*}
Y(x) \gg \frac{x\log x}{(\log_{2} x)\log_{3} x},
\end{equation*}
which exceeds the classical Erd\H{o}s--Rankin scale by a factor of $(\log_{2} x)/(\log_{3} x)^{2}$, although it remains below the Ford--Green--Konyagin--Maynard--Tao scale. Thus the structured treatment of composite survivors yields an unbounded improvement over the classical bound without Maynard sieve weights or a hypergraph covering theorem. This separates its contribution from that of covering prime survivors in groups.

A second manuscript released by OpenAI \cite{OAI2026}, which attributes the proof to GPT-6~Astra, breaks the second restriction instead. It begins with an essentially classical Erd\H{o}s--Rankin sieve using primes up to $x$, leaving a sparse set $S$ of uncovered integers. After fixing the resulting congruences modulo $P(x)$, it varies the resulting interval through translates by multiples of $P(x)$. These translations preserve all of the divisibility conditions already imposed by the primes $p \le x$, while auxiliary primes larger than $x$ are used to make every remaining integer composite simultaneously. The manuscript thereby claims
\begin{equation*}
G(X) \gg
\frac{(\log X)(\log_{2}X)^2(\log_{4}X)}
     {(\log_{3}X)^2}.
\end{equation*}
This is therefore not an improved lower bound for the classical covering function $Y(x)$. Sol remains within that covering problem but changes the structure of the survivor set; Astra retains a more classical preliminary sieve but escapes the restriction to primes at most $x$ by exploiting an additional translation parameter and auxiliary larger primes.

\subsection{Conjectures, computation, and quantitative summary}
\label{subsec:cramer-numerics}
The results surveyed above give lower bounds for the largest prime gaps. A conjectural benchmark is the quadratic-logarithmic scale: Cram\'er's conjecture is commonly stated as the upper bound
\begin{equation*}
G(X) \ll (\log X)^2,
\end{equation*}
and this scale is widely expected to describe the true order of magnitude of $G(X)$. Even the preceding claimed improvements remain far below it.

The leading constant is less clear. Cram\'er's probabilistic model \cite{CRA1936} suggests the stronger prediction
\begin{equation*}
\limsup_{X \to \infty} \frac{G(X)}{(\log X)^2} = 1.
\end{equation*}
The model selects each integer $n \ge 3$ independently with probability $1/\log n$ to represent a prime. It captures many features of prime spacings remarkably well, but it does not fully reflect the congruence restrictions imposed by the small primes. Granville \cite{GRA1995} showed that taking those restrictions into account changes the prediction for the most extreme gaps: the same heuristic reasoning points instead to infinitely many gaps of size at least
\begin{equation*}
\left(2e^{-\gamma} + o(1)\right)(\log p)^2.
\end{equation*}
Here
\begin{equation*}
e^{-\gamma} = 0.561459\ldots, \qquad 2e^{-\gamma} = 1.122918\ldots.
\end{equation*}
This challenges the constant $1$ in the stronger prediction, while remaining compatible with a quadratic-logarithmic order of magnitude.

A conditional route to gaps of this order is supplied by Banks, Ford and Tao \cite{BFT2023}. They show that a sufficiently strong averaged form of the Hardy--Littlewood prime tuples conjecture implies
\begin{equation*}
G(X) \gg (\log X)^2.
\end{equation*}
Their argument deduces the existence of large gaps from estimates for prime tuples, without first obtaining a corresponding lower bound for $Y(x)$.

Their probabilistic model also relates the predicted size of the largest gaps to an extremal interval-sieve problem. Under a further conjecture about that problem, it predicts
\begin{equation*}
G(X) \sim 2e^{-\gamma}(\log X)^2.
\end{equation*}
The precise asymptotic therefore remains uncertain.

A striking numerical example was reported by Nyman and Nicely \cite{NN2003}: the gap of length $1132$ following the prime $p = 1693182318746371$ gives
\begin{equation*}
\frac{1132}{(\log p)^2}  =0.9206385885\ldots.
\end{equation*}
Even this gap falls short of $(\log p)^2$, and further still of $2e^{-\gamma}(\log p)^2$. Such computations illustrate the scale of the conjectures, though finite data cannot decide between their asymptotic predictions.

The enormous distance between what is proved, what has been observed computationally, and what is conjectured is one reason the subject has remained fertile. The problem begins with the elementary observation that consecutive composite numbers can be manufactured by congruences, but making those congruences efficient enough to produce a record gap has repeatedly required new ideas from sieve theory, the distribution of primes, probabilistic combinatorics, and the arithmetic of smooth numbers.

The progression of the main quantitative bounds is summarized in Table~\ref{tab:large-prime-gap-history}. To keep the formulas compact, put
\begin{equation*}
\mathcal{R}(t) := \frac{t(\log t)(\log_{3} t)}{(\log_{2} t)^2}.
\end{equation*}
(Here $\mathcal{R}$ denotes a function, unrelated to the set of reserve primes denoted by $\mathcal{R}$ in the main text.) Thus $\mathcal{R}(x)$ is the classical Rankin scale for the covering function, while
\begin{equation*}
\mathcal{R}(\log X) = \frac{(\log X)(\log_{2} X)(\log_{4} X)}{(\log_{3} X)^2}
\end{equation*}
is the corresponding Erd\H{o}s--Rankin scale for prime gaps.

\clearpage
\begin{landscape}
\thispagestyle{empty}
\landscapepagenumber

\makeatletter
\setlength{\@fptop}{0pt}
\setlength{\@fpbot}{0pt plus 1fil}
\makeatother

\begin{table}[p]
\vspace*{-15mm}
\centering
\small
\setlength{\tabcolsep}{4pt}
\renewcommand{\arraystretch}{1.2}

\caption{Successive lower bounds for long gaps between consecutive primes and for the associated covering problem.}
\label{tab:large-prime-gap-history}

\begin{tabularx}{\linewidth}{
  @{}
  r
  L{0.185\linewidth}
  L{0.185\linewidth}
  L{0.225\linewidth}
  Z
  @{}
}
\toprule
Year
&
Author(s)
&
$Y(x)$
&
$G(X)$
&
Main advance
\\
\midrule

1830
&
Legendre \cite{LEG1830}
&
$\ge (2-\epsilon)x$
&
---
&
Constructed $Y(p_n)\ge 2p_{n-1}-1$; conjectured this essentially optimal.
\\

1929
&
Backlund \cite{BAC1929}
&
$\ge (2-\epsilon)x$
&
$\ge (2-\epsilon)\log X$
&
Combined Legendre's covering bound with the PNT; first factor $>1$
beyond the average gap.
\\

1930
&
Brauer--Zeitz \cite{BZ1930,ZEI1930}
&
$\ge (4-\epsilon)x$
&
$\ge (4-\epsilon)\log X$
&
Uniform counterexamples to Legendre’s proposed extremality; improved $2$ to $4$.
\\

1931
&
Westzynthius \cite{WES1931}
&
$\ge (2e^{\gamma} - \epsilon)x\log_{2} x / \log_{3} x$
&
$\ge (2e^{\gamma} - \epsilon)(\log X)(\log_{3} X)/\log_{4} X$
&
Greedy sieving with a separate smooth-number estimate.
\\

1934
&
Ricci \cite{RIC1934} (claimed)
&
$\gg x\log_{2} x$
&
$\gg (\log X)(\log_{3} X)$
&
Polynomial extension using Brun's sieve.
\\

1935
&
Erd\H{o}s \cite{ERD1935}
&
$\gg x\log x/(\log_{2} x)^2$
&
$\gg (\log X)(\log_{2} X)/(\log_{3} X)^2$
&
Prime ranges, Brun sieve and sharper control of smooth survivors.
\\

1938
&
Rankin \cite{RAN1938}
&
$\ge \left(\frac{1}{3}-\epsilon\right)\mathcal{R}(x)$
&
$\ge \left(\frac{1}{3}-\epsilon\right)\mathcal{R}(\log X)$
&
Rankin's trick for bounding smooth numbers.
\\

1963
&
Sch\"onhage \cite{SCH1963}
&
$\ge \left(\frac{e^{\gamma}}{2}-\epsilon\right)\mathcal{R}(x)$
&
$\ge \left(\frac{e^{\gamma}}{2}-\epsilon\right)\mathcal{R}(\log X)$
&
More efficient selection of the prime moduli.
\\

1963
&
Rankin \cite{RAN1963}
&
$\ge \left(e^{\gamma}-\epsilon\right)\mathcal{R}(x)$
&
$\ge \left(e^{\gamma}-\epsilon\right)\mathcal{R}(\log X)$
&
de Bruijn's sharper smooth-number estimates.
\\

1990
&
Maier--Pomerance \cite{MP1990}
&
$\ge \left(c_0 e^{\gamma}-\epsilon\right)\mathcal{R}(x)$
&
$\ge \left(c_0 e^{\gamma}-\epsilon\right)\mathcal{R}(\log X)$
&
Pairs of survivors covered by one large prime.
\\

1997
&
Pintz \cite{PIN1997}
&
$\ge \left(2e^{\gamma}-\epsilon\right)\mathcal{R}(x)$
&
$\ge \left(2e^{\gamma}-\epsilon\right)\mathcal{R}(\log X)$
&
Near-perfect matching of survivors.
\\

2016
&
Maynard \cite{MAY2016}
&
$Y(x)/\mathcal{R}(x) \to \infty$
&
$G(X)/\mathcal{R}(\log X) \to \infty$
&
Maynard sieve weights in the final covering stage.
\\

2016
&
Ford--\allowbreak Green--\allowbreak Konyagin--\allowbreak Tao \cite{FGKT2016}
&
$Y(x)/\mathcal{R}(x) \to \infty$
&
$G(X)/\mathcal{R}(\log X) \to \infty$
&
Linear configurations of primes.
\\

2018
&
Ford--\allowbreak Green--\allowbreak Konyagin--\allowbreak
Maynard--\allowbreak Tao \cite{FGKMT2018}
&
$\gg \mathcal{R}(x)\log_{2} x$
&
$\gg \mathcal{R}(\log X)\log_{3} X$
&
Explicit iterated-logarithmic gain; efficient hypergraph covering.
\\

2026
&
GPT-5.6~Sol \cite{GPT2026}\textsuperscript{\(\dagger\)}
&
$\gg
\mathcal{R}(x)
\left(\log_{2} x / \log_{3} x\right)^2$
&
$\gg
\mathcal{R}(\log X)
\left(\log_{3} X / \log_{4} X \right)^2$
&
Tilted preliminary sieve; separate composite and prime layers.
\\

2026
&
OpenAI (GPT-6~Astra)
\cite{OAI2026}\textsuperscript{\(\dagger\)}
&
---
&
$\gg \mathcal{R}(\log X)\log_{2} X$
&
Short-translates theorem replaces the one-prime-per-survivor final stage.
\\

\bottomrule
\end{tabularx}

\medskip

\begin{minipage}{\linewidth}
\footnotesize
The years in the table are publication years; the Maynard and Ford--Green--Konyagin--Tao results, as well as the subsequent five-author work, first appeared as preprints in 2014. All asymptotic bounds in the table are understood to hold for all sufficiently large $x$ or $X$, as appropriate. Here $\epsilon > 0$ is arbitrary, $\gamma = 0.577215\ldots$ is the Euler--Mascheroni constant ($e^{\gamma} = 1.781072\ldots$), and $c_0 = 1.312560\ldots$ is the positive solution of $(4/c_0) - e^{-4/c_0} = 3$. The $Y(x)$ column uses the modern covering-function formulation; this was not necessarily the formulation of the original results. In particular, the 1929 covering bound is Legendre's earlier construction, which Backlund transferred to prime gaps using the prime number theorem. The symbol $\dagger$ marks an unrefereed claim. The 2026 Astra result is left blank in the $Y(x)$ column because its short-translates step uses additional arithmetic beyond the classical problem of choosing one residue class modulo each prime $p \le x$.
\end{minipage}

\end{table}

\end{landscape}

\clearpage 

\section{AI assistance, provenance and verification}
\label{app:ai-provenance}

The theorem, proof and much of the exposition in this paper were generated through an extended dialogue with ChatGPT (OpenAI; model displayed as GPT-6~Astra), drawing on the tilted sieve of GPT-5.6~Sol \cite{GPT2026}. The author posed the questions, supplied source material, worked through the proposed arguments, and shaped the exposition. He takes responsibility for the paper but does not claim the theorem or proof as his own mathematical discovery. The purpose of the paper is to make this part of the tilted sieve accessible to other readers.

\subsection{Mathematical development}
\label{subsec:provenance-of-research-contributions}

This paper grew out of discussions for a related project, still in preparation, concerning the classical Erd\H{o}s--Rankin construction and the use of probability distributions to shape the integers left by a preliminary sieve. The author asked what could be obtained from GPT-5.6~Sol's tilt and covering of composite survivors without using Maynard sieve weights or a hypergraph covering theorem. The request specified the bound stated in Theorem~\ref{thm:tilted-covering-bound} and a proposed strategy developed in the preceding discussions. GPT-6~Astra supplied a detailed proof and adapted the construction to establish that bound.

The underlying power tilt and inverse-probability selection of composite groups come from \cite{GPT2026}. The model adapted these ideas to a shorter interval, chose parameters, and supplied the estimates and the assembly of the covering. At this interval length, whole residue classes of composite candidates can serve as groups, and prime survivors can be covered individually. These contributions are credited to the model as an adaptation of GPT-5.6~Sol's method; no claim of historical novelty rests on the model's description of its output.

The author subsequently worked through the argument in detail. He requested explanations of probability identities, local expansions, divisor counts, correlation estimates, and the choice of parameters, and checked calculations as the exposition developed. He identified ambiguities and missing explanations, revised terminology and notation, and directed the organization of the proof. The model supplied successive drafts and elaborations in response. The finished presentation therefore incorporates many exchanges beyond the initial proposed proof.

\subsection{Exposition and appendices}
\label{subsec:scope-of-assistance}

The same assistance was used in preparing the historical survey, the expositions of the classical Erd\H{o}s--Rankin construction, and the general probabilistic framework. It included locating and comparing references, discussing mathematical arguments, drafting proofs and explanatory passages, suggesting notation and organization, and identifying possible errors or omissions. Its role extended well beyond language editing.

In reviewing the historical sources, GPT-6~Astra also identified the omitted mixed-factor case in Ricci's survivor count \cite[pp.~203--204]{RIC1934}, discussed in the historical appendix. The author subsequently checked this observation against Ricci's argument.

The author supplied sources, questioned proposed accounts and arguments, selected the material to retain, and revised the text to suit the intended audience. Descriptions of established results are attributed to the mathematical sources cited in the paper. This declaration was also drafted with AI assistance.

\subsection{Verification and responsibility}
\label{subsec:verification-and-responsibility}

The author's checking included working through the calculations and deductions in the main proof and requesting intermediate steps where the reasoning was unclear. Model-assisted discussion formed part of that process. Agreement by a model with an argument is not itself evidence of correctness; the justification for the results is the proof presented in the paper.

No language model is listed as an author. The named author takes full responsibility for the mathematical claims, exposition, citations, and other contents of the paper. This responsibility is distinct from the attribution of the mathematical development described above.

\subsection{Supporting records}
\label{subsec:supporting-records}

The ancillary files \texttt{prompt.txt} and \texttt{response.txt} contain plain-text extracts of the initial request for the simplified construction and the model's response. Editorial notes describe the changes made to prepare these extracts, including the omission of formatting instructions and references to unpublished work. The response retains its original mathematical content and presentation, apart from the removal of references to lemmas outside the extract. No corrections or improvements from the subsequent revisions have been incorporated.

The request followed earlier mathematical discussion and already specified the target bound and a proposed proof strategy. These files therefore document a starting point for the present exposition, not the full development of its underlying ideas. The extended discussions, checking, and revisions that followed are not reproduced. Readers may use the prompt to explore similar questions themselves, but the outcome depends on the model, the supplied context, and the particular run; the files do not establish that the same argument can be regenerated from the prompt alone.


\begin{thebibliography}{99}

\bibitem{BAC1929}
Backlund,~R.~J.
``\"Uber die Differenzen zwischen den Zahlen, die zu den $n$ ersten Primzahlen teilerfremd sind.''
{\em Ann.~Acad.~Sci.~Fenn.~Ser.~A}\ 32(2):1--9, 1929.
\href{https://zbmath.org/55.0687.01}{JFM:55.0687.01}

\bibitem{BFT2023}
Banks,~W., K.~Ford and T.~Tao.
``Large prime gaps and probabilistic models.''
{\em Invent.~Math.}\ 233(3):1471--1518, 2023.
\DOI{10.1007/s00222-023-01199-0}

\bibitem{BZ1930}
Brauer,~A. and H.~Zeitz.
``\"Uber eine zahlentheoretische Behauptung von Legendre.''
{\em Sitzungsber.~Berliner Math.~Ges.}\ 29:116--125, 1930.
\href{https://zbmath.org/56.0156.02}{JFM:56.0156.02}

\bibitem{BRU1920}
Brun,~V.
``Le crible d'Eratosthène et le théorème de Goldbach.''
{\em Videnskapsselskapets Skrifter. I. Mat.-naturv. Klasse}, 1920, no.~3, pp.~1--36.
\url{https://archive.org/details/lecriblederatost00brun}

\bibitem{CON2018}
Constantinescu,~P.
``Large gaps between primes.''
First-year project, London School of Geometry and Number Theory, 2018, 28~pp.
\href{https://petruconstantinescu.github.io/large_gaps.pdf}{\nolinkurl{petruconstantinescu.github.io/large_gaps.pdf}}

\bibitem{CRA1936}
Cram\'er,~H.
``On the order of magnitude of the difference between consecutive prime numbers.''
{\em Acta Arith.}\ 2(1):23--46, 1936.
\DOI{10.4064/aa-2-1-23-46}

\bibitem{DEB1951a}
de Bruijn,~N.~G.
``The asymptotic behaviour of a function occurring in the theory of primes.''
{\em J.~Indian Math.~Soc.~\textup{(}N.S.\textup{)}} 15:25--32, 1951.

\bibitem{DEB1951b}
de Bruijn,~N.~G.
``On the number of positive integers $\le x$ and free of prime factors $> y$.''
{\em Proc.~Kon.~Ned.~Akad.~Wetensch.~Ser.~A}\ 54(1):50--60, 1951.
\DOI{10.1016/S1385-7258(51)50008-2}

\bibitem{DIR1837}
Dirichlet,~P.~G.~L.
``Beweis des Satzes, dass jede unbegrenzte arithmetische Progression, deren erstes Glied und Differenz ganze Zahlen ohne gemeinschaftlichen Factor sind, unendlich viele Primzahlen enth\"alt.''
1837. Reprinted in R.~Dedekind (ed.),
{\em Vorlesungen \"uber Zahlentheorie}, pp.~342--359, 1879;
reissued by Cambridge University Press, Cambridge, 2013,
\DOI{10.1017/CBO9781139237321.012}

\bibitem{DUP1859}
Dupr\'e,~A.
{\em Examen d'une proposition de Legendre relative \`a la th\'eorie des nombres}.
Mallet-Bachelier, Paris, 1859.
\href{https://gallica.bnf.fr/ark:/12148/bpt6k939379s}%
     {\nolinkurl{gallica.bnf.fr/ark:/12148/bpt6k939379s}}

\bibitem{EP4DISC}
Erd\H{o}s Problems.
\href{https://www.erdosproblems.com/forum/thread/4/proof-claims}%
   {``Proof claims.''}
Discussion thread for Erd\H{o}s Problem~\#4, maintained by T.~F.~Bloom,
accessed 13~September~2026; archived at
\href{https://web.archive.org/web/20260914015658/https://www.erdosproblems.com/forum/thread/4/proof-claims}%
   {\nolinkurl{web.archive.org}}.

\bibitem{ERD1935}
Erd\H{o}s,~P.
``On the difference of consecutive primes.''
{\em Quart.~J.~Math.~Oxford Ser.}\ os-6(1):124--128, 1935.
\DOI{10.1093/qmath/os-6.1.124}

\bibitem{ERD1986}
Erd\H{o}s,~P.
``Some problems on number theory.''
In {\em Proceedings of the Seventeenth Southeastern International
Conference on Combinatorics, Graph Theory, and Computing}
(Boca Raton, FL, 1986),
{\em Congr.~Numer.}\ 54:225--244, 1986.
\href{https://users.renyi.hu/~p_erdos/1986-16.pdf}%
   {\texttt{users.renyi.hu/\textasciitilde p\_erdos/1986-16.pdf}}

\bibitem{FGKT2016}
Ford,~K., B.~Green, S.~Konyagin and T.~Tao.
``Large gaps between consecutive prime numbers.''
{\em Ann.~of Math.~\textup{(}2\textup{)}} 183(3):935--974, 2016.
\DOI{10.4007/annals.2016.183.3.4}

\bibitem{FGKMT2018}
Ford,~K., B.~Green, S.~Konyagin, J.~Maynard and T.~Tao.
``Long gaps between primes.''
{\em J.~Amer.~Math.~Soc.}\ 31(1):65--105, 2018.
\DOI{10.1090/jams/876}

\bibitem{GPT2026}
GPT-5.6~Sol.
\href{https://github.com/DottedCalculator/ai-math/blob/main/Erdos_4_GPT_5.6_Sol.pdf}%
   {``A tilted residue-class construction for long prime-free
   intervals.''}
Preprint, 48~pp., 25~August~2026; archived at
\href{https://web.archive.org/web/20260914014847/https://github.com/DottedCalculator/ai-math/blob/main/Erdos_4_GPT_5.6_Sol.pdf}%
   {\nolinkurl{web.archive.org}}.
Lean~4 formalization by B.~Alexeev at
\href{https://github.com/plby/lean-proofs/blob/main/src/latest/ErdosProblems/Erdos4Tilted.lean}%
   {\nolinkurl{github.com/plby/lean-proofs}}.

\bibitem{GRA1995}
Granville,~A.
``Harald Cram\'er and the distribution of prime numbers.''
{\em Scand.~Actuar.~J.}\ 1995(1):12--28, 1995.
\DOI{10.1080/03461238.1995.10413946}

\bibitem{HL1933}
Heilbronn,~H. and E.~Landau.
``Bemerkungen zur vorstehenden Arbeit von Herrn Bochner.''
{\em Math.~Z.}\ 37:10--16, 1933.
\DOI{10.1007/BF01474553}

\bibitem{LEG1830}
Legendre,~A.-M.
{\em Th\'eorie des nombres}.
3rd ed., vol.~2, Firmin Didot fr\`eres, Paris, 1830, pp.~71--79.
\DOI{10.3931/e-rara-59737}

\bibitem{MAI1985}
Maier,~H.
``Primes in short intervals.''
{\em Michigan Math.~J.}\ 32(2):221--225, 1985.
\DOI{10.1307/mmj/1029003189}

\bibitem{MP1990}
Maier,~H. and C.~Pomerance.
``Unusually large gaps between consecutive primes.''
{\em Trans.~Amer.~Math.~Soc.}\ 322(1):201--237, 1990.
\DOI{10.1090/S0002-9947-1990-0972703-X}

\bibitem{MS2007}
Maier,~H. and C.~L.~Stewart.
``On intervals with few prime numbers.''
{\em J.~Reine Angew.~Math.}\ 608:183--199, 2007.
\DOI{10.1515/CRELLE.2007.057}

\bibitem{MAY2016}
Maynard,~J.
``Large gaps between primes.''
{\em Ann.~of Math.~\textup{(}2\textup{)}} 183(3):915--933, 2016.
\DOI{10.4007/annals.2016.183.3.3}

\bibitem{MV2006}
Montgomery,~H.~L. and R.~C.~Vaughan.
{\em Multiplicative number theory I: Classical theory}.
Cambridge Studies in Advanced Mathematics, vol.~97,
Cambridge University Press, Cambridge, 2006.
\DOI{10.1017/CBO9780511618314}

\bibitem{NN2003}
Nyman,~B. and T.~R.~Nicely.
``New prime gaps between $10^{15}$ and $5\times10^{16}$.''
{\em J.~Integer Seq.}\ 6(3):Article 03.3.1, 6~pp., 2003.
\href{https://cs.uwaterloo.ca/journals/JIS/VOL6/Nicely/nicely2.html}%
     {\nolinkurl{cs.uwaterloo.ca/journals/JIS/VOL6/Nicely/nicely2.html}}

\bibitem{OAI2026}
OpenAI.
\href{https://cdn.openai.com/pdf/51126fac-1b68-4128-9666-c908bcc16033/long_gaps.pdf}%
   {``Improved long gaps between primes.''}
Preprint, 8~pp., 2026; archived at
\href{https://web.archive.org/web/20260908140915/https://cdn.openai.com/pdf/51126fac-1b68-4128-9666-c908bcc16033/long_gaps.pdf}%
   {\nolinkurl{web.archive.org}}.
Lean~4 formalization at
\href{https://github.com/openai/LongGapsBetweenPrimes}%
   {\nolinkurl{github.com/openai/LongGapsBetweenPrimes}}.
The paper attributes the proof to GPT-6~Astra.

\bibitem{PIN1997}
Pintz,~J.
``Very large gaps between consecutive primes.''
{\em J.~Number Theory}\ 63(2):286--301, 1997.
\DOI{10.1006/jnth.1997.2081}

\bibitem{PS1989}
Pippenger,~N. and J.~Spencer.
``Asymptotic behavior of the chromatic index for hypergraphs.''
{\em Journal of Combinatorial Theory, Series A} 51(1):24--42, 1989.
\DOI{10.1016/0097-3165(89)90074-5}

\bibitem{POM1989}
Pomerance,~C.
``Two methods in elementary analytic number theory.''
In {\em Number theory and applications}, R.~A.~Mollin (ed.), pp.~135--161.
NATO Adv.~Sci.~Inst.~Ser.~C: Math.~Phys.~Sci.~265.
Kluwer Academic Publishers, Dordrecht, 1989.

\bibitem{RAM2022}
Ramar\'e,~O.
``Rankin's trick and Brun's sieve.''
In {\em Excursions in multiplicative number theory}, pp.~251--258.
Birkh\"auser Advanced Texts Basler Lehrb\"ucher, Birkh\"auser, Cham, 2022.
\href{https://doi.org/10.1007/978-3-030-73169-4_25}%
     {doi:10.1007/978-3-030-73169-4\_25}

\bibitem{RAN1938}
Rankin,~R.~A.
``The difference between consecutive prime numbers.''
{\em J.~London Math.~Soc.}\ s1-13(4):242--247, 1938.
\DOI{10.1112/jlms/s1-13.4.242}

\bibitem{RAN1963}
Rankin,~R.~A.
``The difference between consecutive prime numbers~V.''
{\em Proc.~Edinburgh Math.~Soc.~\textup{(}2\textup{)}} 13(4):331--332, 1963.
\DOI{10.1017/S0013091500025633}

\bibitem{RIC1934}
Ricci,~G.
``Ricerche aritmetiche sui polinomi, II. (Intorno a una proposizione non vera di Legendre).''
{\em Rend.~Circ.~Mat.~Palermo}\ 58:190--208, 1934.
\DOI{10.1007/BF03019710}

\bibitem{SCH1963}
Sch\"onhage,~A.
``Eine Bemerkung zur Konstruktion grosser Primzahll\"ucken.''
{\em Arch.~Math.} 14:29--30, 1963.
\DOI{10.1007/BF01234916}

\bibitem{WES1931}
Westzynthius,~E.
``\"Uber die Verteilung der Zahlen, die zu den $n$ ersten Primzahlen teilerfremd sind.''
{\em Comment.~Phys.-Math.}\ 5(25):1--37, 1931.
\href{https://zbmath.org/0003.24601}{Zbl:0003.24601}

\bibitem{ZEI1930}
Zeitz,~H.
{\em Elementare Betrachtung \"uber eine zahlentheoretische Behauptung
von Legendre}.
Privatdruck, Berlin, 1930.

\end{thebibliography}
\end{document}